\documentclass{birkjour}

\usepackage{amssymb}

\usepackage{graphicx}
\usepackage{latexsym}
\usepackage{amsmath}
\usepackage{amssymb}
\usepackage[colorlinks]{hyperref}
\usepackage{paralist}

 \newtheorem{thm}{Theorem}[section]
 
 \newtheorem{lem}[thm]{Lemma}
 
 \newtheorem{assumption}[thm]{Assumption}
 \newtheorem{defn}[thm]{Definition}
 \newtheorem{rem}[thm]{Remark}
 
 \numberwithin{equation}{section}

\begin{document}

%
%
%
%
%
%
%
%
%
\title[Navier-Stokes and Darcy-Forchheimer-Brinkman systems on manifolds]{Transmission problems for the Navier-Stokes and Darcy-Forchheimer-Brinkman systems in Lipschitz domains on compact Riemannian manifolds}
\author[M. Kohr]{Mirela Kohr}
\address{%
Faculty of Mathematics and Computer Science, Babe\c{s}-Bolyai University,\\
1 M. Kog\u alniceanu Str.,
400084 Cluj-Napoca, Romania}

\email{mkohr@math.ubbcluj.ro}

\thanks{M. Kohr acknowledges the support of the grant PN-II-ID-PCE-2011-3-0994 of the Romanian National Authority for Scientific Research, CNCS-UEFISCDI. The research has been also partially supported by the grant EP/M013545/1: "Mathematical Analysis of Boundary-Domain Integral Equations for Nonlinear PDEs" from the EPSRC, UK}

\author[S. E. Mikhailov]{Sergey E. Mikhailov}
\address{Department of Mathematics, Brunel University London,\\
             Uxbridge, UB8 3PH, United Kingdom}

\email{sergey.mikhailov@brunel.ac.uk}


\author[W. L. Wendland]{Wolfgang L. Wendland}
\address{Institut f\"ur Angewandte Analysis und Numerische Simulation, Universit\"at Stuttgart,\\
Pfaffenwaldring, 57,
70569 Stuttgart, Germany}

\email{wendland@mathematik.uni-stuttgart.de}

\thanks{}



\dedicatory{{\em Published in:} J. Math. Fluid Mech., {\em 2016, DOI 10.1007/s00021-016-0273-6}}

\begin{abstract}
The purpose of this paper is to study boundary value problems of transmission type for the Navier-Stokes and Darcy-Forchheimer-Brinkman systems in two complementary Lipschitz domains on a compact Riemannian manifold of dimension $m\in \{2,3\}$. We exploit a layer potential method combined with a fixed point theorem in order to show existence and uniqueness results when the given data are suitably small in $L^2$-based Sobolev spaces.
\end{abstract}

\maketitle

\section{Introduction}
\label{sec:1}

Let $M$ be {an infinitely} smooth, compact, boundaryless Riemannian manifold of dimension $m\geq 2$. Let $\alpha >0$ be a given constant. Then the following linear equations
\begin{equation}
\label{matrix-pseudo-diff1}
{\mathcal L}_{\alpha }({\bf u},\pi ):=({\bf L}+\alpha {\mathbb I}){\bf u}+d\pi ={\bf f}, \ \delta {\bf u}=0
\end{equation}
determine the {\it Brinkman system}. Note that $${\bf L}:=2{\rm{Def}}^{*}{\rm{Def}}= -\triangle +d\delta -2\rm{Ric}$$ is the natural operator that appears in the structure of the Stokes system on $M$ (cf. \cite[pp. 161, 162]{Ebin-Marsd}; see also \cite{D-M-M,M-T}), ${\rm{Def }}$ is the deformation operator, $\triangle :=-(d\delta +\delta d)$ is the Hodge Laplacian, $d$ is the exterior derivative operator, $\delta $ is the exterior co-derivative operator, and ${\rm{Ric}}$ is the Ricci tensor of $M$ (see, e.g., \cite[Section 1]{D-M-M}, \cite[Chapters I-V]{Lang}, \cite{M-T}, \cite[Appendix B]{Taylor}). In the next section we give more details on these operators. For $\alpha =0$, {system} {\eqref{matrix-pseudo-diff1} reduces to the {incompressible} {\it Stokes system},
\begin{equation}
\label{Stokes}
\begin{array}{lll}
{\mathcal L}_{0}({\bf u},\pi ):={\bf L}{\bf u}+d\pi ={\bf f}, \ \delta {\bf u}=0.
\end{array}
\end{equation}
The nonlinear system
\begin{equation}
\label{D-B-F-1}
{\bf L}{\bf v}+\alpha {\bf v}+k|{\bf v}|{\bf v}+\beta \nabla _{\bf v}{\bf v}+dp={\bf f},\
\delta {\bf u}=0
\end{equation}
is called the {incompressible} {\it Darcy-Forchheimer-Brinkman system}. Note that $\nabla $ is the Levi-Civita connection, and $\nabla _{\bf u}{\bf u}$ is the covariant derivative of ${\bf u}$ with respect to ${\bf u}$. In Euclidean setting such a system describes flows in porous media saturated with viscous incompressible fluids when the inertia of such a fluid can not be neglected. The constants $\alpha ,k,\beta >0$ are determined by the physical properties of such a porous medium (for further details we refer the reader to the book by Nield and Bejan \cite[p. 17]{Ni-Be} and the references therein).
When {$\alpha=k=0$}, \eqref{D-B-F-1} reduces to the incompressible {\it Navier-Stokes system}
\begin{equation}
\label{N-S}
\begin{array}{lll}
{\bf L}{\bf v}+\beta \nabla _{\bf v}{\bf v}+dp={\bf f},\
\delta {\bf u}=0,
\end{array}
\end{equation}
that plays a main role in fluid mechanics.

{Note that the operators involved in the PDE systems \eqref{matrix-pseudo-diff1}-\eqref{N-S} are variable-coefficient operators due to their dependence on the differential and metric structures of the manifold.}

{{The considered} partial differential equations and related boundary value problems on {manifolds} have various practical applications. A suggestive example is given by the Navier-Stokes equations which reduce to the equations of thin layers when the thickness of such a {layer} tends to zero (cf., e.g., \cite{D-M-M}). For further relevant applications of the Navier-Stokes equations on manifolds we refer the reader to \cite{D-M-M, Ebin-Marsd, T-Z, T-W} and the references therein (see also \cite{Ci}). On the other hand, the analysis of transmission problems for the Navier-Stokes and Darcy-Forchheimer-Brinkman systems on compact surfaces (e.g., on the sphere $S^2$) is well motivated by the geophysical model of flow of water or other viscous fluids, which pass through porous rocks or porous soil (see also \cite{Ivancevic}).}

Layer potential methods have been often used in the mathematical analysis of elliptic boundary value problems on Lipschitz domains in ${\mathbb R}^n$, $n\geq 2$. Fabes, Kenig and Verchota in \cite{Fa-Ke-Ve} have studied the $L^2$-Dirichlet problem for the Stokes system in Lipschitz domains in Euclidean setting, by reducing such a problem to the analysis of related boundary integral operators. By using the technique of boundary integral equations, Mitrea and Wright in \cite{M-W} have obtained well-posedness results for the main boundary value problems of Dirichlet, Neumann, and transmission type for the Stokes system in Lipschitz domains in ${\mathbb R}^n$, $n\geq 2$, with data in $L^p$, Sobolev and Besov spaces. Medkov\'{a} in \cite{Med-CVEE} obtained well-posedness results for $L^2$-solutions of boundary problems of transmission type for the Brinkman system in Lipschitz domains in ${\mathbb R}^n$, $n\geq 3$. Other applications of layer potential theory for elliptic boundary value problems can be found in \cite{Fa-Me-Mi}, \cite{Gi-Tr} and \cite{12}.

{Another branch of integral methods for PDEs is related with parametrix (Levi function). Parametrix-based direct segregated systems of boundary-domain integral equations (BDIEs) for mixed boundary value problems of Dirichlet-Neumann type corresponding to a scalar second-order divergent elliptic partial differential equation with a variable coefficient in interior and exterior domains in ${\mathbb R}^3$ were analysed in \cite{CMN-1} and \cite{Ch-Mi-Na-3}, respectively. In domains with interior cuts (cracks) such systems have been studied in \cite{Ch-Mi-Na-2}. In \cite{Ch-Mi-Na-1}, the variable-coefficient transmission problems with interface crack for second order elliptic partial differential equations in a bounded composite domain consisting of adjacent anisotropic subdomains separated by an interface, were reduced to the localized direct segregated boundary-domain integral equations.
In \cite{Ch-Mi-Na} and \cite{CMN-ArXiv2015}, the so-called two-operator technique was used for reduction of more general scalar equations and systems of PDEs to localized boundary-domain integral equations.
Equivalence of the BDIEs to corresponding boundary value problems and the invertibility of the BDIE operators in the $L^2$-based Sobolev spaces have been  analysed in all these papers.
In \cite{Mikh-1,Mikh-2} some nonlinear boundary value problems  were reduced to direct localized boundary-domain integro-differential formulations.

Mitrea, Mitrea and Shi in \cite{M-M-Q} used a boundary integral method to study variable coefficient transmission problems on non-smooth compact manifolds.
Mitrea and Taylor in \cite{M-T} have used a boundary integral method in the analysis of the $L^2$-Dirichlet problem for the Stokes system on arbitrary Lipschitz domains on a compact Riemannian manifold, extending the results {of} \cite{Fa-Ke-Ve} from the Euclidean setting to compact Riemannian manifolds (see also \cite{H-M-T, M-T1, M-T2}). By using a layer potential analysis, Dindo\u{s} and Mitrea \cite{D-M} have obtained the well-posedness of the Poisson problem for the Stokes system on $C^1$, or, more generally, Lipschitz domains in a compact Riemannian manifold, with data in Sobolev and Besov spaces.
The authors in \cite{10-new2} have extended the notion of the Brinkman differential operator from the Euclidean setting to compact Riemannian manifolds, and defined pseudodifferential Brinkman operators as operators with variable coefficients on compact Riemannian manifolds. They have investigated the well-posedness of related transmission problems in $L^2$-based Sobolev spaces on Lipschitz domains in compact Riemannian manifolds of dimension $m\geq 2$. The authors in \cite{JDDE} have obtained well-posedness results in $L^2$-based Sobolev spaces for Poisson-transmission problems expressed in terms of $L^{\infty }$-perturbations of the Stokes system on complementary Lipschitz domains in a compact Riemannian manifold of arbitrary dimension $n\geq 2$, with a parameter $\mu \in (0,1)$ involved in one of the transmission conditions.

Choe and Kim \cite{Choe-Kim} obtained well-posedness results for the Dirichlet problem for the Navier-Stokes system on a bounded Lipschitz domain in ${\mathbb R}^3$ with connected boundary. Russo and Tartaglione \cite{Russo-Tartaglione-2} used a double-layer potential formulation in the analysis of the Robin problem for the Stokes and Navier-Stokes systems in bounded or exterior Lipschitz domains in ${\mathbb R}^3$ (see also \cite{Russo-Starita,Russo-Tartaglione}).
The authors in \cite{K-L-W1-new} have combined a layer potential analysis with the Schauder fixed point theorem to show the existence of solutions for a Poisson problem of a semilinear Brinkman system on a bounded Lipschitz domain in ${\mathbb R}^n$ ($n\geq 2$) with Dirichlet or Robin boundary condition and given data in Sobolev and Besov spaces. Existence results for boundary value problems of Robin type for the Brinkman system and the nonlinear Darcy-Forchheimer-Brinkman system in bounded Lipschitz domains in Euclidean setting have been obtained in \cite{K-L-W1} (see also \cite{K-L-W, K-L-W3, K-L-W4, Mirela-Sergey}). Recently, the authors in \cite{K-L-M-W} have obtained existence and uniqueness results in weighted Sobolev spaces for transmission problems for the non-linear Darcy-Forchheimer-Brinkman system and the linear Stokes system in two complementary Lipschitz domains in ${\mathbb R}^3$, by exploiting a layer potential method for the Stokes and Brinkman systems combined with a fixed point theorem.

Using variational arguments,  Amrouche and Nguyen in \cite{Amrouche-1} obtained existence and uniqueness results in weighted Sobolev spaces for the Poisson problem of Dirichlet type associated with the Navier-Stokes system in exterior Lipschitz domains in ${\mathbb R}^3$. Amrouche and Rodr\'{i}guez-Bellido in \cite{A-R} studied boundary problems for Stokes, Oseen and Navier-Stokes systems with singular data (see also \cite{Galdi}).

Dindo\u{s} and Mitrea in \cite{D-M-new} have combined a fixed point theorem with well-posedness results from the linear theory for the Poisson problem corresponding to the Laplace operator in Sobolev and Besov spaces on Lipschitz domains in compact Riemannian manifolds, and developed a sharp theory for semilinear Poisson problems of Dirichlet and Neumann type for $L^{\infty }$-perturbations of the Laplace-Beltrami operator on Lipschitz domains in compact Riemannian manifolds.
{In \cite{D-M} they} also obtained existence and uniqueness results for the Poisson problem of Dirichlet type for the Navier-Stokes system on a Lipschitz domain in a compact Riemannian manifold with data in Sobolev and Besov spaces.
}

The purpose of this paper is to study boundary value problems of transmission type for the Navier-Stokes and Darcy-Forchheimer-Brinkman systems in two complementary Lipschitz domains of a compact Riemannian manifold of dimension $m\in \{2,3\}$. One of the transmission conditions is expressed in terms of a parameter $\mu \in (0,\infty )$, and another one in terms of an $L^{\infty}$-symmetric tensor field ${\mathcal P}$, which satisfies a positivity condition. Due to the choice of the range of the parameter $\mu $ and of conditions satisfied by ${\mathcal P}$, our results in the linear case extend those in \cite[Theorems 4.1 and 5.1]{10-new2}. First, we use a boundary integral method for the Stokes system
to show the well-posedness result in $L^2$-based Sobolev spaces {(with a range of smoothness index)} for the Poisson problem of transmission type associated with the Stokes system in such Lipschitz domains (see Theorem \ref{Poisson-transmission-S0}). Next, we use this well-posedness result and the invertibility of a related operator in order to show the well-posedness result in $L^2$-based Sobolev spaces for the Poisson problem of transmission type associated with the Stokes and Brinkman systems in two complementary Lipschitz domains of a compact Riemannian manifold of dimension $m\geq 2$ (see Theorem   \ref{Poisson-transmission-Stokes}). Then, based on the well-posedness result in the linear case combined with a fixed point theorem, we obtain an existence and uniqueness result for the transmission problem corresponding to the nonlinear Navier-Stokes and Darcy-Forchheimer-Brinkman systems, when the given data are suitably small in $L^2$-based Sobolev spaces. Due to technical details which require the continuity of some Sobolev embeddings we need to restrict our analysis to compact Riemannian manifolds of dimension $m\in \{2,3\}$ (see Theorem \ref{Poisson-transmission-NS-DFB-0}). The existence result for the nonlinear problems developed in this paper extends some results obtained in \cite{K-L-M-W} in the context of Euclidean setting to the case of compact Riemannian manifolds.

\section{Preliminaries}
\label{sec:2}
In this section we present the differential operators in systems \eqref{matrix-pseudo-diff1}-\eqref{N-S} on a Lipschitz domain in a compact Riemannian manifold.

Let $(M,\langle \cdot,\cdot \rangle )$ be an infinitely smooth, compact, boundaryless manifold\footnote{ Due to the Whitney--Nash theorem (see, e.g., \cite[Theorem 10.15]{Lee} and \cite[Theorem 2]{Nash} the compact Riemannian manifold M, $\dim (M)=m$, admits a smooth isometrical embedding into ${\mathbb R}^n$ with $2n=m(3m+11)$ as a closed submanifold.} of dimension $m\geq 2$, equipped with a smooth Riemannian metric tensor {$g=g_{jk}dx^j\otimes dx^k$}. Here and all along the paper, we use the repeated index summation convention. We denote by $(g^{jk})$ the inverse of $(g_{jk})$, i.e., $g^{j\ell }g_{\ell k}=\delta _{jk}$. Also, denote $g:={\rm{det}}(g_{jk})$, {and let $g>0$}. As usual, $T(M)\!=\!\bigcup _{p\in M}{T}_p(M)$ denotes the tangent bundle of $M$, where $T_pM$ is the tangent space at $p\in M$, and $T^*(M)\!=\!\bigcup _{p\in M}{T}_p^*(M)$ is the cotangent bundle. Let\footnote{{A {\it vector field} $X$ on $M$ is a map $X:M\to TM$, such that $\pi \circ X$ is the identity, where $\pi $ is the projection from $TM$ to $M$, and $X$ is {\it smooth} if $X:M\to TM$ is smooth}.} $\mathfrak{X}(M)$ be the space of {infinitely} smooth vector fields on $M$ and let $\Lambda ^1TM$ be the first exterior power bundle corresponding to $TM$, i.e., the space of differential one forms on $M$. If {${\bf x}\!=\!(x^1,\cdots ,x^m)$ is a local coordinate system in one chart $(U,\varphi )$ on $M$ around $p\in U$, with $\varphi (p)\!=\!0$, then the maps
\begin{align}
\partial _i:C^{\infty }(U)\to {\mathbb R},\ f\mapsto \frac{\partial f\circ \varphi ^{-1}}{\partial x^i}(0),\ 1=1,\ldots ,m\nonumber
\end{align}
are tangent vectors at $p$. They are linearly independent and determine a basis of $T_pM$.}
Then any smooth vector field ${\bf v}$ can be written as ${\bf v}=v^k{\partial }_k$, where the components $v^k$ of ${\bf v}$ are smooth functions on the domain of ${\bf x}$. In addition,
\begin{equation}
\label{inner-prduct}
\langle \cdot ,\cdot\rangle :{\mathfrak X}(U)\times {\mathfrak X}(U)\to C^{\infty }(U),\ \langle X,Y\rangle =X^jg_{jk}Y^k,\ \forall \ X,Y\in {\mathfrak X}(U),
\end{equation}
is a symmetric positive definite tensor field of type $(0,2)$.

{The gradient and divergence operators are defined, locally, on $M$, as\footnote{When the metric tensor is the identity one, we recover the usual definition of the gradient and divergence operators in the Euclidean setting, {and $\partial _jf\equiv \dfrac{\partial f}{\partial x^j}$}.}} (cf. e.g., \cite[(3.11), (3.12)]{Taylor})}
\begin{align}
\label{gradient-operator}
&{\rm{grad}}\,f:=(g^{jk}\partial _jf)\partial _k,\ \forall \ f\in C^{\infty }(M),
\end{align}
\begin{align}
&{\rm{div}}\, X:=\frac{1}{\sqrt{g}}\partial _k(\sqrt{g}X^k),\ \ \forall \ X=X^j\partial _j\in TM.
\end{align}
For each $p\in M$, the cotangent space ${T}_p^*(M)$ can be naturally identified, via the scalar product $\langle\cdot,\cdot\rangle_{_p}$, with the tangent space ${T}_p(M)$, and hence the cotangent bundle $T^*M$ can be identified with the tangent bundle $T(M)$, since the Riemannian metric on $T(M)$ transfers also to $T^*M$ . In addition, the space of differential one forms $\Lambda ^1TM$ can be identified with the space of smooth vector fields $\mathfrak{X}(M)$ via the isometry {$\partial _j\longmapsto g_{j\ell }dx^{\ell }$}, whose inverse is {$dx^j=g^{j\ell }\partial _\ell $}. Hence, a vector field $X=X^k\partial _k\in {\mathfrak X}(M)$ can be identified with the one-form {$X_jdx^j=X^kg_{kj}dx^j$}, where $X_j=g_{jk}X^k$, $X^k=g^{k\ell }X_\ell$. According to \eqref{inner-prduct} we have the pointwise inner product of forms
\begin{equation}
\label{inner-prduct-1}
{\langle dx^j,dx^k\rangle =g^{jk}},\ \ \langle X,Y\rangle =X_jg^{jk}Y_k=X^kg_{k\ell }Y^\ell =X^kY_k.
\end{equation}
With respect to the above inner product, the operators ${\rm{grad}}$ and $-{\rm{div}}$ are adjoint to each other. Consequently, the notation $\langle \cdot ,\cdot \rangle $ is used also for the pointwise inner product of differential one forms. The same notation, $\langle \cdot ,\cdot \rangle $, will be also used for the pairing between two dual spaces. Nevertheless, its meaning will be understood from the context. Let us also note that the volume element in $M$, $d{\rm{Vol}}$, is given in local coordinates by { $d{\rm{Vol}}=g^{\frac{1}{2}}dx^1\cdots dx^m$}.

Identification between the spaces $T^*M$ and $TM$ and $\Lambda ^1TM$ with $\mathfrak{X}(M)$ leads to the identification of the gradient operator ${\rm{grad }}:C^{\infty }(M)\to {\mathfrak X}(M)$ with the exterior derivative operator
\begin{equation}
\label{d}
d:C^{\infty }(M)\to C^{\infty }(M, \Lambda ^1TM),\ \ {d=\partial _jdx^j},
\end{equation}
and to the identification of the differential operator $-{\rm{div }}:{\mathfrak X}(M)\to C^{\infty }(M)$ with the exterior co-derivative operator
\begin{equation}
\label{delta}
\delta :C^{\infty }(M, \Lambda ^1TM)\to C^{\infty }(M),\ \ \delta =d^*.
\end{equation}
Further details about differential geometry on manifolds can be consulted in, e.g., \cite[Section 1]{D-M-M}, \cite[Chapters I-V]{Lang}, \cite{M-T}, \cite[Appendix B]{Taylor}, \cite[Part II]{Wloka}.

\subsection{Sobolev spaces and related results}

{Let ${\mathcal D}(M)$ be the space of infinitely differentiable functions in $M$ and ${\mathcal D}'(M)$ be the corresponding space of distributions on $M$, i.e., the dual of ${\mathcal D}(M)$. The spaces ${\mathcal D}'(M,\Lambda ^1TM)$ and ${\mathcal D}'(M,S^2T^*M)$ can be similarly defined.}
The Lebesgue space of (equivalence classes of) measurable, square integrable functions on $M$ is denoted by $L^2(M)$, and $L^{\infty }(M)$ is the space of (equivalence
classes of) essentially bounded measurable functions on $M$. Let $s\in {\mathbb R}$. Then the $L^2$-based Sobolev (Bessel potential) space $H^s({\mathbb R}^m)$ is defined by
$$H^s({\mathbb R}^m):=\left\{({\mathbb I}-\triangle )^{-\frac{s}{2}}f:f\in L^2({\mathbb R}^m)\right\}.$$ Let $\{(U_i,\varphi _i)\}_i$ be a collection of charts such that $\{U_i\}_i$ determines a finite cover of $M$. Then for any distribution $\Phi \in C_0^{\infty }(U_i)^*$, the pull-back $\varphi _i^*(\Phi )\in C_0^{\infty }(\varphi _i(U_i))^*$ is defined by $\varphi _i^*(\Phi )(f):=\Phi (f\circ \varphi _i)$ for all $f\in C_0^{\infty }(\varphi _i(U_i))$. Extending $\varphi _i^*(\Phi )$ by zero outside $\varphi _i(U_i)$, $\varphi _i^*(\Phi )$ can be treated as a distribution on ${\mathbb R}^m$. Let $\{\chi _i\}_i$ be a smooth partition of unity subordinate to $\{U_i\}_i$. Then for $s\in {\mathbb R}$, $H^s(M)$ is the space of all distributions $\Phi $ defined in $M$ such that
\begin{equation}
\label{Sobolev-M}
\|\Phi \|_{H^s(M)}=\sum _i\|\varphi ^*(\chi _i\Phi )\|_{H^s({\mathbb R}^m)}
\end{equation}
(see, e.g., \cite[Definition 3]{Holst}).
Note that the space $H^{-s}(M)$ is the dual of $H^s(M)$, and $H^s(M,\Lambda^1TM):=H^s(M)\otimes \Lambda^1TM$ is the Sobolev space of one forms on $M$ (see, e.g., \cite[Chapter 4, Section 3]{Taylor}, \cite[Chapter 8]{Wloka}).

Let $\Omega _{+}:=\Omega \subset M$ denote a Lipschitz domain, i.e., an open connected set whose boundary is locally the graph of a Lipschitz function.
{Throughout the paper we adopt the following
\begin{assumption}
\label{connected-set}
$\Omega _{-}:=M\setminus \overline{\Omega }$ is {\it a connected and non-empty set.}
\end{assumption}}
We also consider the $L^2$-based (Bessel potential) Sobolev spaces on $\Omega $
\begin{align}
\label{space-1}
H^s(\Omega ):=\{f|_{\Omega }: f\in H^s(M)\},\, \widetilde{H}^s(\Omega ):=\{f\in H^s(M): {\rm{supp }}f\subseteq \overline{\Omega }\},
\end{align}
and note that for any $s\in {\mathbb R}$ (see also \cite[Chapter 4, Section 3]{Taylor}, \cite[(4.14)]{M-T2})
\begin{equation}
\label{duality-spaces}
\left(H^s(\Omega )\right)'=\widetilde{H}^{-s}(\Omega ),\
{H}^{-s}(\Omega )=\big(\widetilde{H}^{s}(\Omega )\big)'.
\end{equation}
The $L^2$-based Sobolev spaces of one forms on $\Omega $ are given by
\begin{align}
\label{space-2}
\!\!\!\!\!{H}^{s}(\Omega ,\Lambda ^1TM)\!:=\!{H}^{s}(\Omega )\!\otimes \!\Lambda ^1TM|_{\Omega },\, \widetilde{H}^{s}(\Omega ,\Lambda ^1TM)\!:=\!\widetilde{H}^{s}(\Omega )\!\otimes \!\Lambda ^1TM.
\end{align}

Now let $s\in [0,1)$. Then the boundary Sobolev space ${H}^s(\partial \Omega )$ can be defined by using the space ${H}^s({\mathbb R}^{n-1})$, a partition of unity and pull-back. In addition, we have ${H}^{-s}(\partial \Omega )=\left({H}^s(\partial \Omega )\right)'$, and ${H}^{s}(\partial \Omega ,\Lambda ^1TM):={H}^{s}(\partial \Omega )\otimes \Lambda ^1TM|_{\partial \Omega }$ is the $L^2$-based boundary Sobolev space of one forms.

For further details related to Sobolev spaces on compact manifolds we refer the reader to \cite[Chapter 4]{Taylor}, \cite[Chapter 8]{Wloka}, \cite{Agr-2}[Chapter 2], \cite{D-M-new, M-T2, Triebel}.

A useful result for the problems we are going to investigate is the following trace lemma (see \cite{Co}, \cite[Proposition 3.3]{J-K1}, \cite[Theorem 2.3, Lemma 2.6]{Mikh}, \cite{Mikh-3}, \cite[Theorem 2.5.2]{M-W}):
\begin{lem}
\label{trace-operator1}
Assume that $\Omega \subset M$ is a bounded Lipschitz domain with boundary $\partial \Omega $, and set $\Omega _{+}:=\Omega $ and $\Omega _{-}:=M\setminus \overline{\Omega }$. Let $s\in (0,1)$. Then there exist linear and continuous trace operators\footnote{The trace operators defined on Sobolev spaces of one forms on the domains ${\Omega }_{\pm }$ are also denoted by $\gamma _{\pm }$.} $\gamma _{\pm }:{H}^{s+\frac{1}{2}}(\Omega _{\pm })\to H^s({\partial\Omega })$ such that $\gamma _{\pm }f=f|_{\partial \Omega }$ for any $f\in C^{\infty }(\overline{\Omega }_\pm )$. These operators are surjective and have $($non-unique$)$ linear and continuous right inverse operators $\gamma _{\pm }^{-1}:H^s({\partial\Omega })\to {H}^{s+\frac{1}{2}}(\Omega _\pm ).$
\end{lem}


\subsection{\bf The deformation operator}

{Let $\nabla $ denote the Levi-Civita connection associated to the Riemannian metric $g$ of the Riemannian manifold $M$. Let us briefly recall its definition by following, e.g., \cite{Lang, Taylor, Wloka}.}

{An {\it affine connection} on the manifold $M$ is a map
\begin{align}
\nabla :{\mathfrak X}(M)\times {\mathfrak X}(M)\to {\mathfrak X}(M),\ \ (X,Y)\mapsto \nabla _XY\nonumber
\end{align}
satisfying the following conditions\footnote{We use the notation $C^{\infty }(M;R):=C^{\infty }(M)$, and note that $df(X)Y=(Xf)Y$.}:
\begin{itemize}
\item[(ii)]
For each fixed $Y\in {\mathfrak X}(M)$ the map $X\mapsto \nabla _XY$ is $C^{\infty }$-linear, i.e., $\nabla _{fX+hZ}Y=f\nabla _XY+h\nabla _ZY$, for all $f,h\in C^{\infty }(M)$ and $X,Z\in {\mathfrak X}(M)$.
\item[(ii)]
For each $X\in {\mathfrak X}(M)$ the map $Y\mapsto \nabla _XY$ is ${\mathbb R}$-linear, i.e.,
\begin{align}
\label{conormal-derivative-vectors1}
\nabla _X(aY+bZ)=a\nabla _XY+b\nabla _XZ,\ \forall \ a,b\in {\mathbb R},\ \forall\ Y,Z\in {\mathfrak X}(M).
\end{align}
\item[(iii)]
For all $f\in C^{\infty }(M)$,
\begin{align}
\label{conormal-derivative-vectors1n}
&\nabla _X(fY)=df(X)Y+f\nabla _XY.
\end{align}
\end{itemize}
Relation \eqref{conormal-derivative-vectors1n} shows that $\nabla _X$, the covariant derivative along $X$, satisfies the Leibniz product rule, i.e., it acts as a derivation.

An affine connection $\nabla $ on a Riemannian manifold $(M,g)$ is called {\it compatible with $g$} if $\nabla g=0$, or, equivalently, the following compatibility relation with respect to the metric on $M$ holds
\begin{equation}
\label{conormal-derivative-vectors2}
Z\langle X,Y\rangle =\nabla _Z\langle X,Y\rangle =\langle \nabla _ZX,Y\rangle +\langle X,\nabla _ZY\rangle ,\ \forall \ X,Y,Z\in {\mathfrak X}(M),
\end{equation}
where $\langle \cdot ,\cdot \rangle $ is the inner product on tangent vectors, given by \eqref{inner-prduct} (cf., e.g., \cite[Chapter 1 \S 11, Chapter 2 \S 2]{Taylor}).}

{An affine connection $\nabla $ is called {\it torsion-free} if it satisfies the relation\footnote{{If $X,Y\in {\mathfrak X}(M)$, then the {\it Lie bracket} $[X,Y]$ is the smooth vector field given by $$[X,Y](f)=X(Y(f))-Y(X(f)),\ \forall \ f\in C^{\infty }(M)$$
(cf., e.g., \cite[Proposition 1.3]{Lang}).}}
\begin{equation}
\label{conormal-derivative-vectors3}
\nabla _XY - \nabla _YX =[X,Y],\ \forall \ X,Y\in {\mathfrak X}(M).
\end{equation}
An affine connection $\nabla $ on $(M,g)$, which is compatible with the Riemannian metric $g$ and is torsion-free, i.e., it satisfies conditions \eqref{conormal-derivative-vectors1}-\eqref{conormal-derivative-vectors2}, is called a {\it Levi-Civita connection} of $M$.}

{The fundamental theorem of Riemannian geometry asserts that given a Riemannian manifold $(M,g)$, there exists a unique Levi-Civita connection $\nabla $, which is determined uniquely by the torsion-free condition \eqref{conormal-derivative-vectors3}
(cf., e.g., \cite[Proposition 11. 1, Chapter 1 \S11]{Taylor}).}

Let $X$ and $Y$ be vector fields on $M$, with components $X^i$ and $Y^k$. Then the {\it covariant derivative} $\nabla _XY$ of $Y$ with respect to $X$ is given by
\begin{align}
\label{Christoffel3}
\nabla _YX=Y^i\big({\partial _i X^k}+X^j\Gamma _{ij}^k\big)\partial _k,\,
\end{align}
where {$\Gamma _{jk}^{\ell }$ are the Christoffel symbols of the second kind, given by
\begin{align}
\label{Christoffel-second-kind}
\Gamma _{\ell j}^k=2^{-1}g^{kr}\left\{\partial _jg_{r\ell }+\partial _\ell g_{rj}-\partial _r g_{\ell j}\right\},
\end{align}
and $(g^{kr})$ is the inverse of the matrix $(g_{kr})$ (see, e.g., \cite[Chapter 3, (3.9)]{Taylor}).}
{Moreover, if $X,Y,Z\in {\mathfrak X}(M)$, then
\begin{align}
\langle \nabla _ZX,Y\rangle =Z^i\big({\partial _iX^k}+X^j\Gamma _{ij}^k\big)g_{k\ell }Y^{\ell }.
\end{align}}
If $X\in {\mathfrak X}(M)$, then $\nabla X:{\mathfrak X}(M)\times {\mathfrak X}(M)\to C^{\infty }(M)$ is the tensor field of type $(0,2)$ given by
\begin{equation}
\label{nabla1}
(\nabla X)(Y,Z):=\langle \nabla _ZX,Y\rangle ,\ \ \forall \ Y,Z\in {\mathfrak X}(M).
\end{equation}

{In local coordinates and with the notation $X_k=g_{kj}X^j$, $\nabla X$ is given by the matrix of type $m\times m$ with the components\footnote{In many applications, a Riemannian manifold $\widetilde{M}$ of dimension $m$ is embedded into an Euclidean space $({\mathbb R}^n,\langle \cdot ,\cdot \rangle _{{\mathbb R}^n})$ ($m<n$) through a second order continuously differentiable mapping $\boldsymbol{\Psi }:V\subset {\mathbb R}^m\to {\mathbb R}^n$, such that the vectors $\dfrac{\partial \boldsymbol{\Psi }}{\partial x^j}(p)$, $j=1,\ldots ,m$, determine a basis of the tangent space at the point $\boldsymbol{\Psi }(p)\in \widetilde{M}$, and that the scalar product on ${\mathbb R}^n$ is compatible with a metric $\widetilde{g}$ on $\widetilde{M}$, given by $\widetilde{g}_{jk}=\left\langle \dfrac{\partial \boldsymbol{\Psi }}{\partial x^j}, \dfrac{\partial \boldsymbol{\Psi }}{\partial x^k}\right\rangle _{{\mathbb R}^n}$. In this case, the Christoffel symbols corresponding to $\widetilde{g}$ are given by the formula $\widetilde{\Gamma }_{i j}^\ell =\widetilde{g}^{\ell r}\left\langle \dfrac{\partial ^2 \boldsymbol{\Psi }}{\partial x^i\partial x^j},\dfrac{\partial \boldsymbol{\Psi }}{\partial x^r}\right\rangle_{{\mathbb R}^n}$ (see, e.g., \cite[Theorem 3.8.1 (**)]{Kl}, \cite[4.7 Definition (ii)]{Ku}).} $\left({\partial _i}X_\ell -X_k \Gamma_{i\ell}^k\right)_{i,\ell =1,\ldots ,m}$ (see also \cite[Chapter 2, (3.13)]{Taylor}, \cite[(1.10)]{D-M-M}).}

The symmetric part of $\nabla X$ is
denoted by ${\rm{Def }}\, X$ and is called the {\it deformation} of $X$. Consequently,
\begin{equation}
\label{deformation-tensor-new}
({\rm{Def }}\, X)(Y,Z)=\frac{1}{2}\{\langle \nabla _Y X,Z\rangle + \langle \nabla _Z X,Y\rangle \},\quad \forall \ Y, Z \in {\mathfrak X}(M).
\end{equation}
Let ${\mathfrak L}_X$ denote the Lie derivative in the direction of the vector field $X$. Then the deformation tensor ${\rm{Def }}\ X$ can be equivalently defined as ${\rm{Def }}\ X=\frac{1}{2}{\mathfrak L}_Xg$, where $g$ is the metric tensor of $M$ (see, e.g., \cite[p. 958]{M-T}). Therefore, ${\rm{Def }}\ X\in S^2T^*M$, where $S^2T^*M$ denotes the set of symmetric
tensor fields of type $(0,2)$, and the {\it deformation operator}
\begin{equation}
\label{Def}
{\rm{Def }}:C^{\infty }(M,TM)\to C^{\infty }(M, S^2T^*M)
\end{equation}
extends to a linear and continuous operator
\begin{equation}
\label{Def-1}
{\rm{Def }}:H^1(M,\Lambda ^1TM)\to L^2(M,S^2T^*M).
\end{equation}
If $X\in {\mathfrak X}(M)$, $X=X^j\partial _j$, we use the notation $X_{j;\ell }:=\partial _\ell X_j-\Gamma _{j\ell }^kX_k$, where $X_{\ell }=g_{\ell k}X^k$. Then ${\rm{Def }}\, X$ has the local representation
\begin{equation}
\label{deformation-tensor-new-sm}
({\rm{Def }}\, X)_{j\ell }=\frac{1}{2}\left(X_{j;\ell }+X_{\ell ;j}\right).
\end{equation}
In addition, the adjoint ${\rm{Def }}^*$ of the operator (\ref{Def}) is given by
${\rm{Def}}^*{\bf u}=-{\rm{div}}\ {\bf u}, \ \forall \ u\in S^2T^*M$
(cf., e.g., \cite[Chapter 2, p. 137]{Taylor}).

A deformation-free vector field $X\in {\mathfrak X}(M)$,
\begin{equation}
\label{Killing-field} {\rm{Def}}\ X=0 \ \mbox{ on }\ M \Longleftrightarrow {\mathfrak L}_Xg=0\ \mbox{ on }\ M,
\end{equation}
is called a {\it Killing field}. Thus, the flow generated by such a field consists of isometries, that is, the flow leaves the metric $g$ invariant. 
Moreover, by \eqref{deformation-tensor-new-sm}, $X$ is a Killing field (i.e., $X$ generates a group of isometries) if and only if (cf., e.g., \cite[Chapter 2, (3.33)]{Taylor}, see also \cite{Semmelmann})
\begin{equation}
\label{Killing-field-local}
X_{j;\ell }+X_{\ell ;j}=0, \ \ j, \ell =1,\ldots ,m.
\end{equation}
Throughout the paper we often need the assumption below (cf. \cite[(3.1)]{M-T}).
\begin{assumption}
\label{H}
{$M$} is a smooth, compact, boundaryless Riemannian manifold that does not have any non-trivial Killing vector field.
\end{assumption}
Assumption \ref{H} assures the invertibility of the elliptic operator\\ ${\bf L}:C^{\infty }(M,\Lambda ^1TM)\to C^{\infty }(M,\Lambda ^1TM)$ given by \eqref{Laplace-operator},
which {implies existence of} the fundamental solution of the Stokes system on $M$. In addition, the  extended operator ${\bf L}:H^1(M,\Lambda ^1TM)\to H^{-1}(M,\Lambda ^1TM)$ is invertible, as well (see also \cite[(3.3)]{D-M}, \cite[(3.3)]{M-T}).

\subsection{\bf Brinkman, Navier-Stokes and Darcy-Forchheimer-Brinkman systems on a compact Riemannian manifold}

We now consider the second-order elliptic differential operator
\begin{equation}
\label{Laplace-operator}
{\bf L}:\mathfrak{X}(M)\to \mathfrak{X}(M),\quad
{\bf L}:=2{\rm{Def}}^{*}{\rm{Def}}= -\triangle +d\delta -2\rm{Ric},
\end{equation}
where $\triangle :=-(d\delta +\delta d)$ is the Hodge Laplacian and ${\rm{Ric}}$ is the Ricci tensor of $M$ (see, e.g., \cite[(2.6)]{D-M}). In the Euclidean setting and for free divergence vector fields, ${\bf L}$ reduces to the Laplace operator, which is a second order differential operator with constant coefficients. Nevertheless, in the case of a Riemannian manifold, ${\bf L}$ is an elliptic second order differential operator with variable coefficients, as it depends on the metric structure of the manifold, {and is the main operator that appears in the structure of the Stokes system on such a manifold (cf. \cite[pp. 161, 162]{Ebin-Marsd}; see also \cite{D-M-M,M-T}).} For any $s\in (0,1)$, the operator (\ref{Laplace-operator}) extends to a bounded linear operator
\begin{equation}
\label{Laplace-operator-new}
{\bf L}=2{\rm{Def}^{*}}{\rm{Def}}:H^{s+\frac{1}{2}}(M,\Lambda ^1TM)\to H^{s-\frac{3}{2}}(M,\Lambda ^1TM)
\end{equation}
(see, e.g., \cite[p. 177]{La-Mi}). Let {$V\in {L^{\infty}(M,\Lambda ^1TM\otimes \Lambda ^1TM)}$} be a given {symmetric} tensor field which satisfies the following positivity condition with respect to the $L^2(M,\Lambda ^1TM)$-inner product (denoted by $\langle \cdot ,\cdot \rangle _M$)
\begin{equation}
\label{V-positive}
{\langle V{\bf v},{\bf v}\rangle _M\geq 0,\ \forall \ {\bf v}\in L^2(M,\Lambda ^1TM).}
\end{equation}
This condition implies immediately that
\begin{equation}
\label{V-positive-open}
\langle V{\bf u},{\bf u}\rangle _{\Omega }\geq 0,\ \forall \ {\bf u}\in L^2(\Omega,\Lambda ^1TM).
\end{equation}
Then we define the {\it generalized Brinkman operator} $B_V$ as
\begin{align}
\label{eq:2.24a-new}
&B_V:H^{s+\frac{1}{2}}(M,\Lambda ^1TM)\times H^{s-\frac{1}{2}}(M)\to H^{s-\frac{3}{2}}(M,\Lambda ^1TM)\times H^{s-\frac{1}{2}}(M),\nonumber\\
&B_V:=\left(\begin{array}{cc}
{\bf L} + V\ & \ d\\ \delta \ & \ 0
\end{array}\right)
\end{align}
{(see, e.g., \cite{10-new2}). Note that the operator
\begin{align}
\left(\begin{array}{cc}
V & d\\
\delta & 0
\end{array}
\right):H^{s+\frac{1}{2}}(M,\Lambda ^1TM)\!\times \!H^{s-\frac{1}{2}}(M)\!\to \!H^{s-\frac{3}{2}}(M,\Lambda ^1TM)\!\times \!H^{s-\frac{1}{2}}(M)\nonumber
\end{align}
is compact, due to the continuity of $V:H^{s+\frac{1}{2}}(M,\Lambda ^1TM)\to L^2(M,\Lambda ^1TM)$ and the compactness of the embedding $L^2(M,\Lambda ^1TM)\hookrightarrow H^{s-\frac{3}{2}}(M,\Lambda ^1TM)$. Therefore, $B_V$ appears as a compact perturbation of the {\it Stokes operator $B_0$}. In particular, if $V\equiv \alpha {\mathbb I}$, where ${\mathbb I}$ is the identity operator and $\alpha >0$ is a constant, then the operator}
\begin{align}
\label{eq:2.24a}
&B_\alpha :H^{s+\frac{1}{2}}(\Omega ,\Lambda ^1TM)\times H^{s-\frac{1}{2}}(\Omega )\to H^{s-\frac{3}{2}}(\Omega ,\Lambda ^1TM)\times H^{s-\frac{1}{2}}(\Omega ),\nonumber
\end{align}
\begin{align}
&B_\alpha :=
\left(\begin{array}{cc}
{\bf L} + \alpha {\mathbb I}\ & \ d\\ \delta \ & \ 0
\end{array}\right)
\end{align}
is called the {\it Brinkman operator}. Both {operators, $B_\alpha $ and $B_0$,} are Agmon-Douglis-Nirenberg elliptic (see, e.g., \cite{H-W,10-new2}), and, in the Euclidean setting, they play a main role in fluid mechanics and porous media (see, e.g., \cite{Ni-Be}).

Let $s\in (0,1)$. Then the nonlinear system
\begin{equation}
\label{eq:2.24a-ms}
\left\{\begin{array}{lll}
{\bf L}{\bf u}+\nabla _{\bf u}{\bf u}+d\pi ={\bf f}\in H^{s-\frac{3}{2}}(\Omega ,\Lambda ^1TM)\ \mbox{ in }\ \Omega ,\\
\delta {\bf u}=0\ \mbox{ in }\ \Omega ,
\end{array}\right.
\end{equation}
given in terms of the Levi-Civita connection and with the unknowns $({\bf u},\pi )\in H^{s+\frac{1}{2}}(\Omega ,\Lambda ^1TM)\times H^{s-\frac{1}{2}}(\Omega )$, is the {\it Navier-Stokes system} on the compact Riemannian manifold $M$ (cf., e.g., \cite{Ebin-Marsd}; see also \cite{M-T,D-M,H-M-T}). Recall that $\Gamma _{rj}^\ell$ are the Christoffel symbols of second kind associated to the metric $g$ of $M$, and note that, if ${\bf u}=u^j\partial _j$, then $\nabla _{{\bf u}}{\bf u}$ has the local representation
\begin{equation}
\label{locally}
\left(\nabla _{{\bf u}}{\bf u}\right)^\ell =u^j\partial _ju^\ell+\Gamma _{rj}^\ell u^ru^j.
\end{equation}
Let $k,\beta >0$ be given constants and {$V\in {L^{\infty }(M,\Lambda ^1TM\otimes \Lambda ^1TM)}$} be a { symmetric} tensor field satisfying the positivity condition \eqref{V-positive}. Taking into account the form (\ref{eq:2.24a-ms}) of the Navier-Stokes system, we now define the nonlinear {\it generalized Darcy-Forchheimer-Brinkman system} on the compact Riemannian manifold $M$, as
\begin{equation}
\label{eq:2.24a-sm-new}
\left\{\begin{array}{lll}
{\bf L}{\bf v}+V{\bf v}+k|{\bf v}|{\bf v}+\beta \nabla _{\bf v}{\bf v}+dp={\bf f}\in H^{s-\frac{3}{2}}(\Omega ,\Lambda ^1TM)\ \mbox{ in }\ \Omega ,\\
\delta {\bf u}=0\ \mbox{ in }\ \Omega ,
\end{array}\right.
\end{equation}
where $({\bf v},p)\in H^{s+\frac{1}{2}}(\Omega ,\Lambda ^1TM)\times H^{s-\frac{1}{2}}(\Omega )$ are the unknowns of this system. Having in view the form of this system in the Euclidean setting (see, e.g., \cite{Ni-Be}), in the particular case $V\equiv \alpha {\mathbb I}$, where $\alpha >0$ is a constant, we obtain the nonlinear {\it Darcy-Forchheimer-Brinkman system} on a compact manifold, as
\begin{equation}
\label{eq:2.24a-sm}
\left\{\begin{array}{lll}
{\bf L}{\bf v}+\alpha {\bf v}+k|{\bf v}|{\bf v}+\beta \nabla _{\bf v}{\bf v}+dp={\bf f}\in H^{s-\frac{3}{2}}(\Omega ,\Lambda ^1TM)\ \mbox{ in }\ \Omega ,\\
\delta {\bf u}=0\ \mbox{ in }\ \Omega .
\end{array}\right.
\end{equation}

\begin{itemize}
\item[$\bullet $]
{Throughout the paper we assume that $\dim (M)\!\in \!\{2,3\}$, {\it whenever the Navier-Stokes and Darcy-Forchheimer-Brinkman systems are involved}}.
\end{itemize}

\subsection{\bf The conormal derivative operator for the Brinkman system}

If $\Omega \subset M$ is a Lipschitz domain, denote by $d\sigma $ the surface measure on {its boundary, $\partial \Omega $,} and by $\nu $ the outward unit {normal}, which is defined a.e. on $\partial \Omega $, with respect to $d\sigma $. Let $\Omega _{+}:=\Omega $ and $\Omega _{-}:=M\setminus \overline{\Omega }$.
{Let $S$ either be an open subset or a surface in $M$. Then, all along the paper, we use the
notation $\langle \cdot ,\cdot \rangle _S$ for the duality pairing of two dual Sobolev spaces defined on $S$.}
Let $({\bf u},\pi)\in C^1(\overline\Omega_{\pm},\Lambda ^1TM)\times C^0(\overline\Omega_{\pm})$. {Then} the
{\it interior and exterior conormal derivatives or the traction fields} ${{\bf t}^\pm}({\bf u},\pi )$ for the {\it {generalized} incompressible Brinkman} (or {\it Stokes}) {\it operator} \eqref{eq:2.24a} are defined by using the formula
\begin{align}
\label{2.37-}
{\bf t}^{\pm}({\bf u},\pi ):={\gamma _\pm}\left(-\pi{\mathbb I}+2{\rm{Def}}({\bf u})\right)\nu ,
\end{align}
where $\nu{=\nu^+}$ is the outward unit normal to $\Omega_{ +}$, defined a.e. on $\partial {\Omega }$. Then {for the generalized incompressible Brinkman operator} we obtain the first Green identity
\begin{align}
\label{special-case-1}
{\pm}\left\langle {\bf t}^{\pm}({\bf u},\pi ),{{\bf w}} \right\rangle _{_{\partial\Omega  }}=&2\langle {\rm{Def}}\,{\bf u},{\rm{Def}}\,{{\bf w}} \rangle _{\Omega_{\pm}}+\langle{V} {\bf u},{{\bf w}} \rangle _{\Omega_{\pm}}+\langle \pi ,\delta {{\bf w}} \rangle _{\Omega_{\pm}}\nonumber\\
&-\left\langle {\mathcal L}_{V}({\bf u},\pi ),{{\bf w}} \right\rangle _{{\Omega_{\pm}}}, \forall \ {{\bf w}} \in {\mathcal D}(M,\Lambda ^1TM),
\end{align}
where
\begin{align}
\label{space-Brinkman-generalized-Robin-new}
{\mathcal L}_{V}({\bf u},\pi ):={\bf L}{\bf u}+V{\bf u}+d\pi.
\end{align}

Let {$\mathring E_\pm$} be the operator of extension of functions {from $H^{q}(\Omega _{\pm })$ or vector fields (one forms) from $H^{q}(\Omega _{\pm },\Lambda ^1TM)$ or tensor fields from {$H^{q}(\Omega _{\pm },S^2TM)$}, ${q}\ge 0$,} by zero on $M\setminus \Omega_{\pm }$. Following the proof of Theorem 2.16 in \cite{Mikh}, let us define the operator $\widetilde E_\pm$ on $H^{q}(\Omega _{\pm })$, for $0\le {q}<\frac{1}{2}$, as
\begin{align}
\label{ext-1}
\widetilde E_\pm:=\mathring E_\pm ,
\end{align}
and for $-\frac{1}{2}<{q}<0$, as
\begin{align}
\label{ext-2}
\!\!\!\!\!\langle\widetilde E_\pm h,v\rangle_{\Omega _\pm}:=\langle h, \widetilde E_\pm v\rangle_{\Omega _\pm}\!=\!{\langle h, \mathring E_\pm v\rangle_{\Omega _\pm}},\
h\in H^{q}(\Omega _{\pm }),\, v\in H^{-q}(\Omega _{\pm }).
\end{align}
Then, evidently ${\widetilde E}_\pm: H^{q}(\Omega_\pm)\to \widetilde H^{q}(\Omega_\pm)$, $-1/2<q<1/2$, is a bounded linear {extension} operator.
Similar definition and properties hold also for vector and tensor fields.

If $s\in (0,1)$ and $({\bf u},\pi )\in H^{s+\frac{1}{2}}(\Omega _{\pm },\Lambda ^1TM)\times H^{s-\frac{1}{2}}(\Omega _{\pm })$, we {have}
$
{\mathcal L}_{V}({\bf u},\pi )
\in H^{s-\frac{3}{2}}(\Omega _{\pm },\Lambda ^1TM).
$
Then identity \eqref{special-case-1} {suggests} the following weak definition of the conormal derivative in the setting of $L^2$-based Sobolev spaces on Lipschitz domains in compact Riemannian manifolds (cf. \cite[Lemma 3.2]{Co}, {\cite[Lemma 2.2]{K-L-W1},}
\cite[Definition 3.1, Theorem 3.2]{Mikh}, \cite[Definition 5.2]{Mikh-3}, \cite[Definition 2.3]{Mikh-2015}, \cite[Proposition 3.6]{M-M-W}, \cite[Theorem 10.4.1]{M-W} in the Euclidean setting, and \cite[Lemma 2.4]{10-new2} in the case of a matrix operator of type (\ref{eq:2.24a}) on Riemannian manifolds).
{\begin{defn}
\label{conormal-derivative-generalized}
Let $M$ be a compact Riemannian manifold with ${\rm{dim}}(M)\geq 2$. Let $\Omega _{+}:=\Omega \subset M$ be a Lipschitz domain. Let $\Omega _{-}:=M\setminus \overline{\Omega }$. Let $s\in (0,1)$ and {$V\in L^{\infty }(M,\Lambda ^1TM\otimes \Lambda ^1TM)$} be a {symmetric} tensor field satisfying the positivity condition \eqref{V-positive}. We consider the space\footnote{The condition in (\ref{space-Brinkman-generalized-Robin}) is suggested by the linear form of the incompressible Navier-Sokes equation in the Euclidean setting.}
\begin{align}
\label{space-Brinkman-generalized-Robin}
\pmb{\mathcal H}^{s+\frac{1}{2}}(\Omega _{\pm },&{\mathcal L}_{V}):=\Big\{({\bf u},\pi ,\tilde{\bf f})\in H^{s+\frac{1}{2}}(\Omega _{\pm },\Lambda ^1TM)\times H^{s-\frac{1}{2}}(\Omega _{\pm })\nonumber\\
&\times\widetilde{H}^{s-\frac{3}{2}}(\Omega _{\pm },\Lambda ^1TM):\tilde{\bf f}|_{\Omega _\pm }={\mathcal L}_{V}({\bf u},\pi ) {\text{ in } \Omega _{\pm }}\Big\}.
\end{align}
Then the generalized conormal derivative operator\footnote{The $\pm $ sign in the left-hand side of the formula (\ref{conormal-derivative-generalized-1}) corresponds to the domain $\Omega _\pm $. Also, for the sake of brevity we use the notations ${\bf t}_{V}^{\pm }({\bf v},p)$ instead of
${\bf t}_{V}^{\pm }({\bf v},p;{\bf 0})$, which is in fact the {\it canonical conormal derivative} (cf. \cite[Section 3.3]{Mikh}).}
\begin{align*}
\pmb{\mathcal H}^{s+\frac{1}{2}}(\Omega _{\pm },{\mathcal L}_{V})\ni ({\bf u},\pi ,\tilde{\bf f})\longmapsto {\bf t}_{V}^{\pm }({\bf u},\pi ;\tilde{\bf f})\in H^{s-1}(\partial \Omega ,\Lambda ^1TM).\nonumber
\end{align*}
is defined as
\begin{align}
\label{conormal-derivative-generalized-1}
&\pm \left\langle{\bf t}_{V}^{\pm }({\bf u},\pi ;\tilde{\bf f}),{{\bf\Phi} }\right\rangle _{_{\partial \Omega }}:=2\left\langle {{{\widetilde E}_\pm}\rm{Def}}\ {\bf u},{\rm{Def}}(\gamma _{\pm }^{-1}{{\bf\Phi}})\right\rangle _{\Omega _{\pm }}
+\left\langle {\widetilde E}_\pm (V{\bf u}),\gamma _{\pm }^{-1}{{\bf\Phi}}\right\rangle _{\Omega _{\pm }}\nonumber\\
&+\left\langle {{\widetilde E}_\pm}\pi, \delta (\gamma _{\pm }^{-1}{{\bf\Phi}})\right\rangle _{\Omega _{\pm }}-\langle \tilde{\bf f},\gamma _{\pm }^{-1}{{\bf\Phi}}\rangle _{\Omega _{\pm }},\ \forall \ {\bf\Phi} \in H^{1-s}(\partial \Omega,\Lambda ^1TM).
\end{align}
\end{defn}
\begin{lem}
\label{conormal-derivative-generalized-n1}
{Let $s\in (0,1)$.} Then the generalized conormal derivative operator
$$
{\bf t}_{V}^{\pm }:\pmb{\mathcal H}^{s+\frac{1}{2}}(\Omega _{\pm },{\mathcal L}_{V})\to H^{s-1}(\partial \Omega ,\Lambda ^1TM)
$$
is linear, bounded and independent of the choice of the right inverse\\ $\gamma _{\pm }^{-1}:H^{1-s}(\partial \Omega,\Lambda ^1TM)\to H^{\frac{3}{2}-s}(\Omega,\Lambda ^1TM)$ of the trace operator\\ $\gamma _{\pm }:H^{\frac{3}{2}-s}(\Omega,\Lambda ^1TM)\to H^{1-s}(\partial \Omega,\Lambda ^1TM)$. In addition, the following {Green identity} holds
\begin{align}
\label{conormal-derivative-generalized-2}
&\pm \big\langle{\bf t}_{V}^{\pm }({\bf u},\pi ;\tilde{\bf f}),\gamma _{\pm }{\bf w}\big\rangle_{_{\partial \Omega }}=2\left\langle {{{\widetilde E}_\pm}\rm{Def}}\ {\bf u},{\rm{Def}}\ {\bf w}\right\rangle _{\Omega _{\pm }}\nonumber\\
&\hspace{2em}+\left\langle {\widetilde E}_\pm (V{\bf u}),{\bf w}\right\rangle _{\Omega _{\pm }}+\left\langle {{\widetilde E}_\pm}\pi, \delta {\bf w}\right\rangle _{\Omega _{\pm }}
-\big\langle \tilde{\bf f},{\bf w}\big\rangle _{\Omega _{\pm }},\nonumber\\
&\hspace{5em}\forall\ ({\bf u},\pi ,\tilde{\bf f})\in \pmb{\mathcal H}^{s+\frac{1}{2}}(\Omega _{\pm },{\mathcal L}_{V}),\ {\bf w}\in H^{\frac{3}{2}-s}(\Omega _\pm ,\Lambda ^1TM).
\end{align}
\end{lem}}
The proof of Lemma \ref{conormal-derivative-generalized-n1} is based on arguments similar to those for \cite[Lemma 2.2]{K-L-W1} and \cite[Lemma 2.2]{JDDE}, but we omit them for the sake of brevity.
\begin{rem}\label{R2.4}
In view of formula (\ref{conormal-derivative-generalized-1}), the conormal derivatives for the ope\-rators ${B}_V$ and ${B}_0$ are related by the formula
{\begin{align}
\label{connection-conormals}
{\bf t}_{V}^{\pm }\left({\bf u},\pi ;\tilde{\bf f}\right)={\bf t}_{0}^{\pm }\left({\bf u},\pi ;\tilde{\bf f}-{\widetilde E}_\pm\left(V{\bf u}\right)\right).
\end{align}}
\end{rem}
\noindent For further arguments, we need the following result\footnote{${\mathcal{L}}(X,Y):=\{T:X\to Y:\mbox{ T is linear and bounded}\}$, and ${\rm{Ker}}\{T:X\to Y\}:=\{x\in X: T(x)=0\}$.} (see, {e.g.} \cite[Proposition 10.6]{Agr-1}, \cite[Lemma 11.9.21]{M-W}).
\begin{lem}
\label{Fredholm-kernel}
Assume that $X_j$, $Y_j$, $j=1,2$, are Banach spaces such that the inclusions $X_1\hookrightarrow X_2$ and
$Y_1\hookrightarrow Y_2$ are continuous and the second of them has dense range. Let $T\in {\mathcal{L}}(X_1,Y_1)\cap {\mathcal{L}}(X_2,Y_2)$ be a Fredholm operator
such that ${\rm{index}}(T:X_1\to Y_1)={\rm{index}}(T:X_2\to Y_2).$ Then
$${\rm{Ker}}\{T:X_1\to Y_1\}={\rm{Ker}}\{T:X_2\to Y_2\}.$$
\end{lem}

\section{Newtonian and layer potentials for the Stokes system}

{{Let} ${\rm{OPS}}_{{\rm{cl}}}^{r}$ denote the set of all classical pseudodifferential operators of order $r$ on $M$ (see, e.g., \cite[Chapter 8]{Wloka} for the definition and properties of such operators on a compact manifold).}

Let $({\mathcal G}(\cdot ,\cdot ),\Pi (\cdot ,\cdot ))\in {\mathcal D}'(M,S^2TM) \times {\mathcal D}'(M,\Lambda ^1TM)$ be the fundamental solution of the Stokes operator ${\mathcal B}_0$ in $M$, {\cite{M-T}, (see also \cite{10-new2} for the extension to the Brinkman operator ${\mathcal B}_V$, with $V\in C^{\infty }(M)$). Then ${\mathcal G}(\cdot ,\cdot )$ and $\Pi (\cdot ,\cdot )$ satisfy the following equations on $M$
\begin{equation}
\label{Brinkman-operator4} {\bf L}_x{\mathcal G}(x,y)+d_x\Pi (x,y)={{\rm Dirac}}_y(x),\ \ \delta _x{\mathcal G}(x,y)=0,
\end{equation}
where ${\rm{Dirac}}_y$ is the Dirac distribution with mass at $y$, and the subscript $x$ added to a differential operator refers to the action of that operator with respect to the variable $x$.
In view of \cite[(3.22)]{D-M} there exists an operator $\Upsilon \in OPS_{{\rm{cl}}}^{0}(M,{\mathbb R})$ such that the Schwartz kernel $\Xi $ of $\Upsilon ^{\top }$ satisfies the equation
\begin{equation}
\label{pseudo-operator-zero-1}
{\bf L}_x\Pi ^{\top}(y,x)=d_x\Xi (x,y),
\end{equation}
and that (see also \cite[(3.27)]{D-M})
\begin{align}
\label{dl-sm}
\delta_x \Pi ^{\top }(y,x)=-{\rm{Dirac}}_y(x).
\end{align}

\subsection{Newtonian potentials for the Stokes operator}

Let {$s\in (0,1)$.} Let ${\boldsymbol\varphi}\in {\mathcal D}(M,\Lambda ^1TM)$. Then the volume velocity and pressure potentials of the Stokes system, ${\mathcal N}_{{M}}{\boldsymbol\varphi}$ and ${\mathcal Q}_{{M}}{\boldsymbol\varphi}$, are, respectively, defined at any ${\bf x}\in M$ by
\begin{align}
\label{NoT}
\!\!\!\!\!&\big({\mathcal N}_{{M}}{\boldsymbol\varphi}\big)({\bf x}):=
\big\langle {\mathcal G}({\bf x}, \cdot),{\boldsymbol\varphi}\big\rangle_{_{M}}
=\int _{M}{{\mathcal G}({\bf x},{\bf y}){\boldsymbol\varphi}({\bf y})} \, d{\rm{Vol}}_{\bf y},\\
\label{QoT}
\!\!\!\!\!&\big({\mathcal Q}_{{M}}{\boldsymbol\varphi}\big)({\bf x}) :=
\big\langle{\Pi }({\bf x},\cdot),{\boldsymbol\varphi}\big\rangle _{M}=\int _{M} \left\langle {\Pi}({\bf x},{\bf y}),{\boldsymbol\varphi}({\bf y})\right\rangle d{\rm{Vol}}_{\bf y},
\end{align}
and the operators
${\mathcal N}_{M}:C^{\infty }(M,\Lambda ^1TM)\to C^{\infty }(M,\Lambda ^1TM)$ and\\
${\mathcal Q}_{{M}}:C^{\infty }(M,\Lambda ^1TM)\to C^{\infty }(M)$ are continuous. In addition, definitions \eqref{NoT} and \eqref{QoT} can be extended to Sobolev {spaces,} and the operators
\begin{align}
\label{Newtonian-Stokes-R}
&{\mathcal N}_{M}:{H}^{s-\frac{3}{2}}(M,\Lambda ^1TM)\to H^{s+\frac{1}{2}}(M,\Lambda ^1TM),\\
\label{Newtonian-Stokes-R1}
&{\mathcal Q}_{M}:{H}^{s-\frac{3}{2}}(M,\Lambda ^1TM)\to H^{s-\frac{1}{2}}(M)
\end{align}
are linear and continuous, due to the fact that the volume velocity and pressure potentials ${\mathcal N}_{M}$ and ${\mathcal Q}_{M}$ are pseudodifferential operators of order $-2$ and $-1$, respectively { (cf., e.g., \cite[Theorem 3.2]{10-new2}; see also \cite[(5.12), (5.13)]{D-M}), and \cite[Lemma 3.2]{K-L-M-W} for the similar mapping properties in the Euclidean setting)}.

Let $r_{{\Omega }_{\pm }}$ be the operators {restricting} (scalar-valued or vector-valued) distributions in $M$ to $\Omega_{\pm}$. Then we define the Newtonian velocity and pressure potentials ${\mathcal N}_{{\Omega }_{\pm }}\tilde{\bf F}_{\pm}$ with densities ${\widetilde{\bf F}}_{\pm }\in \widetilde{H}^{s-\frac{3}{2}}({\Omega }_{\pm },\Lambda ^1TM)$ as the restrictions
\begin{align}
\label{proof-2-B}
{\mathcal N}_{{\Omega }_{\pm }}\widetilde{\bf F}_{\pm}:=r_{{\Omega _{\pm }}}\big({\mathcal N}_{M}\widetilde{\bf F}_{\pm }\big),\ \
{\mathcal Q}_{{\Omega }_{\pm }}\widetilde{\bf F}_{\pm}:=r_{{\Omega _{\pm }}}\big({\mathcal Q}_{M}\widetilde{\bf F}_{\pm }\big).
\end{align}
By the continuity of the volume velocity potential operator given in \eqref{Newtonian-Stokes-R},
as well as the continuity of the restriction operators $r_{\Omega _{\pm}}:{H}^{s+\frac{1}{2}}(M,\Lambda ^1TM)\to {H}^{s+\frac{1}{2}}(\Omega _{\pm},\Lambda ^1TM)$ and since $\widetilde{H}^{s-\frac{3}{2}}({\Omega }_{\pm },\Lambda ^1TM)$ is a subspace of the space ${H}^{s-\frac{3}{2}}(M,\Lambda ^1TM)$, we deduce that
the Newtonian velocity potentials
\begin{align}
\label{Newtonian-velocity}
{\mathcal N}_{{\Omega }_{\pm }}:\widetilde{H}^{s-\frac{3}{2}}({\Omega }_{\pm },\Lambda ^1TM)\to
{H}^{s+\frac{1}{2}}({\Omega }_{\pm },\Lambda ^1TM)
\end{align}
are continuous as well.
A similar argument as above implies the continuity of the Newtonian pressure potential operator
\begin{align}
\label{Newtonian-pressure}
{\mathcal Q}_{{\Omega }_{\pm }}:\widetilde{H}^{s-\frac{3}{2}}({\Omega }_{\pm },\Lambda ^1TM)\to
{H}^{s-\frac{1}{2}}({\Omega }_{\pm }).
\end{align}
Let ${\widetilde{\bf F}}_{\pm}\in \widetilde{H}^{s-\frac{3}{2}}(\Omega _{\pm },\Lambda ^1TM)$.
Due to (\ref{Brinkman-operator4}){-(\ref{dl-sm})}, the {pair} $\big({\mathcal N}_{\Omega _{\pm }}{\widetilde{\bf F}}_{\pm},{\mathcal Q}_{\Omega _{\pm }}{\widetilde{\bf F}}_{\pm }\big)$ {satisfies} the relations
\begin{align}
\label{4.21}
&{\mathcal N}_{{\Omega }_{\pm }}{{\widetilde{\bf F}}}_{\pm }\in {H^{s+\frac{1}{2}}}({\Omega }_{\pm },\Lambda ^1TM),\ {\mathcal Q}_{{\Omega }_{\pm }}{{\widetilde{\bf F}}}_{\pm }\in H^{s-\frac{1}{2}}({\Omega }_{\pm }),
\\
\label{Newtonian-s1}
&{\bf L}{\mathcal N}_{{\Omega }_{\pm }}{{\widetilde{\bf F}}}_{\pm }+d{\mathcal Q}_{{\Omega }_{\pm }}{{\widetilde{\bf F}}}_{\pm }={{\widetilde{\bf F}}}_{\pm },\ \delta \left({\mathcal N}_{{\Omega }_{\pm }}{{\widetilde{\bf F}}}_{\pm }\right)=0\ \mbox{ in }\ {\Omega }_{\pm }.
\end{align}

\subsection{Layer potentials for the Stokes system}

Let $s\in [0,1]$ and ${\boldsymbol\psi}\in H^{s-1}(\partial \Omega ,\Lambda ^1TM)$. Then the single-layer velocity and pressure potentials ${\bf V}_{\partial \Omega }{\boldsymbol\psi}$ and ${\mathcal Q_{\partial \Omega}\boldsymbol\psi}$ for the Stokes system are given by
\begin{equation}
\label{single-layer-Brinkman1} ({\bf V}_{\partial \Omega }{\boldsymbol\psi})(x):=-\langle {\mathcal
G}(x,\cdot ),{\boldsymbol\psi}\rangle _{\partial \Omega },\ ({\mathcal Q_{\partial \Omega}\boldsymbol\psi})(x):=-\langle \Pi (x,\cdot ),{\boldsymbol\psi}\rangle _{\partial \Omega }, \ x\in M\setminus {\partial \Omega }
\end{equation}
(see \cite[(2.19),(2.20)]{D-M}). Now let ${\boldsymbol \phi}\in H^s_{\nu }(\partial \Omega ,\Lambda ^1TM)$ be given, where
\begin{equation}
\label{Sobolev-normal} H^s_{\nu }(\partial \Omega ,\Lambda ^1TM):=
\left\{{\boldsymbol \phi}^0\in H^s(\partial \Omega ,\Lambda ^1TM):\left\langle \nu ,{\boldsymbol \phi}^0\right\rangle _{\partial \Omega }=0\right\}.
\end{equation}
Then the double-layer velocity and pressure potentials ${\bf W}_{\partial \Omega }{\boldsymbol \phi}$ and ${\mathcal P}_{\partial \Omega }{\boldsymbol \phi}$ for the Stokes system are defined at any $x\in M\setminus {\partial \Omega }$ by
\begin{align}
\label{double-layer-Brinkman1} &({\bf W}_{\partial \Omega }{\boldsymbol \phi})(x)
:=\int _{\partial \Omega }\left\langle 2({\rm{Def }}\ {\mathcal G}(x,y))\nu (y)-\Pi ^{\top}(y,x){\otimes }\nu (y),{\boldsymbol \phi}(y)\right\rangle d\sigma _y,\\
\label{double-layer-Brinkman7} &({\mathcal P}_{\partial \Omega }{\boldsymbol \phi})(x)
:=\int _{\partial \Omega }\left\langle 2({\rm{Def }}\ \Pi (x,y))\nu (y)+\Xi (x,y)\nu (y)
,{\boldsymbol \phi}(y)\right\rangle d\sigma _y,
\end{align}
where $\Xi $ is the Schwartz kernel of the pseudodifferential operator $\Upsilon $, and $\Pi ^{\top}$ is the transpose of $\Pi $ (see \cite[(3.1),(3.25)]{D-M}). {Note that {in \eqref{double-layer-Brinkman1} and further on,} $({\rm{Def }}\, {\mathcal G}(x,\cdot ))\nu $ yields the action of the third order tensor field ${\rm{Def }}\, {\mathcal G}(x,\cdot )$ {on} $\nu $, while $\Pi ^{\top}(\cdot ,x)\nu $ is a dyadic product between the vector fields $\Pi ^{\top}(\cdot ,x)$ and $\nu $. Both terms, $({\rm{Def }}\, {\mathcal G}(x,\cdot ))\nu $ and $\Pi ^{\top}(\cdot ,x)\nu $, are second order tensor fields. Moreover, {in \eqref{double-layer-Brinkman7} and further on,} $({\rm{Def }}\, \Pi (x,\cdot ))\nu $ yields the action of the second order tensor field ${\rm{Def }}\, \Pi (x,\cdot )$ on $\nu $, while $\Xi (x,\cdot)\nu $ is a multiplication of the unit conormal $\nu $ with the {scalar} $\Xi (x,\cdot)$. Both $({\rm{Def }}\, \Pi (x,\cdot ))\nu $ and $\Xi (x,\cdot)\nu $ are vector fields.}

In addition, the principal value of ${\bf W}_{\partial \Omega }{\boldsymbol \phi}$ is denoted by ${\bf K}_{\partial \Omega }{\boldsymbol \phi}$ and is defined at a.e. $x\in \partial \Omega $ by
\begin{align}
\label{double-layer-Brinkman3}
&({\bf K}_{\partial \Omega }{\boldsymbol \phi})(x):={\rm{p.v.}}\int_{\partial \Omega}\big\langle {2}\left({\rm{Def }}_{y}\ {\mathcal G}(x,{y})\right)\nu {(y)}\!-\!{\Pi ^{\top}(y,x)\otimes \nu (y)},{\boldsymbol \phi}(y)\big\rangle d\sigma _y\nonumber\\
&=\lim_{\varepsilon \to 0}\!\int _{\{y\in \partial \Omega :\ {\rm{dist}}(x,y)>\varepsilon \}}\big\langle {2}\left({\rm{Def }}_{y}\ {\mathcal G}(x,{y})\right)\nu {(y)}-{\Pi ^{\top}(y,x)\otimes \nu (y)},{\boldsymbol \phi}(y)\big\rangle d\sigma _y,
\end{align}
where ${\rm{p.v.}}$ means the principal value, and ${\rm{dist}}(x,y)$ is the geodesic distance.

By (\ref{Brinkman-operator4}), (\ref{pseudo-operator-zero-1}), (\ref{dl-sm}), the pairs
$\big({\bf V}_{\partial\Omega }{\boldsymbol\psi},{{\mathcal Q}_{\partial\Omega}\boldsymbol\psi}\big)$ and $\big({\bf W}_{\partial \Omega }{\boldsymbol \phi},{\mathcal P}_{\partial \Omega }{\boldsymbol \phi}\big)$ satisfy the equations of the Stokes system

\begin{align}
\label{double-layer-Brinkman10}
&{\bf L}{\bf V}_{\partial \Omega }{\boldsymbol\psi}+d{\mathcal Q_{\partial \Omega}\boldsymbol\psi}=0,\ \delta {\bf V}_{\partial \Omega }{\boldsymbol\psi}=0 \mbox{ in } M\setminus \partial \Omega ,\\
&{\bf L}{\bf W}_{\partial \Omega }{\boldsymbol \phi}+d{\mathcal P}_{\partial \Omega }{\boldsymbol \phi}=0,\ \delta {\bf W}_{\partial \Omega }{\boldsymbol \phi}=0 \mbox{ in } M\setminus \partial \Omega .
\end{align}
For the main properties of the layer potentials, we refer the reader to
\cite[Proposition 4.2.5, 4.2.9, Corollary 4.3.2, Theorems 5.3.6, 5.4.1, 5.4.3, 10.5.3]{M-W}
for the Stokes system in the Euclidean setting,
\cite[Theorem 2.1, (3.5), Proposition 3.5]{D-M}, \cite[Theorems 3.1, 6.1]{M-T}
for the Stokes system in compact Riemannian manifolds, and
\cite[Theorems 4.3, 4.9, 4.11, (131), (132), (137), Lemma 5.4]{10-new2} for pseudodifferential Brinkman operators in Riemannian manifolds. Theorem \ref{single-layer-operator-Brinkman} shows some of these properties.

\section{Transmission problems for the Stokes and generalized Brinkman systems}
Let ${\mu }>0$ {be a constant}, {$V\in {L^{\infty }(M,\Lambda^1TM\otimes \Lambda ^1TM)}$} be a {symmetric} tensor field satisfying condition \eqref{V-positive} and ${\mathcal P}\in L^{\infty }(\partial \Omega ,\Lambda ^1TM\otimes \Lambda ^1TM)$ be a {symmetric tensor field} satisfying condition \eqref{Poisson-transmission3}. Recall that $\widetilde E_{\pm }$ {are the extension operators defined by} \eqref{ext-1}, \eqref{ext-2}. For $s\in (0,1)$, define the spaces
\begin{align}
\label{spaces-Poisson-transmission-MS}
&H_*^{s-\frac{1}{2}}(\Omega _{+}):=\big\{q\in H^{s-\frac{1}{2}}(\Omega _{+}): \big\langle \widetilde E_+ q,1\big\rangle _{\Omega _{+}}=0\big\},\\
&H^{s}_{\nu }(\partial \Omega ,\Lambda ^1TM):=\big\{{\bf \Phi }\in
H^{s}(\partial \Omega ,\Lambda ^1TM):\langle {\bf \Phi },\nu \rangle _{\partial \Omega }=0\big\},
\\
\label{X1-sm}
&{{\mathcal X}_{s}}:=
\big({H^{s+\frac{1}{2}}(\Omega _{+},\Lambda ^1TM)}\times H_*^{s-\frac{1}{2}}(\Omega _{+})\big)\nonumber\\
&\hspace{7em}\times \big({H^{s+\frac{1}{2}}(\Omega _{-},\Lambda ^1TM)}\times H^{s-\frac{1}{2}}(\Omega _{-})\big),\\
\label{Y1-sm}
&{{\mathcal Y}_{s}}:=\big(\widetilde{H}^{s-\frac{3}{2}}(\Omega _{+},\Lambda ^1TM)\times \widetilde{H}^{s-\frac{3}{2}}(\Omega _{-},\Lambda ^1TM)\big)\nonumber
\\
&\hspace{7em}\times \big(H^{s}_{\nu }(\partial \Omega ,\Lambda ^1TM)\times H^{s-1}(\partial \Omega ,\Lambda ^1TM)\big).
\end{align}

{We will use} the above {spaces,} ${\mathcal X}_{s}$ and ${\mathcal Y}_{s}$, {to} consider the following Poisson problem of transmission type for the {incompressible} Stokes and generalized Brinkman systems in the complementary Lipschitz domains $\Omega _{\pm }$
\begin{equation}
\label{Poisson-transmission4}
\left\{\begin{array}{ll}
({\bf L}+V){\bf u}_{+}+d\pi _{+}=\tilde{\bf f}_{+}|_{\Omega _{+}},\
\delta {\bf u}_{+}=0\ \mbox{ in }\ \Omega _{+},\\
{\bf L}{\bf u}_{-}+d\pi _{-}=\tilde{\bf f}_{-}|_{\Omega _{-}},\ \delta {\bf u}_{-}=0
\ \mbox{ in }\ \Omega _{-},\\
{\mu }\gamma _{+}{\bf u}_{+}-\gamma _{-}{\bf u}_{-}={\bf h}\ \mbox{ on }\ \partial \Omega ,\\
{\bf t}_{V}^{+}\left({\bf u}_{+},\pi _{+};\tilde{\bf f}_{+}\right)
-{{\bf t}_0^-}\left({\bf u}_{-},\pi _{-};\tilde{\bf f}_{-}\right)
+{\mathcal P} \gamma _{+}{\bf u}_{+}={\bf r}\ \mbox{ on }\ \partial \Omega .
\end{array}
\right.
\end{equation}
Let us also consider {another problem,}
\begin{equation}
\label{Poisson-transmission4-new}
\left\{\begin{array}{ll}
({\bf L}+V){\bf u}_{+}+d \pi _{+}=\tilde{\bf f}_{+}|_{\Omega _{+}},\ \delta {\bf u}_{+}=0
\ \mbox{ in }\ \Omega _{+},\\
{\bf L}{\bf u}_{-}+d\pi _{-}=\tilde{\bf f}_{-}|_{\Omega _{-}},\
\delta {\bf u}_{-}=0
\ \mbox{ in }\ \Omega _{-},\\
\gamma _{+}{\bf u}_{+}-\gamma _{-}{\bf u}_{-}={\bf h}\ \mbox{ on }\ \partial \Omega ,\\
{\bf t}_{V}^{+}\left({\bf u}_{+},\pi _{+};\tilde{\bf f}_{+}\right)
-{\mu }{{\bf t}_{0}^-}\left({\bf u}_{-},\pi _{-};\tilde{\bf f}_{-}\right)
+{\mathcal P} \gamma _{+}{\bf u}_{+}={\bf r}\ \mbox{ on }\ \partial \Omega ,
\end{array}
\right.
\end{equation}
where the transmission conditions on $\partial\Omega$ generalize the interface conditions obtained in, e.g., \cite{Ochoa_Tapia-Whitacker1995-1}, \cite{Ochoa_Tapia-Whitacker1995-2}, corresponding to two viscous fluids in $\Omega_+$ and $\Omega_-$, and the constant $\mu >0$ can be interpreted as the ratio of viscosity coefficients of such fluids. In addition, we note that for $\mu >0$ fixed, the following property is immediate:
\begin{align}
\label{change-var}
&\left(({\bf u}_{+},\pi _{+}),({\bf u}_{-},\pi
_{-})\right) \mbox{ solves } (\ref{Poisson-transmission4}) \mbox{ for } \left(\tilde{\bf f}_{+},\tilde{\bf f}_{-},{\bf h},{\bf r}\right)\Longleftrightarrow \nonumber
\\
&\left((\mu {\bf u}_{+},\mu \pi _{+}),({\bf u}_{-},\pi
_{-})\right) \mbox{ solves } (\ref{Poisson-transmission4-new}) \mbox{ for } \left(\mu \tilde{\bf f}_{+},\tilde{\bf f}_{-},{\bf h},\mu {\bf r}\right),
\end{align}
which shows the connection between problems (\ref{Poisson-transmission4}) and (\ref{Poisson-transmission4-new}). Hence, the well-posedness of the transmission problem (\ref{Poisson-transmission4-new}) is equivalent with the well-posedness of the transmission problem (\ref{Poisson-transmission4}). In view of this argument, we next analyze the transmission problem (\ref{Poisson-transmission4}) and show its well-posedness in the space ${\mathcal X}_{s}$ whenever the given data belong to the space ${\mathcal Y}_{s}$, ${s\in (0,1)}$, and $\mu \in (0,\infty )$. We extend the well-posedness results in \cite[Theorems 4.1 and 5.1]{JDDE} and \cite[Theorem 4.1]{K-G-P-W} to a more general case concerning the spaces of given data and the range of the involved parameter.
The proof of such a well-posedness result is based on a layer potential analysis for the Stokes system (see also the proof of \cite[Theorems 4.1]{JDDE}) combined with the invertibility of some related Fredholm operator of index zero.

\subsection{Uniqueness result the transmission problem \eqref{Poisson-transmission4} in the case $V=0$}

{Let us} consider the auxiliary transmission problem for the Stokes {system,}
\begin{equation}
\label{Poisson-transmission4-s}
\left\{\begin{array}{ll}
{\bf L}{\bf u}_{+}+d\pi _{+}=\tilde{\bf f}_{+}|_{\Omega _{+}},\ \delta {\bf u}_{+}=0
\ \mbox{ in }\ \Omega _{+},\\
{\bf L}{\bf u}_{-}+d\pi _{-}=\tilde{\bf f}_{-}|_{\Omega _{-}},\ \delta {\bf u}_{-}=0
\ \mbox{ in }\ \Omega _{-},\\
{\mu }\gamma _{+}{\bf u}_{+}-\gamma _{-}{\bf u}_{-}={\bf h}\ \mbox{ on }\ \partial \Omega ,\\
{{\bf t}_0^{+}}\left({\bf u}_{+},\pi _{+};\tilde{\bf f}_{+}\right)
-{{\bf t}_0^{-}}\left({\bf u}_{-},\pi _{-};\tilde{\bf f}_{-}\right)
+{\mathcal P}\gamma _{+}{\bf u}_{+}={\bf r}\ \mbox{ on }\ \partial \Omega ,
\end{array}
\right.
\end{equation}
i.e., problem (\ref{Poisson-transmission4}) {with $V=0$}, and show the following uniqueness result.
\begin{lem}
\label{Poisson-transmission-Stokes-P}
{Let $M$ satisfy Assumption $\ref{H}$ and $\dim (M)\geq 2$}. Let $\Omega _{+}:=\Omega \subset M$ be a Lipschitz domain. {Let $\Omega _{-}:=M\setminus \overline{\Omega }$ satisfy Assumption $\ref{connected-set}$.} Let $s\in (0,1)$ and ${\mu >0}$ be a given constant. Let ${\mathcal P}\in L^{\infty }(\partial \Omega ,\Lambda ^1TM\otimes \Lambda ^1TM)$ be a {symmetric tensor field} which satisfies the positivity condition \eqref{Poisson-transmission3}. Then for $\left(\tilde{\bf f}_{+},\tilde{\bf f}_{-},{\bf h},{\bf r}\right)\in {\mathcal Y}_{s}$, {the transmission problem} \eqref{Poisson-transmission4-s} has at most one solution $\left(({\bf u}_{+},\pi _{+}),({\bf u}_{-},\pi _{-})\right)\in {\mathcal X}_{s}$.
\end{lem}
\begin{proof}
{Let} $\left(({\bf u}_{+}^0,\pi _{+}^0),({\bf u}_{-}^0,\pi _{-}^0)\right)\in {\mathcal X}_{s}$ { satisfy} the homogeneous version of problem (\ref{Poisson-transmission4-s}). {Then} the vector fields ${\bf u}_{\pm }^0$  admit the layer potential representations {of the Stokes system solutions},
\begin{align}
\label{sm-r1}
&{\bf u}_{+}^0={\bf V}_{\partial \Omega }\left({{\bf t}_0^+}({\bf u}_{+}^0,\pi _+^0)\right)-{\bf W}_{\partial \Omega }\left(\gamma _{+}{\bf u}_{+}^0\right)\ \mbox{ in }\ \Omega _+,\\
\label{sm-r2}
&{\bf u}_{-}^0=-{\bf V}_{\partial \Omega }\left({{\bf t}_0^-}({\bf u}_{-}^0,\pi _{-}^0)\right)+{\bf W}_{\partial \Omega }\left(\gamma _{-}{\bf u}_{-}^0\right)\ \mbox{ in }\ \Omega _{-}
\end{align}
(see, e.g., \cite[(3.7)]{D-M}). By applying the trace {operators $\gamma _{+}$ to (\ref{sm-r1}) and  $\gamma _{-}$ to} (\ref{sm-r2}), and by using formulas {(\ref{single-layer-Brinkman8a}) and} (\ref{double-layer-Brinkman12}), we obtain the equations
\begin{align}
\label{sm-r3}
&\left(\frac{1}{2}{\mathbb I}+{\bf K}_{\partial \Omega }\right)\gamma _{+}{\bf u}_{+}^0={\mathcal V}_{\partial \Omega }\left({{\bf t}_0^+}({\bf u}_{+}^0,\pi _+^0)\right)\ \mbox{ on }\ \partial \Omega ,
\end{align}
\begin{align}
\label{sm-r4}
&\left(-\frac{1}{2}{\mathbb I}+{\bf K}_{\partial \Omega }\right)\gamma _{-}{\bf u}_{-}^0=
{\mathcal V}_{\partial \Omega }\left({{\bf t}_0^-}({\bf u}_{-}^0, \pi _{-}^0)\right)\ \mbox{ on }\ \partial \Omega.
\end{align}
Equations (\ref{sm-r3}), (\ref{sm-r4}) and the transmission conditions
\begin{equation}
\label{q2-0}
\!\!\!\!{\mu }\gamma _{+}{\bf u}_+^0\!-\!\gamma _{-}{\bf u}_{-}^0={\bf 0},\
{{\bf t}_0^+}({\bf u}_{+}^0,\pi _{+}^0)\!-\!{{\bf t}_0^-}({\bf u}_{-}^0,\pi_{-}^0)\!+\!{\mathcal P}\left(\gamma _{+}{\bf u}_{+}^0\right)={\bf 0} \mbox{ on } \partial \Omega ,
\end{equation}
lead to the equation
\begin{equation}
\label{q2}
\left(\frac{1}{2}(1+\mu ){\mathbb I}+(1-\mu ){\bf K}_{\partial \Omega }+{\mathcal V}_{\partial \Omega }{\mathcal P}\right)(\gamma _{+}{\bf u}_{+}^0)={\bf 0}.
\end{equation}
In view of Theorem \ref{single-layer-operator-Brinkman}(iii), the operator
$$\frac{1}{2}(1+\mu ){\mathbb I}+(1-\mu ){\bf K}_{\partial \Omega }+{\mathcal V}_{\partial \Omega }{\mathcal P}:H_\nu ^{s}(\partial \Omega ,\Lambda ^1TM)\to H_\nu ^{s}(\partial \Omega ,\Lambda ^1TM)$$
is invertible,
for any $s\in (0,1)$. Consequently, equation \eqref{q2} has only the trivial solution $\gamma _{+}{\bf u}_{+}^0={\bf 0}$. Then the well-posedness of the Dirichlet problem for the Stokes system ({see \cite[Theorem 5.6]{D-M} and \cite[Theorem 7.1]{M-T}}) as well as the assumption $\langle \widetilde E_+\pi _{+}^0,1\rangle _{\Omega _{+}}=0$ imply that ${\bf u}_{+}^0={\bf 0}$ and $\pi _{+}^0=0$ in $\Omega _{+}$. In addition, the first transmission condition in (\ref{q2-0}) implies that $\gamma _{-}{\bf u}_{-}^0={\bf 0} \mbox{ on } \partial \Omega $. By using again \cite[Theorem 5.1]{D-M} and the second condition in (\ref{q2-0}) we obtain ${\bf u}_{-}^0={\bf 0}$, $\pi _{-}^0=0$ in $\Omega _{-}$, i.e., the desired uniqueness result.
\end{proof}

\subsection{Well-posedness of the transmission problem \eqref{Poisson-transmission4-s}}
We construct a solution $\left(({\bf u}_{+},\pi _{+}),({\bf u}_{-},\pi _{-})\right)\in {\mathcal X}_{s}$ of the transmission problem \eqref{Poisson-transmission4-s} in the form
\begin{align}
\label{proof-1-a}
&{\bf u}_{+}={\mathcal N}_{\Omega _{+}}\tilde{\bf f}_{+}+{\bf v}_{+},\ \
\pi _{+}={\mathcal Q}_{\Omega _{+}}\tilde{\bf f}_{+}+p_{+}+c \mbox{ in } \Omega _{+},\\
\label{3.3-new-s-m}
&{\bf u}_{-}={\mathcal N}_{\Omega _{-}}\tilde{\bf f}_{-}+{\bf v}_{-},\ \
\pi _{-}={\mathcal Q}_{\Omega _{-}}\tilde{\bf f}_{-}+p_{-}+c \mbox{ in } \Omega _{-}.
\end{align}
Note that ${\mathcal N}_{{\Omega }_{\pm }}\tilde{\bf f}_{\pm }$ and ${\mathcal Q}_{{\Omega }_{\pm }}\tilde{\bf f}_{\pm }$ are the Newtonian velocity and pressure potentials for the Stokes system in ${\Omega }_{\pm }$, with a density $\tilde{\bf f}_{\pm }\in \widetilde{H}^{s-\frac{3}{2}}({\Omega }_{\pm },\Lambda ^1TM)$. {In order to satisfy the assumption $\langle \widetilde E_+\pi _{+},1\rangle _{\Omega _{+}}=0$, required by the definition of the space ${\mathcal X}_{s}$ (see \eqref{spaces-Poisson-transmission-MS} and \eqref{X1-sm}), we determine the pressure field $p_{+}$ in \eqref{proof-1-a} in the space $ H_*^{s-\frac{1}{2}}(\Omega _{+})$ and choose the constant $c\in {\mathbb R}$ such that}
\begin{align}
\label{c}
\big\langle \widetilde{E}_{+}({\mathcal Q}_{\Omega _{+}}\tilde{\bf f}_{+})+c,1\big\rangle _{\Omega _{+}}=0.
\end{align}

In view of the mapping properties \eqref{Newtonian-velocity} and \eqref{Newtonian-pressure} we have
\begin{align}
\label{short-notation-2}
{\mathcal N}_{{\Omega }_{\pm }}\tilde{\bf f}_{\pm }\in {H}^{s+\frac{1}{2}}(\Omega _{\pm },\Lambda ^1TM),\
{\mathcal Q}_{{\Omega }_{\pm }}\tilde{\bf f}_{\pm }\in {H}^{s-\frac{1}{2}}(\Omega _{\pm }).
\end{align}
In addition, by using the relations (\ref{Newtonian-s1}) we obtain that
\begin{align}
\label{proof-4}
{\bf L}{\mathcal N}_{\Omega _{\pm }}\tilde{\bf f}_{\pm }+ d{\mathcal Q}_{\Omega _{\pm }}\tilde{\bf f}_{\pm }={\tilde{\bf f}}_{\pm },\ \
\delta \left({\mathcal N}_{\Omega _{\pm }}\tilde{\bf f}_{\pm }\right)=0 \mbox{ in } \Omega _{\pm }.
\end{align}
Consequently, $\left(({\bf u}_{+},\pi _{+}),({\bf u}_{+},\pi _{+})\right)$ given by \eqref{proof-1-a}-\eqref{3.3-new-s-m} is a solution of the transmission problem of Poisson type (\ref{Poisson-transmission4-s}) in the space ${\mathcal X}_{s}$ if and only if $({\bf v}_{\pm },p_{\pm })$ satisfy the following transmission problem for the homogeneous Stokes system
\begin{equation}
\label{Poisson-transmission4-homogeneous}
\!\!\!\!\!\left\{\begin{array}{ll}
{\bf L}{\bf v}_{+}+dp_{+}={\bf 0},\ \delta {\bf v}_{+}=0
\ \mbox{ in }\ \Omega _{+}\\
{\bf L}{\bf v}_{-}+dp_{-}={\bf 0},\ \delta {\bf v}_{-}=0
\ \mbox{ in }\ \Omega _{-}\\
{\mu }\gamma _{+}{\bf v}_{+}-\gamma _{-}{\bf v}_{-}={\bf h}_0\in H^{s}_{\nu }(\partial \Omega ,\Lambda ^1TM)\\
{{\bf t}_0^+}({\bf v}_{+},p_{+})-{{\bf t}_0^-}({\bf v}_{-},p_{-})+{\mathcal P} \gamma _{+}{\bf v}_{+}={\bf r}_0\in H^{s-1}(\partial \Omega ,\Lambda ^1TM),\\
\end{array}\right.
\end{equation}
where
\begin{align}
\label{Poisson-transmission4-homogeneous1}
&{\bf h}_0:={\bf h}-\left({\mu }\gamma _{+}\left({\mathcal N}_{\Omega _{+}}\tilde{\bf f}_{+}\right)-\gamma _{-}\left({\mathcal N}_{\Omega _{-}}\tilde{\bf f}_{-}\right)\right),
\\
\label{Poisson-transmission4-homogeneous2}
&{\bf r}_0\!:=\!{\bf r}\!-\!\!\Big({{\bf t}_0^+}\big({\mathcal N}_{\Omega
_{+}}\tilde{\bf f}_{+},{\mathcal Q}_{\Omega _{+}}\tilde{\bf f}_{+}\big)\!-\!
{{\bf t}_0^-}\big({\mathcal N}_{\Omega _{-}}\tilde{\bf f}_{-},
{\mathcal Q}_{\Omega _{-}}\tilde{\bf f}_{-}\big){+\!{\mathcal P}\gamma _{+}\big({\mathcal N}_{\Omega _{+}}\tilde{\bf f}_{+}\big)}\Big).
\end{align}
Note that the left hand side of the last transmission condition in \eqref{Poisson-transmission4-homogeneous} is now expressed in terms of the canonical conormal derivatives (cf. \cite[Section 3.3]{Mikh}), due to the fact that the right hand sides of the first two equations in \eqref{Poisson-transmission4-homogeneous} are equal to zero.
Also note that the second relation in (\ref{Newtonian-s1}) combined with the divergence theorem and with the assumption ${\bf h}\in H^{s}_{\nu }(\partial \Omega ,\Lambda ^1TM)$ implies that ${\bf h}_0\in H^{s}_{\nu }(\partial \Omega ,\Lambda ^1TM)$. In addition, by (\ref{conormal-derivative-generalized-1}) and the assumption ${\mathcal P}\in L^{\infty }(\partial \Omega ,\Lambda ^1TM)$, we obtain that ${\bf r}_0\in H^{s-1}(\partial \Omega , \Lambda ^1TM)$.

Next we show the well-posedness of the transmission problem \eqref{Poisson-transmission4-homogeneous}.
\begin{lem}
\label{canonical-conormal-derivative}
{Let $M$ satisfy {Assumption} $\ref{H}$ and $\dim (M)\geq 2$}. Let $\Omega _{+}:=\Omega \subset M$ be a Lipschitz domain and $\Omega _{-}:=M\setminus \overline{\Omega }$. Let $s\in (0,1)$ and ${\mu >0}$ {be a given constant}. Then for all given data $\left(\tilde{\bf f}_{+},\tilde{\bf f}_{-},{\bf h},{\bf r}\right)\in {\mathcal Y}_{s}$, the transmission problem \eqref{Poisson-transmission4-homogeneous} has a unique solution $\left(({\bf v}_{+},p_{+}),({\bf v}_{-},p_{-})\right)\in {\mathcal X}_{s}$.
\end{lem}
\begin{proof}
We are looking for a solution of problem (\ref{Poisson-transmission4-homogeneous})
in the form
\begin{align}
\label{Poisson-transmission5}
{{\bf v}_{+}^0}={\bf W}_{\partial \Omega }{\bf \Phi }+{\bf V}_{\partial \Omega }{\boldsymbol \varphi},\ \ {p_{+}^0}={\mathcal P}_{\partial \Omega }{\bf \Phi }+{\mathcal Q}_{\partial \Omega }{\boldsymbol \varphi}\ \mbox{ in }\ \Omega _{+},\\
\label{Poisson-transmission5-S}
{{\bf v}_{-}^0}={\bf W}_{\partial \Omega }{\bf \Phi }+{\bf V}_{\partial \Omega }{\boldsymbol \varphi},\ \ {p_{-}^0}={\mathcal P}_{\partial \Omega }{\bf \Phi }+{\mathcal Q}_{\partial \Omega }{\boldsymbol \varphi}\ \mbox{ in }\ \Omega _{-},
\end{align}
where $({\bf \Phi },{\boldsymbol \varphi})\in H^{s}_{\nu }(\partial \Omega ,\Lambda ^1TM)\times H^{s-1}(\partial \Omega ,\Lambda ^1TM)$ are unknown densities. By (\ref{double-layer-Brinkman10}), these layer potential representations satisfy the Stokes system in $\Omega _{\pm }$. In addition, by the relations (\ref{single-layer-Brinkman8a}) and (\ref{double-layer-Brinkman12}), and by the first transmission condition in (\ref{Poisson-transmission4-homogeneous}) we obtain the equation
\begin{align}
\label{Phi}
\left(-\frac{1}{2}(1+\mu ){\mathbb I}+(\mu -1){\bf K}_{\partial\Omega }\right){\bf \Phi }+(\mu -1){\mathcal V}_{\partial\Omega }{\boldsymbol \varphi}={{\bf h}_{0}} \mbox{ on } \partial\Omega .
\end{align}
Next by using the relations (\ref{eq:dlB13}) and (\ref{double-layer-Brinkman14b}), and the second transmission condition in (\ref{Poisson-transmission4-homogeneous}), we obtain the following
equation
\begin{align}
\label{Poisson-transmission6}
\left({\mathbb I}+{\mathcal P}{\mathcal V}_{\partial \Omega }\right){\boldsymbol \varphi}+\left({\bf D}_{\partial\Omega }^{+}-{\bf D}_{\partial\Omega }^{-}+{\mathcal P}\left(-\frac{1}{2}{\mathbb I}+{\bf K}_{\partial \Omega }\right)\right){\bf \Phi }={{\bf r}_{0}} \mbox{ on } \partial\Omega ,
\end{align}
where the layer potential operator
\begin{align}
\label{2.2tporousmedia-new-s}
{\bf D}_{\partial \Omega }^{+}-{\bf D}_{\partial \Omega }^{-}:H^{s}(\partial \Omega ,\Lambda ^1TM)\to H^{s-1}(\partial \Omega ,\Lambda ^1TM)
\end{align}
is linear and compact due to the property that $\left({\bf D}_{\partial \Omega }^{+}-{\bf D}_{\partial \Omega }^{-}\right){\bf \Phi}\in {\mathbb R}\nu $, for any ${\bf \Phi }\in H_\nu ^s(\partial \Omega ,\Lambda ^1TM)$ (see also \cite[Theorem 4.17]{10-new2}). Note that in the Euclidean setting the operator in (\ref{2.2tporousmedia-new-s}) is just the null operator (see, e.g., \cite[(4.117)]{M-W}, \cite[Theorem 3.1]{K-L-W}).

The system of equations (\ref{Phi}) and (\ref{Poisson-transmission6}) reduces to the equation
\begin{equation}
\label{Phi-varphi}
{\mathcal U}({\bf \Phi },{\boldsymbol \varphi})^{\top }=\left({{\bf h}_{0}},{{\bf r}_{0}}\right)^{\top } \mbox{ in } {\mathbb X}_s,
\end{equation}
with the unknown $({\bf \Phi },{\boldsymbol \varphi})^{\top }\in {\mathbb X}_s$, where
\begin{align}
\label{space-S}
{\mathbb X}_s:=H_{\nu }^{s}(\partial \Omega ,\Lambda ^1TM)\times H^{s-1}(\partial \Omega ,\Lambda ^1TM),
\end{align}
${\mathcal U}:{\mathbb X}_s\to {\mathbb X}_s$ is the {matrix} operator
\begin{equation}
\label{3.7operator-new1}
{\mathcal U}:=\left(
\begin{array}{cc}
{\mathcal K}_{\mu ; \partial \Omega } \ & \ (\mu -1){\mathcal V}_{\partial\Omega }\\
{\bf D}_{\partial\Omega }^{+}-{\bf D}_{\partial\Omega }^{-}+{\mathcal P}\left(-\frac{1}{2}{\mathbb I}+{\bf K}_{\partial \Omega }\right) \ & \ {\mathbb I}+{\mathcal P}{\mathcal V}_{\partial \Omega }
\end{array}
\right),
\end{equation}
and ${\mathcal K}_{\mu ; \partial \Omega }:H_\nu ^{s}(\partial \Omega ,\Lambda ^1TM)\to H_\nu ^{s}(\partial \Omega ,\Lambda ^1TM)$ is the operator given by
\begin{align}
\label{new-oper}
{\mathcal K}_{\mu ; \partial \Omega }:=-\frac{1}{2}(1+\mu ){\mathbb I}+(\mu -1){\bf K}_{\partial\Omega }.
\end{align}
The operator (\ref{3.7operator-new1}) can be written as
\begin{equation}
\label{decomposition}
{\mathcal U}={\mathcal T}+{\mathcal C},
\end{equation}
where
\begin{equation}
\label{T}
\!\!\!{\mathcal T}\!\!=\!\!\left(\begin{array}{ccc}
{\mathcal K}_{\mu ; \partial \Omega } & (\mu -1){\mathcal V}_{\partial\Omega }\\
0 & {\mathbb I}
\end{array}
\right),\, {\mathcal C}\!\!=\!\!
\left(\begin{array}{ccc}
{\bf 0} &  {\bf 0}\\
{\bf D}_{\partial\Omega }^{+}-{\bf D}_{\partial\Omega }^{-}\!+\!{\mathcal P}\left(-\frac{1}{2}{\mathbb I}\!+\!{\bf K}_{\partial \Omega }\right) &  {\mathcal P}{\mathcal V}_{\partial \Omega }
\end{array}
\right).
\end{equation}
We now show that the operator ${\mathcal T}:{\mathbb X}_s\to {\mathbb X}_s$ is Fredholm with index zero for any $s\in (0,1)$ and $\mu >0$.
\begin{itemize}
\item[(i)]
If $\mu =1$ then ${\mathcal T}$ reduces to the isomorphism
\begin{equation}
\label{Fredholm-compact2-new} \left(
\begin{array}{ccc}
-{\mathbb I}\ & \ {\bf 0} \\
{\bf 0}\ & \ {\mathbb I} \\
\end{array}
\right).
\end{equation}
\item[(ii)]
If $\mu \in (0,+\infty )\setminus \{1\}$, then the operator ${\mathcal K}_{\mu ; \partial \Omega }$ given by (\ref{new-oper}) can be written as
\begin{align}
\label{new-oper-1}
{\mathcal K}_{\mu ; \partial \Omega }=(\mu -1)\left(\frac{1}{2}\frac{1+\mu }{1-\mu }{\mathbb I}+{\bf K}_{\partial\Omega }\right),
\end{align}
which is Fredholm with index zero whenever $\mu \in (0,1)$ due to the Fredholm and zero index property of the operators in \eqref{Fredholm3-new-S}. If $\mu \in (1,+\infty )$, this property is still valid. Indeed, we have $\mu ^{-1}\in (0,1)$, and {operator} (\ref{new-oper-1}) can be written in the equivalent form
\begin{align}
\label{new-oper-2}
{\mathcal K}_{\mu ; \partial \Omega }=(\mu -1)\left(-\frac{1}{2}\frac{1+\mu ^{-1}}{1-\mu ^{-1}}{\mathbb I}+{\bf K}_{\partial\Omega }\right).
\end{align}
Then, by using once again the Fredholm and zero index property of the operators in \eqref{Fredholm3-new-S}, we obtain the desired result.
\end{itemize}
Consequently, the operator ${\mathcal K}_{\mu ; \partial \Omega }:H_{\nu }^{s}(\partial \Omega ,\Lambda ^1TM)\to H_{\nu }^{s}(\partial \Omega ,\Lambda ^1TM)$ given by (\ref{new-oper-1}) is Fredholm with index zero, and then the operator ${\mathcal T}:{\mathbb X}_s\to {\mathbb X}_s$ defined in \eqref{T} is Fredholm with index zero as well, for any $\mu \in (0,+\infty )$ and $s\in (0,1)$. The operator ${\mathcal C}:{\mathbb X}_s\to {\mathbb X}_s$ is linear and compact due to the compactness of the operator in (\ref{2.2tporousmedia-new-s}) and of the operators
\begin{align}
&{\mathcal P}\left(-\frac{1}{2}{\mathbb I}+{\bf K}_{\partial \Omega }\right):H_\nu ^s(\partial \Omega ,\Lambda ^1TM)\to H^{s-1}(\partial \Omega ,\Lambda ^1TM),\nonumber\\
&{\mathcal P}{\mathcal V}_{\partial \Omega }:H^{s-1}(\partial \Omega ,\Lambda ^1TM)\to H^{s-1}(\partial \Omega ,\Lambda ^1TM).\nonumber
\end{align}
Hence the operator ${\mathcal U}:{\mathbb X}_s\to {\mathbb X}_s$ given by (\ref{3.7operator-new1}) is   Fredholm with index zero for any $s\in (0,1)$. We now show that ${\mathcal U}$ is also one-to-one, i.e.,
\begin{equation}
\label{kernel-Fredholm1-new}
{\rm{Ker}}\left\{{\mathcal U}:{\mathbb X}_s\to {\mathbb X}_s\right\}=\{0\}.
\end{equation}
To this {end}, we use the continuity of the embedding ${\mathbb X}_s\hookrightarrow {\mathbb X}_{\frac{1}{2}}$, which has dense range for any $s\in \left(\frac{1}{2},1\right)$, while the continuous embedding ${\mathbb X}_{\frac{1}{2}}\hookrightarrow {\mathbb X}_s$ has dense range for any $s\in \left(0,\frac{1}{2}\right)$. Then by Lemma \ref{Fredholm-kernel} we obtain the equality
\begin{equation}
\label{kernel-Fredholm1-new-S}
{\rm{Ker}}\left\{{\mathcal U}:{\mathbb X}_s\to {\mathbb X}_s\right\}={\rm{Ker}}\left\{{\mathcal U}:{\mathbb X}_{\frac{1}{2}}\to {\mathbb X}_{\frac{1}{2}}\right\},\ \forall \ s\in \left(0,1\right),
\end{equation}
which shows that property (\ref{kernel-Fredholm1-new}) holds if and only if
\begin{equation}
\label{kernel-Fredholm1-new-S1}
{\rm{Ker}}\left\{{\mathcal U}:{\mathbb X}_{\frac{1}{2}}\to {\mathbb X}_{\frac{1}{2}}\right\}=\{0\}.
\end{equation}

Let $\left(\Phi ^0,{\boldsymbol \varphi}^0\right)^{\top }\in {\rm{Ker}}\left\{{\mathcal U}:{\mathbb X}_{\frac{1}{2}}\to {\mathbb X}_{\frac{1}{2}}\right\}$, and consider the layer potentials
\begin{align}
\label{3.3-new-aaa}
{\bf u}^{0}={\bf W}_{\partial\Omega }{\bf \Phi }^0+{\bf V}_{\partial\Omega }{\boldsymbol \varphi}^0,\quad
\pi^{0}={\mathcal P}_{\partial\Omega }{\bf \Phi }^0+{\mathcal Q}_{\partial\Omega }
{\boldsymbol \varphi}^0 \mbox{ in } M\setminus \partial\Omega .
\end{align}
Then we have the following inclusions for the restrictions to $\Omega_\pm$,
\begin{align}
\label{4-43}
\left({\bf u}^{0}|_{\Omega _{\pm }},\pi^{0}|_{\Omega _{\pm }}\right)\in H_{\delta}^{1}({\Omega }_{\pm },\Lambda ^1TM)\times L^2({\Omega }_{\pm }),
\end{align}
and the following relations on $\partial \Omega $
\begin{align}
\label{q1}
&{\mu }\gamma _{+}\left({\bf u}^{0}|_{\Omega _{+}}\right)-\gamma _{-}\left({\bf u}^{0}|_{\Omega _{-}}\right)={\bf 0},\\
\label{q0}
&{{\bf t}_0^+}({\bf u}^{0}|_{\Omega _{+}},\pi^{0}|_{\Omega _{+}})-{{\bf t}_0^-}({\bf u}^{0}|_{\Omega _{-}},\pi^{0}|_{\Omega _{-}})+{\mathcal P}{\mathcal V}_{\partial \Omega }\left(\gamma _{+}\left({\bf u}^{0}|_{\Omega _{+}}\right)\right)={\bf 0}.
\end{align}
Therefore, the elements $\left({\bf u}^{0}|_{\Omega _{\pm }},\pi^{0}|_{\Omega _{\pm }}\right)\in H_{\delta}^{1}({\Omega }_{\pm },\Lambda ^1TM)\times L^2({\Omega }_{\pm })$ determine a solution of the homogeneous transmission problem associated to (\ref{Poisson-transmission4-s}) in $\left(H_\delta ^{s+\frac{1}{2}}(\Omega _{+},\Lambda ^1TM)\times H^{s-\frac{1}{2}}(\Omega _{+})\right)\times \left(H_\delta ^{s+\frac{1}{2}}(\Omega _{-},\Lambda ^1TM)\times H^{s-\frac{1}{2}}(\Omega _{-})\right)$. Then by using the Green formula \eqref{conormal-derivative-generalized-2} (with $V=0$ and $s=\frac{1}{2}$) in each of the domains $\Omega _{+}$ and $\Omega _{-}$ and the positivity condition \eqref{Poisson-transmission3} satisfied by ${\mathcal P}$ we deduce (as in \eqref{eq1-s}-\eqref{uniqueness-2}) that $\gamma _{+}\left({\bf u}^{0}|_{\Omega _{+}}\right)={\bf 0}$. The uniqueness result for the Dirichlet problem for the Stokes system (cf. \cite[Theorem 5.1]{D-M}) implies that ${\bf u}^0|_{\Omega _{\pm }}={\bf 0}$ and $\pi ^0|_{\Omega _{\pm }}=c_{\pm }\in {\mathbb R}$ in $\Omega _{\pm }$. By using once again the transmission condition \eqref{q0} and the relation $\gamma _{+}\left({\bf u}^{0}|_{\Omega _{+}}\right)={\bf 0}$, we obtain $c_{+}=c_{-}$. Moreover, formulas (\ref{single-layer-Brinkman8a}), (\ref{double-layer-Brinkman12}), (\ref{eq:dlB13}) and (\ref{double-layer-Brinkman14b}), applied to the layer potentials (\ref{3.3-new-aaa}), together with the conditions $\gamma _{+}\left({\bf u}^{0}|_{\Omega _{+}}\right)={\bf 0}$ and \eqref{q0} yield
\begin{align}
\label{3.3-new1f}
&{\bf 0}=\gamma _{+}\left({\bf u}^{0}|_{\Omega _{+}}\right)-\gamma _{-}\left({\bf u}^{0}|_{\Omega _{-}}\right)=-{\bf \Phi }^0\ \mbox{ on }\ \partial\Omega \\
\label{3.3-new2a}
&{\bf 0}={{\bf t}_0^+}({\bf u}^{0}|_{\Omega _{+}},\pi^{0}|_{\Omega _{+}})-{{\bf t}_0^-}({\bf u}^{0}|_{\Omega _{-}},\pi^{0}|_{\Omega _{-}})={\boldsymbol \varphi}^0\ \mbox{ on }\ \partial\Omega ,
\end{align}
and hence that ${\bf \Phi }^0={\bf 0}$ and ${\boldsymbol \varphi}^0={\bf 0}$.

Therefore, the Fredholm operator of index zero ${\mathcal U}:{\mathbb X}_{\frac{1}{2}}\to {\mathbb X}_{\frac{1}{2}}$ is one-to-one, and hence isomorphism. This result and (\ref{kernel-Fredholm1-new-S}) imply that ${\mathcal U}:{\mathbb X}_{s}\to {\mathbb X}_{s}$ is an isomorphism for any $s\in (0,1)$.
Then equation (\ref{Phi-varphi}) has a unique solution $({\bf \Phi },{\boldsymbol \varphi})^{\top }={\mathcal U}^{-1}\left({\bf h}_0,{\bf r}_0\right)^{\top}\in {\mathbb X}_s$, and the layer potentials (\ref{Poisson-transmission5}) and (\ref{Poisson-transmission5-S}) determine a solution $\left(({\bf v}_{+}^0,p_{+}^0),({\bf v}_{-}^0,p_{-}^0)\right)$ of the problem (\ref{Poisson-transmission4-homogeneous}) in the space
$\left(H_\delta ^{s+\frac{1}{2}}(\Omega _{+},\Lambda ^1TM)\!\times \!H^{s-\frac{1}{2}}(\Omega _{+})\right)\!\times \!\left(H_\delta ^{s+\frac{1}{2}}(\Omega _{-},\Lambda ^1TM)\!\times \! H^{s-\frac{1}{2}}(\Omega _{-})\right)$.
Then by Lemma~\ref{Poisson-transmission-Stokes-P} the element
\begin{align}
\label{solution-c}
\left(({\bf v}_{+},p_{+}),({\bf v}_{-},p_{-})\right)=\left(({\bf v}_{+}^0,p_{+}^0+c_0),({\bf v}_{-}^0,p_{-}^0+c_0)\right),
\end{align}
where the constant $c_0\in {\mathbb R}$ is chosen such that
\begin{align}
\label{c0}
\left\langle \widetilde E_{+}(p_{+}^0)+c_0,1\right\rangle _{\Omega _{+}}=0,
\end{align}
is the unique solution of problem (\ref{Poisson-transmission4-homogeneous}) in the space ${\mathcal X}_{s}$ given by \eqref{X1-sm}.
\end{proof}

Next we show the well-posedness of the Poisson problem of transmission type \eqref{Poisson-transmission4-s} in the space ${\mathcal X}_{s}$, $s\in (0,1)$.
\begin{thm}
\label{Poisson-transmission-S0} {Let $M$ satisfy {Assumption} $\ref{H}$ and $\dim (M)\geq 2$}. Let $\Omega _{+}:=\Omega \subset M$ be a Lipschitz domain. {Let $\Omega _{-}:=M\setminus \overline{\Omega }$ satisfy Assumption $\ref{connected-set}$}. Let $s\in (0,1)$ and ${\mu >0}$ {be a given constant}. Then for all given data $\left(\tilde{\bf f}_{+},\tilde{\bf f}_{-},{\bf h},{\bf r}\right)\in {\mathcal Y}_{s}$, the Poisson problem of transmission type \eqref{Poisson-transmission4-s} has a unique solution $\left(({\bf u}_{+},\pi _{+}),({\bf u}_{-},\pi _{-})\right)\in {\mathcal X}_{s}$ and there exists a linear and continuous operator
\begin{align}
\label{T-1}
{\mathcal S}_0:{\mathcal Y}_{s}\to {\mathcal X}_{s}
\end{align}
delivering this solution. Hence there exists a constant $C\equiv C(s,\mu ,{\mathcal P},\partial \Omega )>0$ such that
\begin{align}
\label{estimate1}
\left\|\left(({\bf u}_{+},\pi _{+}),({\bf u}_{-},\pi _{-})\right)\right\|_{{\mathcal X}_{s}}\leq C\left\|\left(\tilde{\bf f}_{+},\tilde{\bf f}_{-},{\bf h},{\bf r}\right)\right\|_{{\mathcal Y}_{s}}.
\end{align}
\end{thm}
\begin{proof}
The unique solution $\left(({\bf v}_{+},p_{+}),({\bf v}_{-},p_{-})\right)\in {\mathcal X}_{s}$ of the transmission problem (\ref{Poisson-transmission4-homogeneous}), together with relations \eqref{proof-1-a}, \eqref{3.3-new-s-m} {and \eqref{c}}, determine a solution $\left(({\bf u}_{+},\pi _{+}),({\bf u}_{-},\pi _{-})\right)$ of the Poisson problem of transmission type \eqref{Poisson-transmission4-s} in the space ${\mathcal X}_{s}$. In view of Lemma~\ref{Poisson-transmission-Stokes-P}, this solution is unique, whenever $s\in (0,1)$, and depends continuously on the given data, i.e., it satisfies an inequality of type \eqref{estimate1}, due to the boundedness of the involved Newtonian and boundary layer potentials and of the isomorphism ${\mathcal U}^{-1}$. Consequently, the operator delivering this solution, ${\mathcal S}_{0}:{\mathcal Y}_{s}\to {\mathcal X}_{s}$, is linear and continuous for any $s\in (0,1)$, as asserted.
\end{proof}

\subsection{Well-posedness of the transmission problem \eqref{Poisson-transmission4} with $V\neq 0$}

Next, we obtain the main result of this section, i.e., the well-posedness of transmission problem for the Stokes and Brinkman systems \eqref{Poisson-transmission4}, with $V\neq 0$.
\begin{thm}
\label{Poisson-transmission-Stokes} {Let $M$ satisfy {Assumption} $\ref{H}$ and $\dim (M)\geq 2$}. Let $\Omega _{+}:=\Omega \subset M$ be a Lipschitz domain. {Let $\Omega _{-}:=M\setminus \overline{\Omega }$ satisfy Assumption $\ref{connected-set}$.} Assume that {$V\in L^{\infty }(M,\Lambda^1TM\otimes \Lambda ^1TM)$} is a {symmetric} tensor field, which satisfies the positivity condition \eqref{V-positive}, and that ${\mathcal P}\in L^{\infty }(\partial \Omega ,\Lambda ^1TM\otimes \Lambda ^1TM)$ is a symmetric tensor field which satisfies condition \eqref{Poisson-transmission3}. Let $s\in (0,1)$ and ${\mu >0}$ {be a constant}. Then for all given data $\left(\tilde{\bf f}_{+},\tilde{\bf f}_{-},{\bf h},{\bf r}\right)\!\in \!{\mathcal Y}_{s}$, the Poisson problem of transmission type \eqref{Poisson-transmission4} has a unique solution $\left(({\bf u}_{+},\pi _{+}),({\bf u}_{-},\pi _{-})\right)\in {\mathcal X}_{s}$ and there exists a linear and continuous operator
\begin{align}
\label{T-1}
{\mathcal S}_{V}:{\mathcal Y}_{s}\to{\mathcal X}_{s}
\end{align}
delivering this solution. Hence there is a constant $C\equiv C(s,\mu ,V,{\mathcal P},\partial \Omega )>0$ such that
\begin{align}
\label{estimate1-SM}
\left\|\left(({\bf u}_{+},\pi _{+}),({\bf u}_{-},\pi _{-})\right)\right\|_{{\mathcal X}_{s}}\leq C\left\|\left(\tilde{\bf f}_{+},\tilde{\bf f}_{-},{\bf h},{\bf r}\right)\right\|_{{\mathcal Y}_{s}}.
\end{align}
\end{thm}
\begin{proof}
In view of Theorem \ref{Poisson-transmission-S0} there exists a linear and continuous operator ${\mathcal S}_0:{\mathcal Y}_{s}\to {\mathcal X}_{s}$ such that $\left(\left({\bf w}_{+},q_{+}\right),\left({\bf w}_{-},q_{-}\right)\right)={\mathcal S}_0\left({\widetilde{\bf F}}_{+},{\widetilde{\bf F}}_{-},{\bf H},{\bf R}\right)$ is the unique solution of the problem
\begin{equation}
\label{generalized-linear-Robin-N-S-D-B-F-new-SM}
\!\!\!\!\!\left\{\begin{array}{ll}
{\bf L}{\bf w}_{+}+dq_{+}={\widetilde{\bf F}}_{+}|_{\Omega _{+}},\ {\delta {\bf w}_{+}=0}
\ \mbox{ in }\ \Omega _{+},\\
{\bf L}{\bf w}_{-}+dq_{-}={\widetilde{\bf F}}_{-}|_{\Omega _{-}},\ {\delta {\bf w}_{-}=0}
\ \mbox{ in }\ \Omega _{-},\\
{\mu }\gamma _{+}{\bf w}_{+}-\gamma _{-}{\bf w}_{-}={\bf H}\ \mbox{ on }\ \partial \Omega ,\\
{{\bf t}_0^+}\left({\bf w}_{+},q_{+};{\widetilde{\bf F}}_{+}\right)
-{{\bf t}_0^-}\left({\bf w}_{-},q_{-};{\widetilde{\bf F}}_{-}\right)+{\mathcal P} \gamma _{+}{\bf w}_{+}={\bf R} \mbox{ on }\ \partial \Omega ,
\end{array}
\right.
\end{equation}
for all $\left({\widetilde{\bf F}}_{+},{\widetilde{\bf F}}_{-},{\bf H},{\bf R}\right)\in {\mathcal Y}_{s}$. Then {(see Remark \ref{R2.4})} we can rewrite the problem \eqref{Poisson-transmission4} as
{\begin{align}
\big(\big({\bf v}_{+},p_{+}\big),\big({\bf v}_{-},p_{-}\big)\big)
={\mathcal S}_0\left(\tilde{\bf f}_{+}\!-\!{\widetilde E}_{+}\left(V{\bf v}_{+}\right),\tilde{\bf f}_{-},{\bf h},{\bf r}\right),
\end{align}}
or, equivalently, as
\begin{align}
\label{equiv}
\big(\big({\bf v}_{+},p_{+}\big),\big({\bf v}_{-},p_{-}\big)\big)+{\mathcal S}_0\Big({\widetilde E}_{+}\left(V{\bf v}_{+}\right),{\bf 0},{\bf 0},{\bf 0}\Big)={\mathcal S}_0\left(\tilde{\bf f}_{+},\tilde{\bf f}_{-},{\bf h},{\bf r}\right),
\end{align}
where ${\widetilde E}_{+}$ is the operator of extension of vector fields (one forms) defined in $\Omega_+$ by zero on $M\setminus \Omega_{+}$.
Since ${\mathcal S}_0:{\mathcal Y}_{s}\to {\mathcal X}_{s}$ is linear and continuous and the linear map from ${\mathcal X}_{s}$ to ${\mathcal Y}_{s}$, which takes $\big(\big({\bf v}_{+},p_{+}\big),\big({\bf v}_{-},p_{-}\big)\big)$ to $\left({\widetilde E}_{+}\left(V{\bf v}_{+}\right),{\bf 0},{\bf 0},{\bf 0}\right)$, is compact, {due to the continuity of the inclusions \begin{align}
L^{\infty }(\Omega _{+},\Lambda^1TM\otimes \Lambda ^1TM)\cdot H^{s+\frac{1}{2}}(\Omega _{+},\Lambda^1TM)&\hookrightarrow L^2(\Omega _{+},\Lambda^1TM)\nonumber\\
&\hookrightarrow \widetilde{H}^{s-\frac{3}{2}}(\Omega _{+},\Lambda ^1TM),\nonumber
\end{align}
the last of them being compact}, we deduce that the left hand side of \eqref{equiv} defines a Fredholm operator of index zero for any $s\in (0,1)$, denoted by
\begin{align}
\label{T-V}
{\mathcal A}_{V}:{\mathcal X}_{s}\to {\mathcal X}_{s}.
\end{align}
Then by the density of the embedding ${\mathcal X}_{s}\hookrightarrow {\mathcal X}_{\frac{1}{2};\delta }$ whenever $s\in (0,\frac{1}{2})$, with the converse embedding in the case $s\in (\frac{1}{2},1)$, and by Lemma \ref{Fredholm-kernel} we obtain
\begin{align}
\label{s-b-1}
{\rm{Ker}}\left\{{\mathcal A}_{V}:{\mathcal X}_{s}\to {\mathcal X}_{s}\right\}=
{\rm{Ker}}\left\{{\mathcal A}_{V}:{\mathcal X}_{\frac{1}{2};\delta }\to {\mathcal X}_{\frac{1}{2};\delta }\right\},\ \forall \ s\in (0,1).
\end{align}
We will show that the transmission problem \eqref{Poisson-transmission4} has at most one solution in the space ${\mathcal X}_{\frac{1}{2};\delta }$ (i.e., for $s=\frac{1}{2}$) and hence that ${\rm{Ker}}\left\{{\mathcal A}_{V}:{\mathcal X}_{\frac{1}{2};\delta }\to {\mathcal X}_{\frac{1}{2};\delta }\right\}=\{0\}$. Then by \eqref{s-b-1}, ${\rm{Ker}}\left\{{\mathcal A}_{V}:{\mathcal X}_{s}\to {\mathcal X}_{s}\right\}=\{0\}$ for any $s\in (0,1)$. Equivalently, the Fredholm operator with index zero ${\mathcal A}_{V}:{\mathcal X}_{s}\to {\mathcal X}_{s}$ is an isomorphism for any $s\in (0,1)$. Then it follows that equation \eqref{equiv} has a unique solution in ${\mathcal X}_{s}$ and that accordingly problem \eqref{Poisson-transmission4} has a unique solution in ${\mathcal X}_{s}$ for any $s\in (0,1)$.

Let us now show that problem \eqref{Poisson-transmission4} has at most one solution for $s=\frac{1}{2}$.
Indeed, if we assume that $\left(({\bf u}_{+}^0,\pi _{+}^0),({\bf u}_{-}^0,\pi _{-}^0)\right)\in {\mathcal X}_{\frac{1}{2};\delta }$ is a solution of the homogeneous problem associated to \eqref{Poisson-transmission4}, then by applying the {Green identity} (\ref{conormal-derivative-generalized-2}) in each of the domains $\Omega _{+}$ and $\Omega _{-}$, we obtain the relations
\begin{align}
\label{eq1-s}
2\langle {\rm{Def}}\ {\bf u}_{+}^0,{\rm{Def}}\ {\bf u}_{+}^0\rangle
_{\Omega _{+}}+\langle V{\bf u}_{+}^0,{\bf u}_{+}^0\rangle _{\Omega _{+}}&=
\langle {{\bf t}_V^+}({\bf u}_{+}^0,\pi _{+}^0),\gamma _{+}{\bf u}_{+}^0\rangle _{\partial \Omega },
\end{align}
\begin{align}
\label{eq2-s}
2\langle {\rm{Def}}\ {\bf u}_{-}^0,{\rm{Def}}\ {\bf u}_{-}^0\rangle _{\Omega _{-}}&= -\langle {{\bf t}_0^-}({\bf u}_{-}^0,\pi _{-}^0),\gamma _{-}{\bf u}_{-}^0\rangle _{\partial \Omega }.
\end{align}
By using the transmission conditions satisfied by $({\bf u}_{+}^0,\pi _{+}^0)$ and $({\bf u}_{-}^0,\pi _{-}^0)$,
\begin{align}
\label{trans-SB}
\mu \gamma _{+}{\bf u}_{+}^0=\gamma _{-}{\bf u}_{-}^0,\ {{\bf t}_V^+}\left({\bf u}_{+}^0,\pi _{+}^0\right)
-{{\bf t}_0^-}\left({\bf u}_{-}^0,\pi _{-}^0\right)=-{\mathcal P}\gamma _{+}{\bf u}_{+}^0 \mbox{ on } \partial \Omega ,
\end{align}
we obtain that
\begin{align}
\label{uniqueness-2}
&2\mu \langle {\rm{Def}}\ {\bf u}_{+}^0,{\rm{Def}}\ {\bf u}_{+}^0\rangle _{\Omega _{+}}
+\mu \langle V{\bf u}_{+}^0,{\bf u}_{+}^0\rangle _{\Omega _{+}}+2\langle {\rm{Def}}\ {\bf u}_{-}^0,{\rm{Def}}\ {\bf u}_{-}^0\rangle _{\Omega _{-}}\nonumber\\
&=-\mu \langle {\mathcal P}\gamma _{+}{\bf u}_{+}^0,\gamma _{+}{\bf u}_{+}^0\rangle _{\partial \Omega }\leq -\mu C_{\mathcal P}\|\gamma _{+}{\bf u}_{+}^0\|_{L^2(\partial \Omega ,\Lambda ^1TM)}^2,
\end{align}
where the last inequality follows from the strong positivity condition (\ref{Poisson-transmission3}) satisfied by ${\mathcal P}$. However, the left hand side of the above equality is non-negative due to the positivity condition \eqref{V-positive} satisfied by $V$. {Thus, each side of \eqref{uniqueness-2} vanishes, and, in particular, we obtain $\gamma _{+}{\bf u}_{+}^0=\gamma _{-}{\bf u}_{-}^0={\bf 0}$ on $\partial \Omega ,$ i.e., ${\bf u}_{\pm }^0\in H_0^1(\Omega _{\pm },\Lambda ^1TM)$, where
$H_0^1(\Omega ,\Lambda ^1TM)$ is the closure of ${\mathcal D}(\Omega ,\Lambda^1TM)$ in the norm of $H^1(\Omega ,\Lambda ^1TM)$ (see \cite[Theorem 3.33]{Lean}).
Moreover, by \eqref{uniqueness-2}, ${\rm{Def}}\ {\bf u}_{\pm }^0=0$ in $\Omega _{\pm}$. Then the injectivity of the operator
\begin{align}
\label{def-one-to-one}
{\rm{Def}}:H_0^1(\Omega _{\pm },\Lambda ^1TM)\to L^2(\Omega _{\pm },S^2T^*M)
\end{align}
(see \cite[(6.17)]{D-M}) implies that ${\bf u}_{\pm }^0={\bf 0}$ in $\Omega _{\pm}$. Consequently, we find that
\begin{equation}
\label{mu-1}
{\bf u}_{\pm }^0={\bf 0},\ \pi _{\pm }^0=c_{\pm }\in {\mathbb R}\ \mbox{ in }\ \Omega _{\pm},
\end{equation}
where the last relations follow from the Brinkman system in $\Omega _{+}$ and the Stokes system in $\Omega _{-}$, respectively. However, the second transmission condition in \eqref{trans-SB} implies the relation ${\bf t}_{V}^{+}({\bf u}_{+}^0,\pi _{+}^0)={{\bf t}_0^-\left({\bf u}_{-}^0,\pi _{-}^0\right)}$ on $\partial \Omega $, which, together with
the condition $\langle \widetilde E_+\pi _{+}^0,1\rangle _{\Omega _{+}}=0$ {(see \eqref{X1-sm})}, shows that $c_{\pm }=0$.}
Therefore,
\begin{align}
\label{uniq-1-B}
{\bf u}_{\pm }^0={\bf 0},\ \pi _{\pm }^0=0\  \mbox{ in }\ \Omega _{\pm },
\end{align}
i.e., problem \eqref{Poisson-transmission4} has at most one solution in the space ${\mathcal X}_{\frac{1}{2};\delta }$ (i.e., in the case $s=\frac{1}{2}$), as desired. Consequently, problem \eqref{Poisson-transmission4} has a unique solution $\left(({\bf u}_{+},\pi _{+}),({\bf u}_{-},\pi _{-})\right)\in {\mathcal X}_{s}$ for any $s\in (0,1)$. Moreover, by \eqref{equiv} and \eqref{T-V}, $${\left(({\bf u}_{+},\pi _{+}),({\bf u}_{-},\pi _{-})\right)={\mathcal S}_V\left(\tilde{\bf f}_{+},\tilde{\bf f}_{-},{\bf h},{\bf r}\right),}$$ where the operator
\begin{align}
\label{S-V}
{{\mathcal S}_V:{\mathcal Y}_{s}\to {\mathcal X}_{s},\ \ {\mathcal S}_V={\mathcal A}_V^{-1}\circ {\mathcal S}_0}
\end{align}
delivering this solution, is linear and continuous, due to linearity and continuity of the maps ${\mathcal S}_0:{\mathcal Y}_{s}\!\to \!{\mathcal X}_{s}$ (see Theorem \ref{Poisson-transmission-S0}) and ${\mathcal A}_V^{-1}\!:\!{\mathcal X}_{s}\!\to \!{\mathcal X}_{s}$.
\end{proof}

\section{Transmission problems for the Navier-Stokes and generalized Darcy-Forchheimer-Brinkman systems}
\label{nonlin}
\subsection{Problem setting and preliminary lemma}
Let $s\in (0,1)$.
Let {$\mu >0$} and $k,\beta \in {\mathbb R}$ be nonzero constants. Let {$V\in L^{\infty }(M,\Lambda^1TM\otimes \Lambda ^1TM)$} be a {symmetric} tensor field satisfying the positivity condition \eqref{V-positive} and ${\mathcal P}\in L^{\infty }(\partial \Omega ,\Lambda ^1TM\otimes \Lambda ^1TM)$ be a symmetric tensor field satisfying the positivity condition (\ref{Poisson-transmission3}). Next we consider the following transmi\-ssion problem for the nonlinear {incompressible} Navier-Stokes and generalized Darcy-Forchheimer-Brinkman systems in two complementary Lipschitz domains $\Omega _{\pm }$ of a compact Riemannian manifold $M$, with $\dim (M)\in \{2,3\}$:
{\begin{align}
\label{Poisson-transmission-zero-order1}
\!\!\!\!\!\left\{\begin{array}{ll}
{\bf L}{\bf u}_{+}\!+\!V{\bf u}_{+}\!+\!d\pi _{+}\!=\!\tilde{\bf f}_{+}|_{\Omega _+}\!-\!\left(k|{\bf u}_{+}|{\bf u}_{+}\!+\!\beta \nabla _{{\bf u}_{+}}{\bf u}_{+}\right),\
{\delta {\bf u}_{+}\!=\!0} \mbox{ in } \Omega _{+},\\
{\bf L}{\bf u}_{-}+d\pi _{-}=\tilde{\bf f}_{-}|_{\Omega _{-}}-\nabla _{{\bf u}_{-}}{\bf u}_{-} \mbox{ in } \Omega _{-},\
{\delta {\bf u}_{-}=0} \mbox{ in } \Omega _{-},\\
{\mu }\gamma _{+}{\bf u}_{+}-\gamma _{-}{\bf u}_{-}={\bf h} \mbox{ on } \partial \Omega ,\\
{\bf t}_{V}^{+}\left({\bf u}_{+},\pi _{+};\tilde{\bf f}_{+}-{\widetilde{\mathcal I}_{k;\beta ;{\Omega_+}}{\bf u}_+}\right)
-{\bf t}_{0}^{-}\left({\bf u}_{-},\pi _{-};\tilde{\bf f}_{-}-{\widetilde{\mathcal I}_{0;1 ;{\Omega_-}}{\bf u}_-}\right)\\
\hspace{7em}+{\mathcal P}\gamma _{+}{\bf u}_{+}
={\bf r} \mbox{ on } \partial \Omega ,
\end{array}
\right.
\end{align}
where
\begin{align}
\label{nonlin-1}
\widetilde{\mathcal I}_{k_\pm;\beta_\pm ;{\Omega_\pm}}{\bf u}_\pm:
=k_\pm({\widetilde E}^0_\pm|{\bf u}_\pm|)E^{s+\frac{1}{2}}_\pm{\bf u}_\pm
+\beta_\pm ({\widetilde E}^{s-\frac{1}{2}}_\pm\nabla _{{\bf u}_\pm})E^{s+1}_\pm{\bf u}_\pm,
\end{align}
for ${\bf u}_\pm\in H_\delta ^{s+\frac{1}{2}}(\Omega _{\pm},\Lambda ^1TM)$, $\Omega_\pm \subset M$ and the constants $k_\pm,\beta_\pm \in {\mathbb R}$, ${\widetilde E}^{t}_\pm:H^{t}(\Omega_\pm)\to \widetilde H^{t}(\Omega_\pm)$ is the unique linear bounded extension operator for $-1/2<t<1/2$, cf. \cite[Theorem 2.16]{Mikh}, while ${E}^{t}_\pm:H^{t}(\Omega_\pm)\to H^{t}(M)$, $t\in\mathbb R$, is a (non-unique) linear bounded extension operator. {In view of \eqref{locally}}, the expression $({\widetilde E}_\pm^{s-\frac{1}{2}}\nabla _{{\bf u}_\pm}){E^{s+\frac{1}{2}}_\pm}{\bf u}_\pm$ is understood as
\begin{align}
\label{expression-E}
\left(({\widetilde E}_\pm^{s-\frac{1}{2}}\nabla _{{\bf u}_\pm})E^{s+\frac{1}{2}}_\pm{\bf u}_\pm\right)^\ell =&\left({\widetilde E}_\pm^{s-\frac{1}{2}}\partial _ju_\pm^\ell\right)E_\pm^{s+\frac{1}{2}}u_\pm^j\nonumber\\
&+\Gamma _{rj}^\ell\left({\widetilde E}_\pm^{s-\frac{1}{2}}u_\pm^j\right) E_\pm^{s+\frac{1}{2}}u_\pm^r.
\end{align}
For the {restrictions} of $\widetilde{\mathcal I}_{k_\pm;\beta_\pm ;{\Omega_\pm}}{\bf u}_\pm$, we evidently have,
$$
r_{\Omega_\pm}\widetilde{\mathcal I}_{k_\pm;\beta_\pm ;{\Omega_\pm}}{\bf u}_\pm=k_\pm|{\bf u}_\pm|{\bf u}_\pm+\beta_\pm \nabla _{{\bf u}_\pm}{\bf u}_\pm .
$$

We assume that the given data in (\ref{Poisson-transmission-zero-order1}), $\left(\tilde{\bf f}_{+},\tilde{\bf f}_{-},{\bf h},{\bf r}\right)$, belong to the space ${\mathcal Y}_{s}$ defined in (\ref{Y1-sm}), and that they are sufficiently small in a sense that will be described below. Then we show the existence and uniqueness of the solution $\left(({\bf u}_{+},\pi _{+}),({\bf u}_{-},\pi _{-})\right)\in {\mathcal X}_{s}$ of the transmission problem (\ref{Poisson-transmission-zero-order1}), where ${\mathcal X}_{s}$ is the space given by (\ref{X1-sm}). The proof of the existence and uniqueness result is mainly based on the well-posedness result established in Theorem \ref{Poisson-transmission-Stokes} and on a fixed point theorem.

We note {that}
problem \eqref{Poisson-transmission-zero-order1} with $V\equiv \alpha {\mathbb I}$ and $\alpha,k,\beta $ positive constants describes the flow of a viscous incompressible fluid through a porous medium, and the coefficients $\alpha ,k,\beta $ are determined by the physical properties of such a medium (see, e.g., \cite{Ni-Be} for further details).

{First we show the following result that plays a main role in the proof of the well-posedness of the problem \eqref{Poisson-transmission-zero-order1} in the case $s\in (0,1)$. {A slightly different approach to the} particular case {$s\in [\frac{1}{2}, 1)$ is discussed in Appendix \ref{A3}} (see also \cite[Lemma 7.5]{D-M} in the case of general $L^p$-Sobolev spaces,
and \cite[Lemma 5.1]{K-L-M-W} in the Euclidian setting).
\begin{lem}
\label{embeddings}
{Let $M$ be a compact Riemannian manifold with $\dim (M)\in \{2,3\}$}. Let $\Omega \subset M$ be a Lipschitz domain, {$s\in (0,1)$} and $k,\beta \in {\mathbb R}$ be given nonzero constants, and
{\begin{align}
\label{nonlin-2}
{\widetilde{\mathcal I}}_{k;\beta ;{\Omega }}({\bf v}):
=k({\widetilde E}^0|{\bf v}|)E^{s+\frac{1}{2}}{\bf v}
+\beta ({\widetilde E}^{s-\frac{1}{2}}\nabla _{\bf v})E^{s+\frac{1}{2}}{\bf v}.
\end{align}}
Then the operator
\begin{align}
\label{NS-s}
{\widetilde{\mathcal I}}_{k;\beta ;{\Omega }}:H^{s+\frac{1}{2}}({\Omega },\Lambda ^1TM)\to
\widetilde{H}^{s-\frac{3}{2}}(\Omega ,\Lambda ^1TM)
\end{align}
is a nonlinear continuous, positively homogeneous {of} order $2$, and bounded, in the sense that there exists a constant $C_1\equiv C_1({\Omega },k,\beta )>0$ such that
\begin{align}
\label{NS-s2}
&\|{\widetilde{\mathcal I}}_{k;\beta ;{\Omega }}({\bf u})\|_{\widetilde{H}^{s-\frac{3}{2}}(\Omega ,\Lambda ^1TM)}\leq C_{1}\|{\bf u}\|_{{H^{s+\frac{1}{2}}({\Omega },\Lambda ^1TM)}}^2.
\end{align}
In addition, the following inequality holds
\begin{align}
\label{NS-s3}
&\|{\widetilde{\mathcal I}}_{k;\beta ;{\Omega }}({\bf v})-{\widetilde{\mathcal I}}_{k;\beta ;{\Omega }}({\bf w})\|_{\widetilde{H}^{s-\frac{3}{2}}(\Omega ,\Lambda ^1TM)}\\
&\leq {C_1}\left(\|{\bf v}\|_{{H^{s+\frac{1}{2}}({\Omega },\Lambda ^1TM)}}+\|{\bf w}\|_{{H}^{s+\frac{1}{2}}({\Omega },\Lambda ^1TM)}\right)\|{\bf v}-{\bf w}\|_{{H^{s+\frac{1}{2}}({\Omega },\Lambda ^1TM)}},\nonumber
\end{align}
for all ${\bf v},{\bf w}\in H^{s+\frac{1}{2}}({\Omega },\Lambda ^1TM)$,
{and the operator $\widetilde{\mathcal I}_{k;\beta ;{\Omega }}$ does not depend on the chosen extension operator $E^{s+\frac{1}{2}}$ in definition \eqref{nonlin-2}}.
\end{lem}
\begin{proof}
{Let $s\in (0,1)$. The following pointwise multiplication result for Sobolev spaces holds
\begin{align}
\label{t}
H^{s+\frac{1}{2}}(\Omega )\cdot H^{\frac{3}{2}-s}(\Omega )\hookrightarrow H^{t}(\Omega )
\end{align}
for any $t\in [0,1/2)$,
in the sense that if $f\in H^{s+\frac{1}{2}}(\Omega )$ and $h\in H^{\frac{3}{2}-s}(\Omega )$, then
$fh\in H^{t}(\Omega )$, and there is a constant $c\equiv c(\Omega ,s)>0$ such that
$$\|fh\|_{H^{t}(\Omega )}\leq c\|f\|_{H^{s+\frac{1}{2}}(\Omega )}\|h\|_{H^{\frac{3}{2}-s}(\Omega )}$$
(cf., e.g., {\cite[Section 4.4]{Runst-Sickel}}, \cite[Theorem 7.3]{Be}, \cite[Lemma 28]{Holst} and \cite[(7.29)]{D-M}).

Since for any $s\in (0,1)$ there exists $t\in [0,1/2)$ such that $1/2-s<t$, and thus $H^{t}(\Omega )\hookrightarrow H^{\frac{1}{2}-s}(\Omega )$, \eqref{t} implies that
\begin{align}
\label{inclusion-NS}
H^{s+\frac{1}{2}}(\Omega )\cdot H^{\frac{3}{2}-s}(\Omega )\hookrightarrow H^{\frac{1}{2}-s}(\Omega )
\end{align}
and there exists a constant $c\equiv c(\Omega ,s)>0$ such that $$\|fh\|_{H^{\frac{1}{2}-s}(\Omega )}\leq c\|f\|_{H^{s+\frac{1}{2}}(\Omega )}\|h\|_{H^{\frac{3}{2}-s}(\Omega )}.$$

By} \cite[Theorem 8.1]{Be} and \cite[Lemma 28]{Holst}), we obtain that the pointwise multiplication of functions extends uniquely to a continuous bilinear map
\begin{align}
\label{inclusion-NS-1}
H^{s+\frac{1}{2}}(\Omega )\cdot H^{s-\frac{1}{2}}(\Omega )\hookrightarrow {H}^{s-\frac{3}{2}}(\Omega ).
\end{align}

In addition, we {prove} the inclusion
\begin{align}
\label{inclusion-NS-2}
H^{s+\frac{1}{2}}(\Omega)\cdot \widetilde{H}^{s-\frac{1}{2}}(\Omega )\hookrightarrow \widetilde{H}^{s-\frac{3}{2}}(\Omega ),
\end{align}
i.e., ${(E^{s+\frac{1}{2}}u)}\widetilde{v}\in \widetilde{H}^{s-\frac{3}{2}}(\Omega )$ for all ${u}\in H^{s+\frac{1}{2}}(\Omega)$ and $\widetilde{v}\in \widetilde{H}^{s-\frac{1}{2}}(\Omega )$,
{and} there is a constant $\tilde{c}=\tilde{c}(\Omega ,s)>0$, such that
\begin{align}
\label{Newtonian-D-B-F-Robin2-SM}
\|{(E^{s+\frac{1}{2}}u)}\widetilde{v}\|_{\widetilde{H}^{s-\frac{3}{2}}(\Omega )}\leq
\tilde{c}\|{u}\|_{H^{s+\frac{1}{2}}(\Omega )}\|\widetilde{v}\|_{\widetilde{H}^{s-\frac{1}{2}}(\Omega )}.
\end{align}
Let us first prove \eqref{Newtonian-D-B-F-Robin2-SM} in Euclidean setting, by using {the arguments similar to those in} \cite[Theorem 8.1]{Be}.
{We}
first prove \eqref{Newtonian-D-B-F-Robin2-SM} for $\widetilde{v}\in \widetilde{H}^{s-\frac{1}{2}}(\Omega )$ and $u\in {\mathcal D}(\mathbb R^n)$.
{We have,}
\begin{align}
\label{conv-3}
\|u\widetilde{v}\|_{\widetilde{H}^{s-\frac{3}{2}}(\Omega )}=
\sup _{\Psi \in \mathcal D(\overline{\Omega})}\frac{\big|\langle u\widetilde{v},\Psi \rangle _{\widetilde{H}^{s-\frac{3}{2}}(\Omega )\times H^{\frac{3}{2}-s}(\Omega )}\big|}{\|\Psi \|_{H^{\frac{3}{2}-s}(\Omega )}}.
\end{align}

By using the pointwise multiplication result \eqref{inclusion-NS} and that\\
{${r_\Omega}u\in {\mathcal D}(\overline{\Omega})\subset H^{s+\frac{1}{2}}(\Omega )$}, we obtain
\begin{align*}
\big|\langle u\widetilde{v},\Psi \rangle _{\widetilde{H}^{s-\frac{3}{2}}(\Omega )\times H^{\frac{3}{2}-s}(\Omega )}\big|&=\big|\langle \widetilde{v},u\Psi \rangle _{\widetilde{H}^{s-\frac{1}{2}}(\Omega )\times H^{\frac{1}{2}-s}(\Omega )}\big|\nonumber\\
&\leq\|\widetilde{v}\|_{\widetilde{H}^{s-\frac{1}{2}}(\Omega )}\|u\Psi \|_{H^{\frac{1}{2}-s}(\Omega )}\nonumber\\
&\leq {\tilde c}\|\widetilde{v}\|_{\widetilde{H}^{s-\frac{1}{2}}(\Omega )}\|u\|_{H^{s+\frac{1}{2}}(\Omega )}\|\Psi \|_{H^{\frac{3}{2}-s}(\Omega )}.
\end{align*}
{Hence,
\begin{align}
\label{5.14}
\|u\widetilde{v}\|_{\widetilde{H}^{s-\frac{3}{2}}(\Omega )}\leq
\tilde c\|u\|_{H^{s+\frac{1}{2}}(\Omega )}\|\widetilde{v}\|_{\widetilde{H}^{s-\frac{1}{2}}(\Omega )}.
\end{align}
i.e.,} inequality \eqref{Newtonian-D-B-F-Robin2-SM} holds for $\widetilde{v}\in \widetilde{H}^{s-\frac{1}{2}}(\Omega )$ and $u\in {\mathcal D}({\mathbb R}^n)$.

Let us now prove \eqref{Newtonian-D-B-F-Robin2-SM} for {$\widetilde{v}\in \widetilde{H}^{s-\frac{1}{2}}(\Omega )$ and $u\in H^{s+\frac{1}{2}}(\Omega )$}. To this end, note that ${\mathcal D}(\mathbb R^n)$ is dense in $H^{s+\frac{1}{2}}(\mathbb R^n)$, and hence there exists a sequence
$\{\varphi _k\}_{k\geq 1}\subset \mathcal D(\mathbb R^n)$ such that
\begin{align}
\label{conv}
\lim _{k\to \infty }\varphi _k=E^{s+\frac{1}{2}}u \mbox{ in } H^{s+\frac{1}{2}}(\mathbb R^n)
{\quad
\Longrightarrow\quad}
\lim _{k\to \infty }r_\Omega\varphi _k=u \mbox{ in } H^{s+\frac{1}{2}}(\Omega ).
\end{align}
Moreover, {by \eqref{5.14}},
$\left\|{\varphi_j}\widetilde{v}-{\varphi _\ell} \widetilde{v}\right\| _{\widetilde{H}^{s-\frac{3}{2}}(\Omega)}\leq
\tilde c\|\varphi _j-\varphi _\ell \|_{H^{s+\frac{1}{2}}(\Omega )}\|\widetilde{v}\|_{\widetilde{H}^{s-\frac{1}{2}}(\Omega )},$
and by \eqref{conv} we deduce that $\{{\varphi _k}\widetilde{v}\}_{k\geq 1}$ is a Cauchy sequence in $\widetilde{H}^{s-\frac{3}{2}}(\Omega )$, i.e., it converges to an element ${\widetilde w}\in \widetilde{H}^{s-\frac{3}{2}}(\Omega )$. {It is easy to show that the limit is unique, i.e., does not depend on the chosen {sequence $\{\varphi _k\}_{k\geq 1}$}, and we define the product ${( E^{s+\frac{1}{2}}u)}\widetilde v:=\widetilde w$.}
{Hence}
\begin{align}
\label{conv4}
{E^{s+\frac{1}{2}}u\widetilde{v}:=\lim _{k\to \infty }{\varphi _k}\widetilde{v}}\ \mbox{ in } \ \widetilde{H}^{s-\frac{3}{2}}(\Omega ).
\end{align}
Then we obtain that
\begin{align}
\label{conv5}
&\|{(E^{s+\frac{1}{2}}u)}\widetilde{v}\|_{\widetilde{H}^{s-\frac{3}{2}}(\Omega )}\leq \|{\varphi _k}\widetilde{v}-{(E^{s+\frac{1}{2}}u)}\widetilde{v}\|_{\widetilde{H}^{s-\frac{3}{2}}(\Omega )}+
\|{\varphi _k}\widetilde{v}\|_{\widetilde{H}^{s-\frac{3}{2}}(\Omega )}\nonumber\\
&\leq \|{\varphi _k}\widetilde{v}-{ (E^{s+\frac{1}{2}}u)}\widetilde{v}\|_{\widetilde{H}^{s-\frac{3}{2}}(\Omega )}+\tilde{c}\|\varphi _k\|_{H^{s+\frac{1}{2}}(\Omega )}\|\widetilde{v}\|_{\widetilde{H}^{s-\frac{1}{2}}(\Omega )}\nonumber\\
&\leq \|{\varphi _k}\widetilde{v}-{ (E^{s+\frac{1}{2}}u)}\widetilde{v}\|_{\widetilde{H}^{s-\frac{3}{2}}(\Omega )}+\tilde{c}\|\varphi _k-u\|_{H^{s+\frac{1}{2}}(\Omega )}\|\widetilde{v}\|_{\widetilde{H}^{s-\frac{1}{2}}(\Omega )}\nonumber\\
&\hspace{5em}+\tilde{c}\|u\|_{\widetilde{H}^{s+\frac{1}{2}}(\Omega )}\|\widetilde{v}\|_{\widetilde{H}^{s-\frac{1}{2}}(\Omega )}.
\end{align}
Finally, by \eqref{conv}
{and} \eqref{conv5}, {we} obtain inequality \eqref{Newtonian-D-B-F-Robin2-SM}, as desired. {Such an inequality is still valid when $\Omega $ is a Lipschitz domain in the compact Riemannian manifold $M$, as follows from the definition of the involved Sobolev spaces on $M$ (see Section 2.1) and the version of \eqref{Newtonian-D-B-F-Robin2-SM} in the Euclidean setting (see also \cite[Lemma 28]{Holst})}.

Further, by \eqref{Newtonian-D-B-F-Robin2-SM} and the local representation formula \eqref{locally} of $\nabla _{{\bf u}}{\bf v}$, there exists a constant $c_{(0)}^*\equiv c_{(0)}^*(\Omega ,s)>0$ such that
\begin{align}
\label{inclusion-NS-3}
\left\|{\left({\widetilde E}^{s-\frac{1}{2}}\partial_ju^\ell\right)E^{s+\frac{1}{2}}v^j}\right\| _{\widetilde{H}^{s-\frac{3}{2}}(\Omega )}
&\leq c_{(0)}^*\|{\bf u}\|_{H^{s+\frac{1}{2}}(\Omega ,\Lambda ^1TM)}\|{\bf v}\|_{H^{s+\frac{1}{2}}(\Omega ,\Lambda ^1TM)},\nonumber\\
&\hspace{3em}\forall \ {\bf u}, {\bf v}\in H^{s+\frac{1}{2}}(\Omega ,\Lambda ^1TM).
\end{align}
In addition, by using again \eqref{inclusion-NS-2} we obtain that
${\widetilde H^0}(\Omega )\cdot H^{s+\frac{1}{2}}(\Omega )\hookrightarrow \widetilde{H}^{s-\frac{1}{2}}(\Omega )\cdot H^{s+\frac{1}{2}}(\Omega )\hookrightarrow \widetilde{H}^{s-\frac{3}{2}}(\Omega ).$
Hence
\begin{align}
\label{inclusion-NS-3-new}
{({\widetilde E^0}|{\bf u}|)}{E^{s+\frac{1}{2}}}{\bf v}\in \widetilde{H}^{s-\frac{3}{2}}(\Omega ,\Lambda ^1TM),\ \forall \ {\bf u}, {\bf v}\in H^{s+\frac{1}{2}}(\Omega ,\Lambda ^1TM)
\end{align}
and there exists a constant $c_0^*\equiv c_0^*(\Omega ,s)>0$ such that
\begin{align}
\label{inclusion-NS-3-new}
\left\|{({\widetilde E^0}|{\bf u}|)}{E^{s+\frac{1}{2}}}{\bf v}\right\|_{\widetilde{H}^{s-\frac{3}{2}}(\Omega )}&\leq c_0^*\|{\bf u}\|_{H^{s+\frac{1}{2}}(\Omega ,\Lambda ^1TM)}\|{\bf v}\|_{H^{s+\frac{1}{2}}(\Omega ,\Lambda ^1TM)},\nonumber\\
&\hspace{3em}\forall \ {\bf u}, {\bf v}\in H^{s+\frac{1}{2}}(\Omega ,\Lambda ^1TM).
\end{align}
Thus, the nonlinear operator ${\widetilde{\mathcal I}}_{k;\beta ;{\Omega }}:H_{\delta }^{s+\frac{1}{2}}({\Omega },\Lambda ^1TM)\to \widetilde{H}^{s-\frac{3}{2}}(\Omega ,\Lambda ^1TM)$ given by \eqref{nonlin-2} is
positively homogeneous of order 2, and bounded in the sense of \eqref{NS-s2}, with the constant
\begin{align}
\label{const-s}
C_1={|k|}c_{(0)}^*+{|\beta |}c_0^*.
\end{align}
Moreover, inequalities \eqref{inclusion-NS-3} and \eqref{inclusion-NS-3-new} imply that
\begin{align}
&\|{\widetilde{\mathcal I}}_{k;\beta ;{\Omega }}({\bf v})
-{\widetilde{\mathcal I}}_{k;\beta ;{\Omega }}({\bf w})\| _{\widetilde{H}^{s-\frac{3}{2}}(\Omega,\Lambda^1TM)}\nonumber\\
&\leq {|k|}\left\|({\widetilde E^0}|{\bf v}|){E^{s+\frac{1}{2}}}{\bf v}-({\widetilde E^0}|{\bf w}|){ E^{s+\frac{1}{2}}}{\bf w}\right\|_{\widetilde{H}^{s-\frac{3}{2}}(\Omega ,\Lambda ^1TM)}\nonumber
\\
&\quad +{|\beta |}\left\|({\widetilde E}^{s-\frac{1}{2}}\nabla _{{\bf v}}){ E^{s+\frac{1}{2}}}{\bf v}
-{({\widetilde E}^{s-\frac{1}{2}}\nabla _{{\bf w}}) E^{s+\frac{1}{2}}}{\bf w}\right\| _{\widetilde{H}^{s-\frac{3}{2}}(\Omega ,\Lambda ^1TM)}\nonumber
\\
&\leq {|k|}\left\|({\widetilde E^0}(|{\bf v}|-|{\bf w}|)){E^{s+\frac{1}{2}}}{\bf v}
+{(\widetilde E^0|{\bf w}|)}{E^{s+\frac{1}{2}}}({\bf v}-{\bf w})\right\| _{\widetilde{H}^{s-\frac{3}{2}}(\Omega ,\Lambda ^1TM)}\nonumber
\\
&\quad +{|\beta |}\left\|{({\widetilde E}^{s-\frac{1}{2}}\nabla _{{\bf v}-{\bf w}}) E^{s+\frac{1}{2}}}{\bf v}+{({\widetilde E}^{s-\frac{1}{2}}\nabla _{\bf w}) E^{s+\frac{1}{2}}}({\bf v}
-{\bf w})\right\|_{\widetilde{H}^{s-\frac{3}{2}}(\Omega ,\Lambda ^1TM)}\nonumber
\end{align}
\begin{align}
&\leq ({|k|}c+{|\beta |}C_0)\|{\bf v}-{\bf w}\|_{H^{s+\frac{1}{2}}(\Omega ,\Lambda ^1TM)}\left(\|{\bf v}\|_{H^{s+\frac{1}{2}}(\Omega ,\Lambda ^1TM)}+\|{\bf w}\|_{H^{s+\frac{1}{2}}(\Omega )}\right)\nonumber\\
&=C_1\|{\bf v}-{\bf w}\|_{H^1(\Omega ,\Lambda ^1TM)}\left(\|{\bf v}\|_{H^{s+\frac{1}{2}}(\Omega ,\Lambda ^1TM)}+\|{\bf w}\|_{H^{s+\frac{1}{2}}(\Omega ,\Lambda ^1TM)}\right),\nonumber
\end{align}
giving inequality \eqref{NS-s3}. Then the continuity of the operator \eqref{NS-s} is a consequence of \eqref{NS-s3}.

{To prove that the operator $\widetilde{\mathcal I}_{k;\beta ;{\Omega }}$ does not depend on the chosen extension operator $E^{s+\frac{1}{2}}$ in definition \eqref{nonlin-2}, let us consider the operators
$\widetilde{\mathcal I}^1_{k;\beta ;{\Omega }}$ and $\widetilde{\mathcal I}^2_{k;\beta ;{\Omega }}$ obtained using the extensions $E_1^{s+\frac{1}{2}}$ and $E_2^{s+\frac{1}{2}}$, respectively. Then
\begin{align*}
&\|\widetilde{\mathcal I}^1_{k;\beta ;{\Omega }}({\bf v})-\widetilde{\mathcal I}^2_{k;\beta ;{\Omega }}({\bf v})\| _{\widetilde{H}^{s-\frac{3}{2}}(\Omega,\Lambda^1TM)}\nonumber
\\
&\leq {|k|}\left\|(\widetilde E^0|{\bf v}|)E_1^{s+\frac{1}{2}}{\bf v}-(\widetilde E^0|{\bf v}|)E_2^{s+\frac{1}{2}}{\bf v}\right\|_{\widetilde{H}^{s-\frac{3}{2}}(\Omega ,\Lambda ^1TM)}\nonumber
\\
&\hspace{2em}+{|\beta |}\left\|({\widetilde E}^{s-\frac{1}{2}}\nabla _{{\bf v}})E_1^{s+\frac{1}{2}}{\bf v}\!-\!({\widetilde E}^{s-\frac{1}{2}}\nabla _{{\bf v}})E_2^{s+\frac{1}{2}}{\bf v}\right\| _{\widetilde{H}^{s-\frac{3}{2}}(\Omega,\Lambda ^1TM)}\nonumber
\\
&={|k|}\left\|(\widetilde E^0|{\bf v}|)(E_1^{s+\frac{1}{2}}{\bf v}-E_2^{s+\frac{1}{2}}{\bf v})\right\| _{\widetilde{H}^{s-\frac{3}{2}}(\Omega ,\Lambda ^1TM)}\nonumber\\
&\hspace{2em}+{|\beta |}\left\|({\widetilde E}^{s-\frac{1}{2}}\nabla _{\bf w})(E_1^{s+\frac{1}{2}}{\bf v}-E_2^{s+\frac{1}{2}}{\bf v})\right\|_{\widetilde{H}^{s-\frac{3}{2}}(\Omega ,\Lambda ^1TM)}\nonumber
\\
&\leq ({|k|}c+{|\beta |}C_0)\|{\bf v}\|_{H^{s+\frac{1}{2}}(\Omega ,\Lambda ^1TM)}
\left\|{r_{\Omega }\left(E_1^{s+\frac{1}{2}}{\bf v}\!-\!E_2^{s+\frac{1}{2}}{\bf v}\right)}\right\|_{H^{s+\frac{1}{2}}(\Omega ,\Lambda ^1TM)}\!\!=\!0
\end{align*}
since ${r_{\Omega }\left(E_1^{s+\frac{1}{2}}{\bf v}-E_2^{s+\frac{1}{2}}{\bf v}\right)}=0$.}
\end{proof}

\subsection{Existence and uniqueness for {nonlinear} transmission problem \eqref{Poisson-transmission-zero-order1}}

Next we use the {notation
$\widetilde{\mathcal I}_{{\Omega }_{-}}({\bf v}):={\widetilde{\mathcal I}}_{0;1;{\Omega }_{-}}({\bf v})$
}
and show the main result of {Section \ref{nonlin}}, i.e., the existence and uniqueness of {solution} in the space ${\mathcal X}_{s}$ for the nonlinear problem \eqref{Poisson-transmission-zero-order1} when the given data belong to the space ${\mathcal Y}_{s}$, $s\in (0,1)$ (see \eqref{X1-sm} and \eqref{Y1-sm}).
\begin{thm}
\label{Poisson-transmission-NS-DFB-0} {Let $M$ satisfy {Assumption} $\ref{H}$ and $\dim (M)\in \{2,3\}$}. Let $\Omega _{+}:=\Omega \subset M$ be a Lipschitz domain. {Let $\Omega _{-}:=M\setminus \overline{\Omega }$ satisfy Assumption $\ref{connected-set}$}. Let $s\in (0,1)$ and ${\mu >0}$, $k,\beta \in {\mathbb R}$ be given nonzero constants. Let {$V\in L^{\infty }(M,\Lambda^1TM\otimes \Lambda ^1TM)$} be a {symmetric} tensor field satisfying the positivity condition \eqref{V-positive} and ${\mathcal P}\in L^{\infty }(\partial \Omega ,\Lambda ^1TM\otimes \Lambda ^1TM)$ be a symmetric tensor field which satisfies the positivity condition \eqref{Poisson-transmission3}. Then there exist two constants {$\zeta >0$ and $\eta >0$} depending only on $\Omega _{\pm}$, ${\mathcal P}$, $V$, $s$, $k$, $\beta $, $\mu $, with the property that for all data $(\tilde{\bf f}_{+},\tilde{\bf f}_{-},{\bf h},{\bf r})\in {\mathcal Y}_{s}$, which satisfy the condition
\begin{equation}
\label{cond-small}
\big\|(\tilde{\bf f}_{+},\tilde{\bf f}_{-},{\bf h},{\bf r})\big\|_{{\mathcal Y}_{s}}
\leq {\zeta },
\end{equation}
the transmission problem for the generalized Darcy-Forchheimer-Brinkman and Navier-Stokes systems
\eqref{Poisson-transmission-zero-order1} has a unique solution $\left(({\bf u}_{+},\pi _{+}),({\bf u}_{-},\pi _{-})\right)\in {\mathcal X}_{s}$, such that
\begin{equation}
\label{estimate1-zero-order}
\|\left({\bf u}_{+}{\bf u}_{-}\right)\|_{H_{\delta }^{s+\frac{1}{2}}({\Omega }_{+},\Lambda ^1TM)\times H_{\delta }^{s+\frac{1}{2}}({\Omega }_{-},\Lambda ^1TM)}\leq {\eta }.
\end{equation}
Moreover, the solution depends continuously on the given data and satisfies the estimate
\begin{align}
\label{estimate-D-B-F-new1-new1-D-Rn2}
\|\left(({\bf u}_{+},\pi _{+}),({\bf u}_{-},\pi _{-})\right)\|_{{\mathcal X}_{s}}\leq
C\big\|(\tilde{\bf f}_{+},\tilde{\bf f}_{-},{\bf h},{\bf r})\big\|_{{\mathcal Y}_{s}}
\end{align}
for some positive constant $C$ which depends only on ${\Omega }_{+}$, ${\Omega }_{-}$, ${\mathcal P}$, $V$, $s$, $\mu $.
\end{thm}
\begin{proof}
We rewrite the nonlinear transmission problem \eqref{Poisson-transmission-zero-order1} in the form
\begin{equation}
\label{Poisson-Brinkman-smn-1}
\!\!\!\left\{\begin{array}{lll}
{\bf L}{{\bf u}_+}+V{{\bf u}_+}+d{\pi_+}=
\tilde{\bf f}_{+}|_{\Omega _{+}}-{\widetilde{\mathcal I}}_{k;\beta ;\Omega_+}({\bf u}_+)|_{\Omega_+},\
{\delta {\bf u}_+=0 \mbox{ in } {\Omega }_{+},}\\
{\bf L}{{\bf u}_-}+d{\pi _-}=\tilde{\bf f}_{-}|_{\Omega _{-}}-{\widetilde{\mathcal I}}_{\Omega_{-}}({\bf u}_{-})|_{\Omega_{-}},\
{\delta {\bf u}_{-}=0} \mbox{ in }\ {\Omega }_{-},\\
{\mu }\gamma_{+}{{\bf u}_+}-\gamma_{-}{\bf u}_{-}={\bf h} \mbox{ on } \partial \Omega ,\\
{\bf t}_{V}^{+}\left({{\bf u}_+},{\pi_+};\tilde{\bf f}_{+}-{\widetilde{\mathcal I}}_{k;\beta ;{\Omega }_{+}}({\bf u}_+)\right)
-{{\bf t}_0^-}\left({{\bf u}_-},{\pi_-};\tilde{\bf f}_{-}-{\widetilde{\mathcal I}}_{{\Omega }_{-}}({\bf u}_{-})\right)\\
\hspace{7em}+{\mathcal P} \gamma_{+}{\bf u}_{+}={\bf r} \mbox{ on } \partial \Omega .
\end{array}\right.
\end{equation}
Next we construct a nonlinear operator $\left({\mathcal T}_{+},{\mathcal T}_{-}\right)$ mapping a closed ball ${\bf B}_{\eta}$ of the product space $H_{\delta }^{s+\frac{1}{2}}({\Omega }_{+},\Lambda ^1TM)\times H_{\delta }^{s+\frac{1}{2}}({\Omega }_{-},\Lambda ^1TM)$ to ${\bf B}_{\eta}$ and being a contraction on ${\bf B}_{\eta}$. Then the unique fixed point of $\left({\mathcal T}_{+},{\mathcal T}_{-}\right)$ will determine a solution of the nonlinear problem \eqref{Poisson-Brinkman-smn-1}, and hence of the problem \eqref{Poisson-transmission-zero-order1}.

$\bullet $ {\bf The nonlinear operator $\left({\mathcal T}_{+},{\mathcal T}_{-}\right)$ and the existence result}

For a fixed {couple} $(\mathbf u_+,{\bf u}_{-})\in H_{\delta }^{s+\frac{1}{2}}({\Omega }_{+},\Lambda ^1TM)\times H_{\delta }^{s+\frac{1}{2}}({\Omega }_{-},\Lambda ^1TM)$, we consider the linear transmission problem for the Stokes and Brinkman systems
\begin{equation}
\label{Newtonian-D-B-F-new2-D-Rn}
\!\!\!\left\{\begin{array}{lll}
{\bf L}{{\bf u}^0_+}{+V{{\bf u}^0_+}}+d{\pi^0_+}=
\tilde{\bf f}_{+}|_{\Omega _{+}}-{\widetilde{\mathcal I}}_{k;\beta ;\Omega_+}({\bf u}_+)|_{\Omega_+},\
{\delta {\bf u}^0_+=0 \mbox{ in } {\Omega }_{+},}\\
{\bf L}{{\bf u}^0_-}+d{\pi^0_-}=\tilde{\bf f}_{-}|_{\Omega _{-}}-{\widetilde{\mathcal I}}_{\Omega_{-}}({\bf u}_{-})|_{\Omega_{-}},\ {\delta {\bf u}^0_{-}=0} \mbox{ in } {\Omega }_{-},\\
{\mu }\gamma_{+}{{\bf u}^0_+}-\gamma_{-}{\bf u}^0_{-}
={\bf h} \mbox{ on } \partial \Omega ,\\
{{\bf t}_V^+}\left({{\bf u}^0_+},{\pi^0_+};
\tilde{\bf f}_{+}-{\widetilde{\mathcal I}}_{k;\beta ;{\Omega }_{+}}({\bf u}_+)\right)-
{{\bf t}_0^-}\left({{\bf u}^0_-},{\pi^0_-};\tilde{\bf f}_{-}-{\widetilde{\mathcal I}}_{{\Omega }_{-}}({\bf u}_{-})\right)\\
\hspace{7em}+{\mathcal P} \gamma_{+}{\bf u}^0_{+}={\bf r} \mbox{ on } \partial \Omega ,
\end{array}\right.
\end{equation}
{with unknowns $\left(({\bf u}^0_{+},\pi^0_{+}),({\bf u}^0_{-},\pi^0_{-})\right)$}.
By Lemma~\ref{embeddings}, we have ${\widetilde{\mathcal I}}_{k;\beta ;\Omega_+}({\bf u}_+)\in \widetilde{H}^{{s-\frac{3}{2}}}({\Omega }_{+},\Lambda ^1TM)$ and ${\widetilde{\mathcal I}}_{\Omega_{-}}({\bf u}_{-})\in \widetilde{H}^{{s-\frac{3}{2}}}({\Omega }_{-},\Lambda ^1TM)$, and then for given data $(\tilde{\bf f}_+,\tilde{\bf f}_{-},{\bf h},{\bf r})\in {\mathcal Y}_{s}$, the right hand side of (\ref{Newtonian-D-B-F-new2-D-Rn}) belongs to the space ${\mathcal Y}_{s}$ {defined in \eqref{Y1-sm}}. Hence, by Theorem \ref{Poisson-transmission-Stokes}, there exists a unique solution $\left(({{\bf u}^0_+},{\pi^0_+}),({{\bf u}^0_-},{\pi^0_-})\right)${ of problem (\ref{Newtonian-D-B-F-new2-D-Rn}) in the space  ${\mathcal X}_{s}$ defined in \eqref{X1-sm}}.
This solution can be written as
\begin{align}
\label{solution-v0}
&\!\!\left(({{\bf u}^0_+},{\pi^0_+}),({{\bf u}^0_-},{\pi^0_-})\right)\!\!=\!\!
\left(({\mathcal T}_{+}({\bf u}_+,{\bf u}_{-}),{P}_{+}({\bf u}_+,{\bf u}_{-})),({\mathcal T}_{-}({\bf u}_+,{\bf u}_{-}),{P}_{-}({\bf u}_+,{\bf u}_{-}))\right)\nonumber\\
&\hspace{5em}{:=}{\mathcal S}_V\left(\tilde{\bf f}_{+}-{\widetilde{\mathcal I}}_{k;\beta ;\Omega_+}({\bf u}_+),\, \tilde{\bf f}_{-}-{\widetilde{\mathcal I}}_{\Omega_{-}}({\bf u}_{-}),\, {\bf h},\, {\bf r}\right)\in {\mathcal X}_{s},
\end{align}
where ${\mathcal S}_V:{\mathcal Y}_{s}\to {\mathcal X}_{s}$ is the linear and continuous operator introduced in Theorem \ref{Poisson-transmission-Stokes}.
Moreover, the continuity of ${\mathcal S}_V$ and Lemma~\ref{embeddings} imply that there exists a constant $C_*\equiv C_*({\Omega }_{+},{\Omega }_{-},{\mathcal P},V,\mu )>0$ such that the operator
\begin{align}
\label{Newtonian-D-B-F-new3n}
&({\mathcal T}_+,P_+,{\mathcal T}_-,P_-):H_{\delta }^{{s+\frac{1}{2}}}({\Omega }_{+},\Lambda ^1TM)\times H_{\delta }^{{s+\frac{1}{2}}}({\Omega }_{-},\Lambda ^1TM)\to {\mathcal X}_{s},
\end{align}
{defined} by \eqref{solution-v0}, satisfies the inequalities
\begin{align}
\label{estimate-D-B-F-new1-new1-D-Rn}
&\left\|\left(({\mathcal T}_{+}({\bf u}_+,{\bf u}_{-}),{P}_{+}({\bf u}_+,{\bf u}_{-})),({\mathcal T}_{-}({\bf u}_+,{\bf u}_{-}),{P}_{-}({\bf u}_+,{\bf u}_{-}))\right)\right\|_{{\mathcal X}_{s}}
\end{align}
\begin{align}
&\leq C_*\left\|\left(\tilde{\bf f}_{+}-\widetilde{\mathcal I}_{k;\beta;\Omega_+}({\bf u}_+),
\tilde{\bf f}_{-}-\widetilde{\mathcal I}_{\Omega_{-}}({\bf u}_{-}),{\bf h},{\bf r}\right)\right\|_{{\mathcal Y}_{s}}\nonumber
\\
&\leq C_*\Big(\left\|\left(\tilde{\bf f}_{+},\tilde{\bf f}_{-},{\bf h},{\bf r}\right)\right\|_{\mathcal Y_{s;\delta }}+\|\widetilde{\mathcal I}_{k;\beta;\Omega_+}({\bf u}_+)\|_{\widetilde{H}^{{s-\frac{3}{2}}}({\Omega }_{+},\Lambda ^1TM)}\nonumber
\\
&\hspace{3em}+\|\widetilde{\mathcal I}_{\Omega_{-}}({\bf u}_{-})\|_{\widetilde{H}^{{s-\frac{3}{2}}}({\Omega }_{-},\Lambda ^1TM)}\Big)\nonumber\\
&\leq C_*\big\|\big(\tilde{\bf f}_{+},\tilde{\bf f}_{-},{\bf h},{\bf r}\big)\big\|_{{\mathcal Y}_{s}}+C_{*}\tilde{C}_1{\|({\bf u}_{+},{\bf u}_{-})\|_{H_\delta ^{{s+\frac{1}{2}}}({\Omega }_{+},\Lambda ^1TM)\times H_\delta ^{{s+\frac{1}{2}}}({\Omega }_{-},\Lambda ^1TM)}^2},\nonumber
\end{align}
for all $({\bf u}_{+},{\bf u}_{-})\in H_\delta ^{s+\frac{1}{2}}({\Omega }_{+},\Lambda ^1TM)\times H_\delta ^{s+\frac{1}{2}}({\Omega }_{-},\Lambda ^1TM)$, with a constant $\tilde{C}_1\equiv \tilde{C}_1({\Omega }_{+},\Omega _{-},{k,\beta},s)>0$ determined by the constant $C_1$ from Lemma~\ref{embeddings} {corresponding to operator ${\widetilde{\mathcal I}}_{k;\beta ;\Omega _{+}}$} and a similar one corresponding to operator ${\widetilde{\mathcal I}}_{\Omega _{-}}$. Thus, the nonlinear operator given by \eqref{Newtonian-D-B-F-new3n} is continuous and bounded in the sense of \eqref{estimate-D-B-F-new1-new1-D-Rn}.
Moreover, by \eqref{Newtonian-D-B-F-new2-D-Rn} and \eqref{solution-v0}, we deduce that
\begin{equation}
\label{Newtonian-D-B-F-new4n}
\left\{\begin{array}{lll}
{\bf L}{\mathcal T}_+({\bf u}_{+},{\bf u}_{-})+V{\mathcal T}_+({\bf u}_{+},{\bf u}_{-})+d P_+({\bf u}_{+},{\bf u}_{-})\\
\hspace{7em}=\tilde{\bf f}_{+}|_{\Omega _{+}}-{\widetilde{\mathcal I}}_{k;\beta ;\Omega_+}({\bf u}_+)|_{\Omega_+} \mbox{ in } {\Omega }_{+},\\
{\delta {\mathcal T}_+({\bf u}_{+},{\bf u}_{-})=0 \mbox{ in } {\Omega }_{+},}\\
{\bf L}{\mathcal T}_{-}({\bf u}_{+},{\bf u}_{-})+d P_-({\bf u}_{+},{\bf u}_{-})
=\tilde{\bf f}_{-}|_{\Omega _{-}}-{\widetilde{\mathcal I}}_{\Omega_{-}}({\bf u}_{-})|_{\Omega_{-}}
\mbox{ in } {\Omega }_{-},\\
{\delta {\mathcal T}_{-}({\bf u}_{+},{\bf u}_{-})=0 \mbox{ in } {\Omega }_{-},}\\
{\mu }\gamma_{+}\left({\mathcal T}_+({\bf u}_{+},{\bf u}_{-})\right)-\gamma_{-}\left({\mathcal T}_{-}({\bf u}_{+},{\bf u}_{-})\right)={\bf h}\in H_{\nu }^{s}(\partial\Omega ,\Lambda ^1TM),\\
{\bf t}_{V}^{+}\left({\mathcal T}_+({\bf u}_{+},{\bf u}_{-}),P_+({\bf u}_{+},{\bf u}_{-});\tilde{\bf f}_{+}-{\widetilde{\mathcal I}}_{k;\beta ;{\Omega }_{+}}({\bf u}_{+})\right)\\
\hspace{3em}-{{\bf t}_0^-}\left({\mathcal T}_{-}({\bf u}_{+},{\bf u}_{-}),
P_-({\bf u}_{+},{\bf u}_{-});\tilde{\bf f}_{-}-{\widetilde{\mathcal I}}_{\Omega_{-}}({\bf u}_{-})\right)\\
\qquad\qquad\qquad +{\mathcal P}\gamma_{+}\left({\mathcal T}_+({\bf u}_{+},{\bf u}_{-})\right)
={\bf r}\in H^{s-1}(\partial\Omega ,\Lambda ^1TM).
\end{array}\right.
\end{equation}
Hence, if we can show that the nonlinear operator
\begin{align}
\label{nonlin-oper-sm}
\left({\mathcal T}_{+},{\mathcal T}_{-}\right):&H_{\delta }^{s+\frac{1}{2}}({\Omega }_{+},\Lambda ^1TM)\times H_{\delta }^{s+\frac{1}{2}}({\Omega }_{-},\Lambda ^1TM)\nonumber\\
&\to H_{\delta }^{s+\frac{1}{2}}({\Omega }_{+},\Lambda ^1TM)\times H_{\delta }^{s+\frac{1}{2}}({\Omega }_{-},\Lambda ^1TM)
 \end{align}
has a fixed point $({\bf u}_+,{\bf u}_{-})\in H_{\delta}^{s+\frac{1}{2}}({\Omega }_{+},\Lambda ^1TM)\times H_{\delta}^{s+\frac{1}{2}}({\Omega }_{-},\Lambda ^1TM)$, i.e.,
${{\mathcal T}_{+}({\bf u}_{+},{\bf u}_{-})={\bf u}_{+}},\ \ {{\mathcal T}_{-}({\bf u}_{+},{\bf u}_{-})={\bf u}_{-}},$
then {this} ${\bf u}_{\pm }$ and the pressure functions $\pi_\pm =P_\pm({\bf u}_+,{\bf u}_{-})$ will provide a solution of {nonlinear} transmission problem (\ref{Poisson-transmission-zero-order1}) in the space ${\mathcal X}_{s}$.
{In} order to show the desired result we consider the constants $\zeta $ and $\eta $ given by
\begin{align}
\label{Newtonian-D-B-F-new9n}
&{\zeta }=\frac{3}{16\tilde{C}_1C_*^2}>0,\ \ {\eta }=\frac{1}{4\tilde{C}_1C_*}>0
\end{align}
(see also \cite[Lemma 29]{Choe-Kim}, \cite[(5.25)]{K-L-M-W}), and the closed ball
\begin{align}
\label{gamma-0n}
\mathbf{B}_{\eta }:=\Big\{&\left({\bf v}_{+},{\bf v}_{+}\right)\in H_{\delta }^{s+\frac{1}{2}}({\Omega }_{+},\Lambda ^1TM)\times H_{\delta }^{s+\frac{1}{2}}({\Omega }_{-},\Lambda ^1TM):\nonumber\\
&\|\left({\bf v}_{+},{\bf v}_{+}\right)\|_{H_{\delta }^{s+\frac{1}{2}}({\Omega }_{+},\Lambda ^1TM)\times H_{\delta }^{s+\frac{1}{2}}({\Omega }_{-},\Lambda ^1TM)}\leq \eta\Big\}.
\end{align}
In addition, we assume that the given data satisfy {condition \eqref{cond-small}}.
By using (\ref{estimate-D-B-F-new1-new1-D-Rn}), \eqref{Newtonian-D-B-F-new9n}, \eqref{gamma-0n} {and \eqref{cond-small}} we {obtain}
\begin{align}
\label{Newtonian-D-B-F-new9-new-crackn}
&\|\left({\mathcal T}_{+}({\bf v}_+,{\bf v}_{-}),{P}_{+}({\bf v}_+,{\bf v}_{-}),{\mathcal T}_{-}({\bf v}_+,{\bf v}_{-}),{P}_{-}({\bf v}_+,{\bf v}_{-})\right)\|_{{\mathcal X}_{s}}\leq {\frac{1}{4\tilde{C}_1C_*}}
= \eta ,
\end{align}
for all $\left({\bf v}_{+},{\bf v}_{+}\right)\in \mathbf{B}_{\eta }$, which {implies} that $$\left\|\left({\mathcal T}_{+},{\mathcal T}_{-}\right)({\bf u}_+,{\bf u}_{-})\right\|_{H_{\delta }^{s+\frac{1}{2}}({\Omega }_{+},\Lambda ^1TM)\times H_{\delta }^{s+\frac{1}{2}}({\Omega }_{-},\Lambda ^1TM)}\leq \eta,\ \forall \ \left({\bf u}_{+},{\bf u}_{+}\right)\in \mathbf{B}_{\eta },$$
and hence $\left({\mathcal T}_{+},{\mathcal T}_{-}\right)$ maps $\mathbf{B}_{\eta }$ to $\mathbf{B}_{\eta }$.

Now we prove that $\left({\mathcal T}_{+},{\mathcal T}_{-}\right)$ is a Lipschitz continuous mapping on the ball $\mathbf{B}_\eta$. To this {end}, we use \eqref{solution-v0} and Lemma~\ref{embeddings}, and obtain the inequalities
\begin{align}
\label{4.30n}
&\left\|\left({\mathcal T}_{+},{\mathcal T}_{-}\right)({\bf v}_+,{\bf v}_{-})-\left({\mathcal T}_{+},{\mathcal T}_{-}\right)({\bf w}_+,{\bf w}_{-})\right\|_{H_{\delta }^{s+\frac{1}{2}}({\Omega }_{+},\Lambda ^1TM)\times H_{\delta }^{s+\frac{1}{2}}({\Omega }_{-},\Lambda ^1TM)}\nonumber\\
&\leq C_*\Big(\|{{\widetilde{\mathcal I}}_{k;\beta ;{\Omega }_{+}}({\bf v}_{+})}-{{\widetilde{\mathcal I}}_{k;\beta ;{\Omega }_{+}}({\bf w}_{+})}\|_{\widetilde{H}^{s-\frac{3}{2}}({\Omega }_{+},\Lambda ^1TM)}\nonumber\\
&\hspace{3em}+\|{{\widetilde{\mathcal I}}_{{\Omega }_{-}}({\bf v}_{-})}-{{\widetilde{\mathcal I}}_{{\Omega }_{-}}({\bf w}_{-})}\|_{\widetilde{H}^{s-\frac{3}{2}}({\Omega }_{-},\Lambda ^1TM)}\Big)\nonumber\\
& \leq C_*\tilde{C}_1\Big(\big(\|{\bf v}_+\|_{H^{s+\frac{1}{2}}({\Omega }_{+},\Lambda ^1TM)}\!+\!\|{\bf w}_+\|_{H^{s+\frac{1}{2}}({\Omega }_{+},\Lambda ^1TM)}\big)\|{\bf v}_+\!-\!{\bf w}_+\|_{H^{s+\frac{1}{2}}({\Omega }_{+})}\nonumber\\
&\ +\big(\|{\bf v}_{-}\|_{H^{s+\frac{1}{2}}({\Omega }_{-},\Lambda ^1TM)}\!+\!\|{\bf w}_{-}\|_{H^{s+\frac{1}{2}}({\Omega }_{-},\Lambda ^1TM)}\big)\|{\bf v}_{-}-{\bf w}_{-}\|_{H^{s+\frac{1}{2}}({\Omega }_{-})}\Big)\nonumber\\
& \leq 2\eta C_*\tilde{C}_1\|({\bf v}_+,{\bf v}_{-})-({\bf w}_+,{\bf w}_{-})\|_{H_{\delta }^{s+\frac{1}{2}}({\Omega }_{+},\Lambda ^1TM)\times H_{\delta }^{s+\frac{1}{2}}({\Omega }_{-},\Lambda ^1TM)}\nonumber\\
&=\frac{1}{2}\|({\bf v}_+,{\bf v}_{-})-({\bf w}_+,{\bf w}_{-})\|_{H_{\delta }^{s+\frac{1}{2}}({\Omega }_{+},\Lambda ^1TM)\times H_{\delta }^{s+\frac{1}{2}}({\Omega }_{-},\Lambda ^1TM)},\nonumber\\
&\hspace{15em}\forall\ ({\bf v}_+,{\bf v}_{-}),({\bf w}_+,{\bf w}_{-})\in \mathbf{B}_\eta,
\end{align}
where $C_*$ and $\tilde{C}_1$ are the constants in \eqref{estimate-D-B-F-new1-new1-D-Rn} {and from Lemma \ref{embeddings}}.
Thus, $\left({\mathcal T}_{+},{\mathcal T}_{-}\right):\mathbf{B}_{\eta }\to \mathbf{B}_{\eta }$ is a contraction, as desired.
Then the Banach-Caccioppoli fixed point theorem implies that the map $\left({\mathcal T}_{+},{\mathcal T}_{-}\right):\mathbf{B}_{\eta }\to \mathbf{B}_{\eta }$ has a unique fixed point
$\left({\bf u}_{+},{\bf u}_{-}\right){=\left({\mathcal T}_{+},{\mathcal T}_{-}\right)\left({\bf u}_{+},{\bf u}_{-}\right)}\in \mathbf{B}_{\eta }$.
Hence, $\left(({\bf u}_{+},\pi _{+}),({\bf u}_{-},\pi _{-})\right)$, where $\pi_\pm =P_\pm\left({\bf u}_{+},{\bf u}_{-}\right)$ are the pressure functions given by (\ref{solution-v0}), is a solution of { nonlinear} problem (\ref{Poisson-transmission-zero-order1}) in the space ${\mathcal X}_{s}$.
Moreover, the second expression in \eqref{Newtonian-D-B-F-new9n} yields
\begin{align}
\label{estimate-ms}
&C_{*}\tilde{C}_1\left\|\left({\bf u}_{+},{\bf u}_{-}\right)\right\|_{H_{\delta}^{s+\frac{1}{2}}({\Omega }_{+},\Lambda ^1TM)\times H_{\delta}^{s+\frac{1}{2}}({\Omega }_{-},\Lambda ^1TM)}^2\nonumber\\
&\le C_{*}\tilde{C}_1\eta\left\|\left({\bf u}_{+},{\bf u}_{-}\right)\right\|_{H_{\delta}^{s+\frac{1}{2}}({\Omega }_{+},\Lambda ^1TM)\times H_{\delta}^{s+\frac{1}{2}}({\Omega }_{-},\Lambda ^1TM)}\nonumber\\
&=\frac{1}{4}\left\|\left({\bf u}_{+},{\bf u}_{-}\right)\right\|_{H_{\delta}^{s+\frac{1}{2}}({\Omega }_{+},\Lambda ^1TM)\times H_{\delta}^{s+\frac{1}{2}}({\Omega }_{-},\Lambda ^1TM)}.
\end{align}
Finally, by inequalities \eqref{estimate-D-B-F-new1-new1-D-Rn} and \eqref{estimate-ms} we obtain
\begin{align*}
&{\left\|\left(({\bf u}_{+},\pi _{+}),({\bf u}_{-},\pi _{-})\right)\right\|_{{\mathcal X}_{s}}}\nonumber\\
&\leq C_*\big\|\big(\tilde{\bf f}_{+},\tilde{\bf f}_{-},{\bf h},{\bf r}\big)\big\|_{{\mathcal Y}_{s}}+\frac{1}{4}\|({\bf u}_{+},{\bf u}_{-})\|_{H^{s+\frac{1}{2}}({\Omega }_{+},\Lambda ^1TM)\times H^{s+\frac{1}{2}}({\Omega }_{-},\Lambda ^1TM)},
\end{align*}
and hence
$$\|({\bf u}_{+},{\bf u}_{-})\|_{H^{s+\frac{1}{2}}({\Omega }_{+}\Lambda ^1TM)\times H^{s+\frac{1}{2}}({\Omega }_{-}\Lambda ^1TM)}\leq\dfrac{4}{3}C_*\big\|\big(\tilde{\bf f}_{+},\tilde{\bf f}_{-},{\bf h},{\bf r}\big)\big\|_{{\mathcal Y}_{s}},$$
i.e., inequality \eqref{estimate-D-B-F-new1-new1-D-Rn2}, with the constant $C=\dfrac{4}{3}C_*$.

$\bullet $ {\bf Uniqueness result for the {nonlinear} problem (\ref{Poisson-transmission-zero-order1})}

We now show the uniqueness of a solution $\left(({\bf u}_{+},\pi _{+}),({\bf u}_{-},\pi _{-})\right)\in {\mathcal X}_{s}$ of the nonlinear transmission problem (\ref{Poisson-transmission-zero-order1}), such that $({\bf u}_{+},{\bf u}_{-})\in \mathbf{B}_{\eta }${, when condition \eqref{cond-small} is satisfied}. To this end, we assume that $\left(({\bf u}_{+}',{\pi}_{+}'),({\bf u}_{-}',{\pi}_{-}')\right)\!\!\in \!{\mathcal X}_{s}$ is another solution of problem \eqref{Poisson-transmission-zero-order1}, satisfying the condition $({\bf u}_{+}',{\bf u}_{-}')\in \mathbf{B}_{\eta }$. This assumption implies that $\left({\mathcal T}_{+}({\bf u}_{+}',{\bf u}_{-}'),{\mathcal T}_{-}({\bf u}_{+}',{\bf u}_{-}')\right)\in \mathbf{B}_{\eta },$ as $\left({\mathcal T}_{+},{\mathcal T}_{-}\right)$ maps $\mathbf{B}_{\eta }$ to $\mathbf{B}_{\eta }$. We note that the elements $$\left({\mathcal T}_{+}({\bf u}_{+}',{\bf u}_{-}'),P_{+}({\bf u}_{+}',{\bf u}_{-}'),{\mathcal T}_{-}({\bf u}_{+}',{\bf u}_{-}'),P_{-}({\bf u}_{+}',{\bf u}_{-}')\right)$$ are given by (\ref{solution-v0}) and satisfy problem (\ref{Newtonian-D-B-F-new4n}) with ${\bf u}_{\pm }$ replaced by ${\bf u}'_{\pm }$.
Then by (\ref{Poisson-Brinkman-smn-1}) and (\ref{Newtonian-D-B-F-new4n}) we obtain the linear transmission problem
\begin{equation}
\left\{\begin{array}{lll}
({\bf L}+V)\left({\mathcal T}_+({\bf u}_{+}',{\bf u}_{-}')-{\bf u}'_{+}\right)+
d\left(P_+({\bf u}_{+}',{\bf u}_{-}')\!-\!{\pi}_{+}'\right)={\bf 0}\ \mbox{ in }\ \Omega _{+},\\
{\delta {\mathcal T}_+({\bf u}_{+}',{\bf u}_{-}')=0\ \mbox{ in }\ \Omega _{+},}\\
{\bf L}\left({\mathcal T}_{-}({\bf u}_{+}',{\bf u}_{-}')-{\bf u}'_{-}\right)
+d\left(P_-({\bf u}_{+}',{\bf u}_{-}')\!-\!{\pi}_{-}'\right)={\bf 0}\ \mbox{ in }\ \Omega _{-},\\
{\delta {\mathcal T}_{-}({\bf u}_{+}',{\bf u}_{-}')=0\ \mbox{ in }\ \Omega _{-},}\\
\mu \gamma _{+}\left({\mathcal T}_+({\bf u}_{+}',{\bf u}_{-}')-{\bf u}'_{+}\right)-
\gamma _{-}\left({\mathcal T}_{-}({\bf u}_{+}',{\bf u}_{-}')-{\bf u}'_{-}\right)
={\bf 0}\ \mbox{ on }\ \partial \Omega ,\\
{\bf t}_{V}^{+}\left({\mathcal T}_+({\bf u}_{+}',{\bf u}_{-}')
-{\bf u}'_{+},P_+({\bf u}_{+}',{\bf u}_{-}')-{\pi}_{+}'\right)
\\
\hspace{1.5em}
-{{\bf t}_0^-}\left({\mathcal T}_{-}({\bf u}_{+}',{\bf u}_{-}')-{\bf u}'_{-},P_-({\bf u}_{+}',{\bf u}_{-}')-{\pi}_{-}'\right)
\\
\hspace{1.5em}
+{\mathcal P} \gamma _{+}\left({\mathcal T}_+({\bf u}_{+}',{\bf u}_{-}')-{\bf u}'_{+}\right)={\bf 0} \mbox{ on } \partial \Omega ,
\end{array} \right.
\nonumber
\end{equation}
which, has only the trivial solution in the space ${\mathcal X}_{s}$, due to the well-posedness result in Theorem \ref{Poisson-transmission-Stokes}. Therefore,
$$\left({\mathcal T}_+({\bf u}_{+}',{\bf u}_{-}'),{\mathcal T}_{-}({\bf u}_{+}',{\bf u}_{-}')\right)=\left({\bf u}_{+}',{\bf u}_{-}'\right), \ \ P_\pm({\bf u}_{+}',{\bf u}_{-}')={\pi}_{\pm }'.$$
Consequently, $({\bf u}_{+}',{\bf u}_{-}')\in \mathbf{B}_{\eta }$ is a fixed point of the contraction mapping $\left({\mathcal T}_{+},{\mathcal T}_{-}\right):\mathbf{B}_{\eta }\to \mathbf{B}_{\eta }$. Since such a fixed point is unique in $\mathbf{B}_{\eta }$, we deduce that $\left({\bf u}_{+}',{\bf u}_{-}'\right)=\left({\bf u}_{+},{\bf u}_{-}\right)$. In addition, we have $\pi'_\pm=\pi_\pm$. This shows the uniqueness of the solution of the nonlinear transmission problem (\ref{Poisson-transmission-zero-order1}) in $\mathbf{B}_{\eta }$.

Due to, e.g., \cite[Chapter XVI, \S 1, Theorem 3]{Ka-Ak}, the unique solution $\left(({\bf u}_{+},\pi _{+}),({\bf u}_{-},\pi _{-})\right)$ of the transmission problem (\ref{Poisson-transmission-zero-order1}) depends continuously on the data $\big(\tilde{\bf f}_{+},\tilde{\bf f}_{-},{\bf h},{\bf r}\big)\in {\mathcal Y}_{s}$. Indeed, this solution is expressed in terms of the unique fixed point of the contraction $\left({\mathcal T}_{+},{\mathcal T}_{-}\right):\mathbf{B}_{\eta }\to \mathbf{B}_{\eta }$, which is a continuous map with respect to the given data, and this continuity property is an immediate consequence of the continuity of the map ${\mathcal S}_V:{\mathcal Y}_{s}\to {\mathcal X}_{s}$.
\end{proof}

\appendix
{
{ 
}
\section{Mapping properties of layer potentials for the Stokes system in compact Riemannian manifolds}
\label{appendix-layer-potential}
}

{Next we present mapping properties of Stokes layer potentials that have been used to obtain the main results of this paper.}
{First, we note the following result (cf., e.g., \cite[Theorem 5.6]{D-M} for $V=0$; see also \cite[Theorem 4.8]{10-new2}).}
\begin{lem}
\label{Dirichlet}
{Let $M$ satisfy {Assumption} $\ref{H}$ and $\dim (M)\geq 2$}. Let $\Omega \subset M$ be a Lipschitz domain and $V\in L^{\infty }(M,\Lambda ^1TM\otimes \Lambda ^1TM)$ satisfying the positivity condition \eqref{V-positive}. Let $s\in (0,1)$. Then the Poisson problem for the generalized Brinkman system with Dirichlet boundary condition
\begin{equation}
\label{Dirichlet-V}
\left\{\begin{array}{lll}
{\bf L}{\bf v}+V{\bf v}+dp={\bf F}\in H^{s-\frac{3}{2}}(\Omega ,\Lambda ^1TM) \mbox{ in } \Omega ,\\
\delta {\bf v}=0 \mbox{ on } \Omega ,\\
{\gamma _{+}{\bf v}}={\bf g}\in H_{\nu }^{s}(\partial \Omega ,\Lambda ^1TM)
\end{array}
\right.
\end{equation}
is well-posed, i.e., it has a unique solution $({\bf v},p)\in H_\delta ^{s+\frac{1}{2}}(\Omega _{+},\Lambda ^1TM)\times H_*^{s-\frac{1}{2}}(\Omega _{+})$, and there exists a constant $C_V\equiv C_V(s,\Omega )>0$ such that
\begin{equation}
\label{Dirichlet-V-1}
\|{\bf v}\|_{H_\delta ^{s+\frac{1}{2}}(\Omega ,\Lambda ^1TM)}+\|p\|_{H_*^{s-\frac{1}{2}}(\Omega )}\leq C_V\|({\bf F},{\bf g})\|_{H^{s-\frac{3}{2}}(\Omega ,\Lambda ^1TM)\times H_{\nu }^{s}(\partial \Omega ,\Lambda ^1TM)}.
\end{equation}
\end{lem}
\begin{proof}
By using similar arguments to those in the proof of \cite[Theorem 4.8]{10-new2}, we consider the spaces
\begin{align}
\label{spaces-Poisson-transmission}
&H^{s+\frac{1}{2}}_{\delta }(\Omega ,\Lambda ^1TM):=\left\{{\bf u}\in H^{s+\frac{1}{2}}(\Omega ,\Lambda ^1TM):\delta {\bf u}=0 \mbox{ in } \Omega \right\},\\
&D_{s;\delta }:=H_\delta ^{s+\frac{1}{2}}(\Omega ,\Lambda ^1TM)\times H_*^{s-\frac{1}{2}}(\Omega ),\\
&R_{s;\delta }:=H^{s-\frac{3}{2}}(\Omega ,\Lambda ^1TM)\times H_\nu ^{s}(\partial \Omega ,\Lambda
^1TM)
\end{align}
and the operator
\begin{equation}
\label{Dirichlet-V-2}
{\mathcal F}_{V}:D_{s;\delta }\to R_{s;\delta },\ {\mathcal F}_{V}:=\left(
\begin{array}{cc}
{\bf L}+V & d\\
\gamma _{+} & 0
\end{array}
\right),
\end{equation}
which is associated to the Poisson problem \eqref{Dirichlet-V}. We note that
\begin{align}
\label{Dirichlet-V-3}
{\mathcal F}_{V}={\mathcal F}_{0}+{\mathcal F}_{V;0},
\end{align}
where
\begin{equation}
\label{Dirichlet-V-3}
{\mathcal F}_{0}:D_{s;\delta }\to R_{s;\delta },\ {\mathcal F}_{0}:=\left(
\begin{array}{cc}
{\bf L} & d\\
\gamma _{+} & 0
\end{array}
\right)
\end{equation}
is the operator describing the Poisson problem for the Stokes system, and
\begin{equation}
\label{Dirichlet-V-4}
{\mathcal F}_{V;0}:D_{s;\delta }\to R_{s;\delta },\ {\mathcal F}_{V;0}:=\left(
\begin{array}{cc}
V & 0\\
0 & 0
\end{array}
\right).
\end{equation}
The operator ${\mathcal F}_{0}$ is an isomorphism for any $s\in (0,1)$, due to the well-posedness of the Poisson problem for the Stokes system (cf. \cite[Theorem 5.6]{D-M}). Operator \eqref{Dirichlet-V-4} is compact due to the compactness of the embedding $L^2(\Omega ,\Lambda ^1TM)\hookrightarrow H^{s-\frac{3}{2}}(\Omega ,\Lambda ^1TM)$. Thus, ${\mathcal F}_{V}:D_{s;\delta }\to R_{s;\delta }$ is a Fredholm operator with index zero for any $s\in (0,1)$. Then, by Lemma \ref{Fredholm-kernel},
\begin{align}
\label{Dirichlet-V-5}
{\rm{Ker}}\left\{{\mathcal F}_{V}:D_{s;\delta }\to R_{s;\delta }\right\}=
{\rm{Ker}}\left\{{\mathcal F}_{V}:D_{\frac{1}{2};\delta }\to R_{\frac{1}{2};\delta }\right\},\ \forall \ s\in (0,1).
\end{align}

{Let us} assume that $({\bf v}_0,p_0)\in {\rm{Ker}}\left\{{\mathcal F}_{V}:D_{\frac{1}{2};\delta }\to R_{\frac{1}{2};\delta }\right\}$. By using the Green formula \eqref{conormal-derivative-generalized-2} and the condition $\gamma _{+}{\bf v}_0={\bf 0}$ on $\partial \Omega $, we obtain that
\begin{align}
\label{Dirichlet-V-7}
2\langle {\rm{Def}}\ {\bf v}_{0},&{\rm{Def}}\ {\bf v}_{0}\rangle _{\Omega _{+}}
{+\left\langle V{\bf v}_0,{\bf v}_0\right\rangle _{\Omega _{+}}}=
\langle {{\bf t}_0^+}({\bf v}_{0},p_{0}),\gamma _{+}{\bf v}_{0}\rangle _{\partial \Omega }=0,
\end{align}
and hence ${\rm{Def}}\ {\bf v}_{0}=0$ in $\Omega $. Since $\gamma _{+}{\bf v}_{0}={\bf 0}$, i.e., ${\bf v}_{0}\in H_0^1(\Omega ,\Lambda ^1TM)$, and the operator ${\rm{Def}}:H_0^1(\Omega ,\Lambda ^1TM)\to L^2(\Omega ,S^2T^*M)$ is one-to-one (cf. \cite[(6.17)]{D-M}), we deduce that ${\bf v}_{0}={\bf 0}$ in $\Omega $. Then by the Brinkman equation and assumption $p_0\in L_*^2(\Omega )$ we obtain $p_0=0$ in $\Omega $. Hence, $({\bf v}_0,p_0)=({\bf 0},0)$, i.e.,
\begin{align}
\label{Dirichlet-V-6}
{\rm{Ker}}\left\{{\mathcal F}_{V}:D_{\frac{1}{2};\delta }\to R_{\frac{1}{2};\delta }\right\}=\{({\bf 0},0)\}.
\end{align}
Then by relation \eqref{Dirichlet-V-5} {it} follows that the Fredholm operator with index zero ${\mathcal F}_{V}:D_{s;\delta }\to R_{s;\delta }$ is an isomorphism for any $s\in (0,1)$. Consequently, the Poisson problem \eqref{Dirichlet-V} is well-posed, i.e., it has a unique solution $({\bf v},p)\in D_{V;\delta }$, which satisfies an estimate of type \eqref{Dirichlet-V-1}, due to the continuity of the inverse operator ${\mathcal F}_{V}^{-1}:R_{s;\delta }\to D_{s;\delta }$.
\end{proof}

The next theorem presents some of the main mapping properties of the layer potentials for the Stokes system (\cite[Proposition 4.2.5, 4.2.9, Corollary 4.3.2, Theorems 5.3.6, 5.4.1, 5.4.3, 10.5.3]{M-W},
\cite[Theorem 2.1, (3.5), Proposition 3.5]{D-M}, \cite[Theorems 3.1, 6.1]{M-T}, and \cite[Theorems 4.3, 4.9, 4.11, (131), (132), (137), Lemma 5.4]{10-new2} for a pseudodifferential Brinkman operator).
\begin{thm}
\label{single-layer-operator-Brinkman} {Let $M$ satisfy Assumption $\ref{H}$ and $\dim (M)\geq 2$}. Let $\Omega _{+}:=\Omega \subset M$ be a Lipschitz domain. {Let $\Omega _{-}=M\setminus \overline{\Omega }$ satisfy Assumption $\ref{connected-set}$.} Let $s\in (0,1)$.
\begin{itemize}
\item[$(i)$] Then the following operators are linear and bounded,
\begin{align}
\label{ss-s1}
&\left(\!{\bf V}_{\partial\Omega }\!\right)\!|_{\Omega _{\pm }}:H^{s-1}(\partial\Omega ,\Lambda ^1TM)\!\to \!H^{s+\frac{1}{2}}(\Omega _{\pm },\Lambda ^1TM)\nonumber\\
&\left(\!{\mathcal Q_{\partial \Omega}}\!\right)\!|_{\Omega _{\pm }}:H^{s-1}(\partial\Omega ,\Lambda ^1TM)\!\to \!H^{s-\frac{1}{2}}(\Omega _{\pm }),
\\
\label{ds-s1}
&\left(\!{\bf W}_{\partial\Omega }\!\right)\!|_{\Omega _{\pm }}:H^{s}(\partial\Omega ,\Lambda ^1TM)\to H^{s+\frac{1}{2}}(\Omega _{\pm },\Lambda ^1TM)\nonumber\\
&\left({\mathcal P}_{\partial\Omega }\right)|_{\Omega _{\pm }}:H^{s}(\partial\Omega ,\Lambda ^1TM)\to H^{s-\frac{1}{2}}(\Omega _{\pm }).
\end{align}
\item[$(ii)$]
Let ${\boldsymbol\psi}\in H^{s-1}(\partial \Omega ,\Lambda ^1TM)$ and ${\boldsymbol \phi}\in H^s_{\nu }(\partial \Omega ,\Lambda ^1TM)$ be given. Then the following relations hold a.e. on $\partial\Omega $,
\begin{align}
\label{single-layer-Brinkman8a}
&\gamma _+\big({\bf V}_{\partial\Omega}{\boldsymbol\psi}\big)=\gamma _{-}\big({\bf V}_{\partial\Omega}{\boldsymbol\psi}\big)=:{\mathcal V}_{\partial\Omega}{\boldsymbol\psi},\\
\label{double-layer-Brinkman12}
&\frac{1}{2}{\boldsymbol \phi}+\gamma _{+}({\bf W}_{\partial\Omega }{\boldsymbol \phi})=
{-}\frac{1}{2}{\boldsymbol \phi}+\gamma _{-}({\bf W}_{\partial\Omega }{\boldsymbol \phi})
=:{\bf K}_{{\partial\Omega }}{\boldsymbol \phi},\\
\label{eq:dlB13}
&{-}\frac{1}{2}{\bf f}+{{\bf t}^+_0}\left({\bf V}_{\partial\Omega}{\boldsymbol\psi},
{\mathcal Q_{\partial \Omega}\boldsymbol\psi}\right)
=\frac{1}{2}{\bf f}+{{\bf t}^-_0}\left({\bf V}_{\partial\Omega}{\boldsymbol\psi},{\mathcal Q_{\partial \Omega}\boldsymbol\psi}\right)=:{\bf K}_{{\partial\Omega } }^*{\bf f},\\
\label{double-layer-Brinkman14b}
&{{\bf D}_{\partial \Omega }^{+}{\boldsymbol \phi}-{\bf D}_{\partial \Omega }^{-}{\boldsymbol \phi}\in {\mathbb R}\nu ,}
\end{align}
where ${\bf D}_{\partial \Omega }^{\pm }{\boldsymbol \phi}:={\bf t}_0^{\pm }\left({\bf W}_{\partial \Omega }{\boldsymbol \phi},{\mathcal P}_{\partial \Omega }{\boldsymbol \phi}\right),$ and
${\bf K}_{\partial \Omega }^*$
is the transpose of the double-layer potential operator ${\bf K}_{\partial \Omega }$. In addition, the operators
\begin{align}
\label{ss-s2}
&{\mathcal V}_{\partial\Omega }:H^{s-1}(\partial\Omega ,\Lambda ^1TM)\to H^{s}(\partial\Omega ,\Lambda ^1TM),\nonumber\\
&{\bf K}_{\partial\Omega }:H^{s}(\partial\Omega ,\Lambda ^1TM)\to H^{s}(\partial\Omega ,\Lambda ^1TM),
\\
\label{ds-s2}
&{\bf K}_{\partial\Omega }^*\!:\!H^{s-1}(\partial\Omega ,\Lambda ^1TM)\!\to \!
H^{s-1}(\partial\Omega ,\Lambda ^1TM)\nonumber\\
&{\bf D}_{\partial\Omega }\!:\!H^{s}(\partial\Omega ,\Lambda ^1TM)\!\to \!H^{s-1}(\partial\Omega ,\Lambda ^1TM)
\end{align}
are linear and bounded, ${\bf V}_{\partial \Omega }\nu =0,\, {{\mathcal Q}_{\partial \Omega }}\nu =c_{\pm}\in {\mathbb R}$ {in} $\Omega _{\pm}$, and
\begin{align}
\label{Fredholm-single-layer1}
&{\rm{Ker}}\left\{{\mathcal V}_{\partial \Omega }:H^{s-1}(\partial \Omega ,\Lambda ^1TM)\to H^{s}(\partial \Omega ,\Lambda ^1TM)\right\}={\mathbb R}\nu ,
\end{align}
where ${\mathbb R}\nu :=\{c\nu :c\in {\mathbb R}\}.$
\item[$(iii)$]
Let ${\mathcal P}\in L^{\infty }(\partial \Omega ,\Lambda ^1TM\otimes \Lambda ^1TM)$ be a {symmetric tensor field}\footnote{${\mathcal P}\in L^{\infty }(\partial \Omega ,\Lambda ^1TM\otimes \Lambda ^1TM)$ defines a multiplication operator denoted also by ${\mathcal P}$. Consequently, ${\mathcal P}:L^2(\partial \Omega ,\Lambda ^1TM)\to L^2(\partial \Omega ,\Lambda ^1TM)$ is a linear and continuous operator.} with the property that there exists a constant $C_{\mathcal P}>0$ such that
\begin{equation}
\label{Poisson-transmission3}
\begin{array}{ll}
\langle {\mathcal P}{\bf v},{\bf v}\rangle _{\partial
\Omega }\geq C_{\mathcal P}\|{\bf v}\|_{L^2(\partial \Omega ,\Lambda ^1TM)}^2,\ \forall \ {\bf v}\in L^2(\partial \Omega ,\Lambda ^1TM).
\end{array}
\end{equation}
Let {$\mu \in (0,\infty )\setminus \{1\}$} be a given constant. Then the following operators
are isomorphisms
\begin{align}
\label{Fredholm-sm-ms}
&\frac{1}{2}(1\!+\!\mu ){\mathbb I}\!+\!(1\!-\!\mu ){\bf K}_{\partial \Omega }\!+\!{\mathcal V}_{\partial \Omega }{\mathcal P}:H^s_{\nu }(\partial\Omega ,\Lambda ^1TM)\!\to \!H^s_{\nu }(\partial\Omega ,\Lambda ^1TM),\\
\label{Fredholm-adj-sm-ms}
&\frac{1}{2}(1+\mu ){\mathbb I}+(1-\mu ){\bf K}_{\partial \Omega }^*+{\mathcal P}{\mathcal V}_{\partial \Omega }:H^{s-1}(\partial\Omega ,\Lambda ^1TM)/{{\mathbb R}\nu }\nonumber\\
&\hspace{15em}\to H^{s-1}(\partial\Omega ,\Lambda ^1TM)/{{\mathbb R}\nu }.
\end{align}
\end{itemize}
\end{thm}
\begin{proof}
All mapping properties mentioned in (i) and (ii) for the layer potential operators of the Stokes system in Sobolev spaces on Lipschitz domains in the Euclidean setting {or} on compact Riemannian manifolds, as well as their jump relations across a Lipschitz boundary, mentioned in (ii), are well-known (see, e.g., \cite[Propositions 3.3 and 3.4]{D-M}, \cite[Theorem 3.1]{M-T} {and \cite{M-W}}). Next we show the mapping properties in (iii).

Let $\mu \in (0,1)$ be fixed. Let us show that operators \eqref{Fredholm-sm-ms} and \eqref{Fredholm-adj-sm-ms} are isomorphisms for any $s\in (0,1)$. First, we note that the operators ${\bf K}_{\partial \Omega }:L^2(\partial\Omega ,\Lambda ^1TM)\to L^2(\partial\Omega ,\Lambda ^1TM)$, ${\bf K}_{\partial \Omega }:H^1(\partial\Omega ,\Lambda ^1TM)\to H^1(\partial\Omega ,\Lambda ^1TM)$ are bounded (see \cite[(3.50)]{D-M}), and the operators
\begin{align}
\label{Fredholm1}
&\pm \frac{1}{2}\frac{1+\mu }{1-\mu }{\mathbb I}+{\bf K}_{\partial \Omega }:L^2(\partial\Omega ,\Lambda ^1TM)\to L^2(\partial\Omega ,\Lambda ^1TM),\\
\label{Fredholm2}
&\pm \frac{1}{2}\frac{1+\mu }{1-\mu }{\mathbb I}+{\bf K}_{\partial \Omega }:H^{1}(\partial\Omega ,\Lambda ^1TM)\to H^{1}(\partial\Omega ,\Lambda ^1TM)
\end{align}
are Fredholm with index zero {(see \cite[Lemma 5.3]{10-new2})}, where ${\bf K}_{\partial \Omega }{\boldsymbol \phi}:={\bf K}_{0;\partial \Omega }{\boldsymbol \phi}$ is defined in \eqref{double-layer-Brinkman3} for any ${\boldsymbol \phi}\in H^r(\partial \Omega ,\Lambda ^1TM)$, $r\in \{0,1\}$, and, in addition, ${\bf K}_{\partial \Omega }^*$ is the adjoint of ${\bf K}_{\partial \Omega }$.
Then by an interpolation argument, as in the proof of \cite[Theorem 10.5.3]{M-W}, the operators
\begin{align}
\label{Fredholm3-new-S}
&{\pm \frac{1}{2}\frac{1+\mu }{1-\mu }{\mathbb I}+{\bf K}_{\partial \Omega }:H^s(\partial\Omega ,\Lambda ^1TM)\to H^s(\partial\Omega ,\Lambda ^1TM),}
\end{align}
are Fredholm with index zero as well, for any $s\!\in \!(0,1)$. Moreover, the continuity of the operator
${\mathcal V}_{\partial \Omega }{\mathcal P}:L^{2}(\partial\Omega ,\Lambda ^1TM)\to H^1(\partial\Omega ,\Lambda ^1TM)$ and the compactness of the embeddings $H^1(\partial\Omega ,\Lambda ^1TM)\hookrightarrow H^s(\partial\Omega ,\Lambda ^1TM)\hookrightarrow L^2(\partial \Omega ,\Lambda ^1TM)$
imply that the operators
\begin{align}
\label{Fredholm1-new-S}
&\frac{1}{2}\frac{1+\mu }{1-\mu }{\mathbb I}\!+\!{\bf K}_{\partial \Omega }\!+\!\frac{1}{1-\mu }{\mathcal V}_{\partial \Omega }{\mathcal P}:L^2(\partial\Omega ,\Lambda ^1TM)\!\to \! L^2(\partial\Omega ,\Lambda ^1TM),
\\
\label{Fredholm3-new-SM}
&{\frac{1}{2}\frac{1+\mu }{1-\mu }{\mathbb I}\!+\!{\bf K}_{\partial \Omega }+\frac{1}{1-\mu }{\mathcal V}_{\partial \Omega }{\mathcal P}\!:\!H^s(\partial\Omega ,\Lambda ^1TM)\!\to \! H^s(\partial\Omega ,\Lambda ^1TM),}
\end{align}
are also Fredholm with index zero for any $s\in (0,1)$.
Now, the property that
\begin{equation}
\label{Sobolev-normal-1} H^r_{\nu }(\partial \Omega ,\Lambda ^1TM):=
\left\{{\boldsymbol \phi}\in H^r(\partial \Omega ,\Lambda ^1TM): \langle \nu ,{\boldsymbol \phi}\rangle _{\partial \Omega }=0\right\}
\end{equation}
is a closed subspace of $H^r(\partial \Omega ,\Lambda ^1TM)$ for any $r\in [0,1)$ (see, e.g., \cite[Proposition 10.6]{Agr-1})
implies that
\begin{align}
\label{Fredholm3}
&{\frac{1}{2}\frac{1+\mu }{1-\mu }{\mathbb I}\!+\!{\bf K}_{\partial \Omega }+\frac{1}{1-\mu }{\mathcal V}_{\partial \Omega }{\mathcal P}:L_{\nu }^2(\partial\Omega ,\Lambda ^1TM)\!\to \! L_{\nu }^2(\partial\Omega ,\Lambda ^1TM),}\\
\label{Fredholm4}
&{\frac{1}{2}\frac{1+\mu }{1-\mu }{\mathbb I}\!+\!{\bf K}_{\partial \Omega }\!+\!\frac{1}{1-\mu }{\mathcal V}_{\partial \Omega }{\mathcal P}:H_\nu ^{s}(\partial\Omega ,\Lambda ^1TM)\to H_\nu ^{s}(\partial\Omega ,\Lambda ^1TM)}
\end{align}
are Fredholm operators with zero index as well,
for any $s\in (0,1)$. Note that\\ $L^2_{\nu }(\Omega ,\Lambda ^1TM):=H^0_{\nu }(\partial \Omega ,\Lambda ^1TM)$. Moreover, the operator
\begin{align}
\label{Fredholm5}
\!\!\!\!\!\frac{1}{2}\frac{1\!+\!\mu }{1\!-\!\mu }{\mathbb I}\!+\!\!{\bf K}_{\partial \Omega }^*\!+\!\!\frac{1}{1\!-\!\mu }{\mathcal P}{\mathcal V}_{\partial \Omega }\!:\!L^2(\partial\Omega ,\Lambda ^1TM)/{{\mathbb R}\nu }\!\to \!L^2(\partial\Omega ,\Lambda ^1TM)/{{\mathbb R}\nu }
\end{align}
is Fredholm with index zero, as the adjoint of operator \eqref{Fredholm3} and due to the relation $L^2(\partial\Omega ,\Lambda ^1TM)/{{\mathbb R}\nu }=(L^2_{\nu }(\Omega ,\Lambda ^1TM))'$ (cf., e.g., \cite[(5.118)]{M-W}).

Next we show that the operator in \eqref{Fredholm5} is one-to-one by using a similar argument to that in the proof of \cite[Lemma 5.4]{10-new2}. To this aim, we assume that $[{\boldsymbol \varphi} ]\in L^2(\partial\Omega ,\Lambda ^1TM)/{{\mathbb R}\nu }$ (i.e., $[{\boldsymbol \varphi}]={\boldsymbol \varphi} +{\mathbb R}\nu $, where ${\boldsymbol \varphi} \in L^2(\partial\Omega ,\Lambda ^1TM)$) belongs to the null space of this operator. Therefore, due to the properties \eqref{Fredholm-single-layer1} and ${\bf K}_{\partial \Omega }^*\nu \in {\mathbb R}\nu $ (see, e.g., \cite[(3.6)]{D-M}), we obtain the equivalence
\begin{align}
\label{Fredholm6}
&{\bf K}_{\partial \Omega }^*[{\boldsymbol \varphi}]=-\frac{1}{2}\frac{1+\mu }{1-\mu }[{\boldsymbol \varphi}]-\frac{1}{1-\mu }{\mathcal P}{\mathcal V}_{\partial \Omega }[{\boldsymbol \varphi}]
\Longleftrightarrow \nonumber
\\
&{\bf K}_{\partial \Omega }^*{\boldsymbol \varphi}=-\frac{1}{2}\frac{1+\mu }{1-\mu }{\boldsymbol \varphi} -\frac{1}{1-\mu }{\mathcal P}{\mathcal V}_{\partial \Omega }{\boldsymbol \varphi} +c\nu ,
\end{align}
with some $c\in {\mathbb R}$.
Then the fields ${\bf u}_{\pm }\!:=\!\left({\bf V}_{\partial \Omega }{\boldsymbol \varphi} \right)|_{\Omega _{\pm }}\!\in \!H^1(\Omega _{\pm },\Lambda ^1TM)$ and $\pi _{\pm }\!:=\!\left({\mathcal Q}_{\partial \Omega }^s{\boldsymbol \varphi} \right)|_{\Omega _{\pm }}\!\in \!L^2(\Omega _{\pm })$ satisfy the {homogeneous} Stokes system in $\Omega _{\pm }$ and the relations
\begin{align}
\label{Fredholm7}
&\gamma _{+}{\bf u}_{+}=\gamma _{-}{\bf u}_{-}={\mathcal V}_{\partial \Omega }{\boldsymbol \varphi} ,\\
\label{Fredholm8}
&{{\bf t}_0^+}({\bf u}_{+},\pi _{+})=\left(\frac{1}{2}{\mathbb I}+{\bf K}_{\partial \Omega }^*\right){\boldsymbol \varphi} =-\frac{\mu }{1-\mu }{\boldsymbol \varphi} -\frac{1}{1-\mu }{\mathcal P}{\mathcal V}_{\partial \Omega }{\boldsymbol \varphi} +c\nu ,\\
&{{\bf t}_0^-}({\bf u}_{-},\pi _{-} )=\left(-\frac{1}{2}{\mathbb I}+{\bf K}_{\partial \Omega }^*\right){\boldsymbol \varphi} =-\frac{1}{1-\mu }{\boldsymbol \varphi} -\frac{1}{1-\mu }{\mathcal P}{\mathcal V}_{\partial \Omega }{\boldsymbol \varphi} +c\nu
\end{align}
on $\partial \Omega $, due to formulas \eqref{eq:dlB13}. Moreover, by using the Green formula (\ref{conormal-derivative-generalized-2}) (with $V=0$) in each of the domains $\Omega _{+}$ and $\Omega _{-}$, as well as the property that $\langle {\mathcal V}_{\partial \Omega }{\boldsymbol \varphi} ,\nu \rangle _{\partial \Omega }=\langle {\boldsymbol \varphi} ,{\mathcal V}_{\partial \Omega }\nu \rangle _{\partial \Omega }=0$, we obtain the relations
\begin{align}
\label{P0-3-Fredhlm}
2\langle {\rm{Def}}\ {\bf u}_{+},{\rm{Def}}\ {\bf u}_{+}\rangle_{\Omega _{+}}&=
\langle {{\bf t}_0^+}({\bf u}_{+},\pi _{+}),\gamma _{+}{\bf u}_{+}\rangle _{\partial \Omega }
\\
&=-\frac{\mu }{1-\mu }\langle {\boldsymbol \varphi} ,{\mathcal V}_{\partial \Omega }{\boldsymbol \varphi} \rangle _{\partial \Omega }
-\frac{1}{1-\mu }\langle {\mathcal P}{\mathcal V}_{\partial \Omega }{\boldsymbol \varphi} ,{\mathcal V}_{\partial \Omega }{\boldsymbol \varphi} \rangle _{\partial \Omega },\nonumber
\\
\label{P0-4-Fredholm}
2\langle {\rm{Def}}\ {\bf u}_{-},{\rm{Def}}\ {\bf u}_{-}\rangle_{\Omega _{-}}&=
-\langle {{\bf t}_0^-}({\bf u}_{-},\pi _{-}),\gamma _{-}{\bf u}_{-}\rangle _{\partial \Omega }
\\
&=\frac{1}{1-\mu }\langle {\boldsymbol \varphi} ,{\mathcal V}_{\partial \Omega }{\boldsymbol \varphi} \rangle _{\partial \Omega }+\frac{1}{1-\mu }\langle {\mathcal P}{\mathcal V}_{\partial \Omega }{\boldsymbol \varphi} ,{\mathcal V}_{\partial \Omega }{\boldsymbol \varphi} \rangle _{\partial \Omega }.\nonumber
\end{align}
By \eqref{P0-3-Fredhlm}, \eqref{P0-4-Fredholm} and \eqref{Poisson-transmission3} we obtain that
\begin{align}
\label{uniqueness-2NS-F}
2\langle {\rm{Def}}\, {\bf u}_{+},{\rm{Def}}\, {\bf u}_{+}\rangle _{\Omega _{+}}\!\!+\!2\mu \langle {\rm{Def}}\, {\bf u}_{-},{\rm{Def}}\, {\bf u}_{-}\rangle _{\Omega _{-}}
&\!=\!-\langle {\mathcal P}{\mathcal V}_{\partial \Omega }{\boldsymbol \varphi},{\mathcal V}_{\partial \Omega }{\boldsymbol \varphi} \rangle _{\partial \Omega }\\
&\leq -C_{\mathcal P}\|{\mathcal V}_{\partial \Omega }{\boldsymbol \varphi} \|_{L^2(\partial \Omega ,\Lambda ^1TM)}^2,\nonumber
\end{align}
{implying}, in particular, that ${\mathcal V}_{\partial \Omega }{\boldsymbol \varphi} ={\bf 0}$ on $\partial \Omega $. Then by property \eqref{Fredholm-single-layer1} (see also \cite[(2.27)]{D-M}, \cite[Theorem 6.1]{M-T}) we obtain that
${\boldsymbol \varphi} \in {\mathbb R}\nu $ {and thus} $[{\boldsymbol \varphi}]={\mathbb R}\nu ${, that is, $[{\boldsymbol \varphi}]=[{\bf 0}]$ in $L^2(\partial \Omega ,\Lambda ^1TM)/{{\mathbb R}\nu }$}. Consequently, the Fredholm operator of index zero \eqref{Fredholm5} is one-to-one, i.e., an isomorphism. The adjoint operator
\begin{align}
\label{Fredholm3-new1}
\!\!\!\!\!\frac{1}{2}\frac{1+\mu }{1-\mu }{\mathbb I}+{\bf K}_{\partial \Omega }+\frac{1}{1-\mu }{\mathcal V}_{\partial \Omega }{\mathcal P}:L_{\nu }^2(\partial\Omega ,\Lambda ^1TM)\to L_{\nu }^2(\partial\Omega ,\Lambda ^1TM)
\end{align}
is an isomorphism as well.

In addition, the embedding $H^s_{\nu }(\partial\Omega ,\Lambda ^1TM)\hookrightarrow L^2_{\nu }(\partial\Omega ,\Lambda ^1TM)$ is continuous with dense range for any $s\in (0,1)$. Then, by applying Lemma \ref{Fredholm-kernel}, we deduce that the Fredholm operator with index zero \eqref{Fredholm4} is also one-to-one, and hence is an isomorphism, for any $s\in (0,1)$. The multiplication of this operator with $(1-\mu )$ is isomorphism, as well. Consequently, for any $s\in [0,1)$, operators \eqref{Fredholm-sm-ms} and \eqref{Fredholm-adj-sm-ms} are isomorphisms.

The invertibility of operators \eqref{Fredholm-sm-ms} and \eqref{Fredholm-adj-sm-ms} when $\mu \in (1,\infty )$ is equivalent with the invertibility of the operator
\begin{align}
-\frac{1}{2}\frac{1+\mu ^{-1}}{1-\mu ^{-1}}{\mathbb I}\!+\!{\bf K}_{\partial \Omega }\!-\!\frac{\mu ^{-1}}{1-\mu ^{-1}}{\mathcal V}_{\partial \Omega }{\mathcal P}
:H_\nu ^{s}(\partial \Omega ,\Lambda ^1TM)\to H_\nu ^{s}(\partial \Omega ,\Lambda ^1TM)\nonumber
\end{align}
for any $s\in (0,1)$, which follows with an argument similar to that in the case $\mu \in (0,1)$. We omit the details for the sake of brevity.
\end{proof}

We now prove a counterpart of Theorem \ref{single-layer-operator-Brinkman}(iii) for $\mu =1$.
{\begin{lem}
\label{isom-sm}
{Let $M$ satisfy {Assumption} $\ref{H}$ and $\dim (M)\geq 2$}. Let $\Omega \subset M$ be a Lipschitz domain. Let ${\mathcal P}\in L^{\infty }(\partial \Omega ,\Lambda ^1TM\otimes \Lambda ^1TM)$ be a symmetric tensor field which satisfies condition \eqref{Poisson-transmission3}. Then the following operators
are isomorphisms
\begin{align}
\label{aux}
&{\mathbb I}+{\mathcal P}{\mathcal V}_{\partial \Omega }:L^2(\partial \Omega ,\Lambda ^1TM)/{{\mathbb R}\nu }\to L^2(\partial \Omega ,\Lambda ^1TM)/{{\mathbb R}\nu },\\
\label{aux-adjoint}
&{\mathbb I}+{\mathcal V}_{\partial \Omega }{\mathcal P}:L_\nu ^2(\partial \Omega ,\Lambda ^1TM)\to L_\nu ^2(\partial \Omega ,\Lambda ^1TM).
\end{align}
\end{lem}}
\begin{proof}
First, we note that the map ${\mathcal V}_{\partial \Omega }{\mathcal P}:L_\nu ^2(\partial \Omega ,\Lambda ^1TM)\to H_\nu ^1(\partial \Omega ,\Lambda ^1TM)$ is continuous and the embedding $H_\nu ^1(\partial \Omega ,\Lambda ^1TM)\hookrightarrow L_\nu ^2(\partial \Omega ,\Lambda ^1TM)$ is compact. Hence, the operator in \eqref{aux-adjoint} is Fredholm with index zero. Then, by duality, we deduce that the operator in \eqref{aux} is Fredholm with index zero as well. We now show that {operator \eqref{aux}} is also one-to-one. To this {end}, we consider $[{\boldsymbol \varphi}]\in L^2(\partial \Omega ,\Lambda ^1TM)/{{\mathbb R}\nu }$ such that $\left({\mathbb I}+{\mathcal P}{\mathcal V}_{\partial \Omega }\right)[{\boldsymbol \varphi}]=[{\bf 0}]$. Thus, $[{\boldsymbol \varphi}]={\boldsymbol \varphi} +{\mathbb R}\nu $, ${\boldsymbol \varphi} \in L^2(\partial \Omega ,\Lambda ^1TM)$. Moreover, there exists $c_0\in {\mathbb R}$ such that
\begin{equation}
\label{aux-1}
\left({\mathbb I}+{\mathcal P}{\mathcal V}_{\partial \Omega }\right){\boldsymbol \varphi} =c_0\nu ,
\end{equation}
due to the property that ${\mathcal V}_{\partial \Omega }\nu ={\bf 0}$ on $\partial \Omega $.
{Then} the fields ${\bf u}_0:={\bf V}_{\partial \Omega }{\boldsymbol \varphi} \in H^1(\Omega _{\pm },\Lambda ^1TM)$ and $\pi _0:={\mathcal Q}_{\partial \Omega }{\boldsymbol \varphi} \in L^2(\Omega _{\pm })$ satisfy the Stokes system in $\Omega _{+}:=\Omega $ and $\Omega _{-}:=M\setminus \overline\Omega $, and the following relations on $\partial \Omega $
\begin{align}
\label{Fredholm7-sm}
&\gamma _{+}{\bf u}_{0}=\gamma _{-}{\bf u}_{0}={\mathcal V}_{\partial \Omega }{\boldsymbol \varphi} ,
\\
\label{Fredholm8-sm}
&{{\bf t}_0^-}({\bf u}_{0},\pi _{0})=\left(\frac{1}{2}{\mathbb I}+{\bf K}_{\partial \Omega }^*\right){\boldsymbol \varphi} =-\frac{1}{2}{\mathcal P}{\mathcal V}_{\partial \Omega }{\boldsymbol \varphi} +\frac{c_0}{2}\nu +{\bf K}_{\partial \Omega }^*{\boldsymbol \varphi} ,
\\
&{{\bf t}_0^-}({\bf u}_{0},\pi_{0})=\left(-\frac{1}{2}{\mathbb I}+{\bf K}_{\partial \Omega }^*\right){\boldsymbol \varphi} =\frac{1}{2}{\mathcal P}{\mathcal V}_{\partial \Omega }{\boldsymbol \varphi} -\frac{c_0}{2}\nu +{\bf K}_{\partial \Omega }^*{\boldsymbol \varphi} ,
\end{align}
due to formulas \eqref{eq:dlB13} {and \eqref{aux-1}}. In addition, by the Green formula \eqref{conormal-derivative-generalized-2} in $\Omega _{\pm }$ and the relation
$\langle {\mathcal V}_{\partial \Omega }{\boldsymbol \varphi} ,\nu \rangle _{\partial \Omega }=
\langle {\boldsymbol \varphi} ,{\mathcal V}_{\partial \Omega }\nu \rangle _{\partial \Omega }=0$, we obtain \begin{align}
\label{P0-3-Fredhlm-sm}
2\langle {\rm{Def}}\ {\bf u}_{0},{\rm{Def}}\ {\bf u}_{0}\rangle _{\Omega _{+}}&=
\langle {{\bf t}_0^+}({\bf u}_{0},\pi _{0}),\gamma _{+}{\bf u}_{0}\rangle _{\partial \Omega }
\nonumber\\
&=-\frac{1}{2}\langle {\mathcal P}{\mathcal V}_{\partial \Omega }{\boldsymbol \varphi} ,{\mathcal V}_{\partial \Omega }{\boldsymbol \varphi} \rangle _{\partial \Omega }+\langle {\bf K}_{\partial \Omega }^*{\boldsymbol \varphi} ,{\mathcal V}_{\partial \Omega }{\boldsymbol \varphi} \rangle _{\partial \Omega },\\
\label{P0-4-Fredholm-sm}
2\langle {\rm{Def}}\ {\bf u}_{0},{\rm{Def}}\ {\bf u}_{0}\rangle _{\Omega _{-}}&=
-\langle {{\bf t}_0^-}({\bf u}_{-},\pi _{-}),\gamma _{-}{\bf u}_{-}\rangle _{\partial \Omega }\nonumber\\
&=-\frac{1}{2}\langle {\mathcal P}{\mathcal V}_{\partial \Omega }{\boldsymbol \varphi} ,{\mathcal V}_{\partial \Omega }{\boldsymbol \varphi} \rangle _{\partial \Omega }-\langle {\bf K}_{\partial \Omega }^*{\boldsymbol \varphi} ,{\mathcal V}_{\partial \Omega }{\boldsymbol \varphi} \rangle _{\partial \Omega }.
\end{align}
By \eqref{P0-3-Fredhlm-sm}, \eqref{P0-4-Fredholm-sm} and \eqref{Poisson-transmission3} we deduce the relation
\begin{align}
\label{uniqueness-2NS-F-sm}
2\langle {\rm{Def}}\,{\bf u}_{0},{\rm{Def}}\, {\bf u}_{0}\rangle _{\Omega _{0}}+\langle {\rm{Def}}\, {\bf u}_{0},{\rm{Def}}\, {\bf u}_{0}\rangle _{\Omega _{-}}&=-\langle {\mathcal P}{\mathcal V}_{\partial \Omega }{\boldsymbol \varphi} ,{\mathcal V}_{\partial \Omega }{\boldsymbol \varphi} \rangle _{\partial \Omega }\\
&\leq -C_{\mathcal P}\|{\mathcal V}_{\partial \Omega }{\boldsymbol \varphi} \|_{L^2(\partial \Omega ,\Lambda ^1TM)}^2,\nonumber
\end{align}
where the left-hand side is non-negative, while the right-hand side is non-positive. Therefore, both sides vanish, and, in particular, we deduce that ${\mathcal V}_{\partial \Omega }{\boldsymbol \varphi} ={\bf 0}$ on $\partial \Omega $, i.e., ${\boldsymbol \varphi}\in {\rm{Ker}}\left\{{\mathcal V}_{\partial \Omega }:L^2(\partial \Omega ,\Lambda ^1TM)\to L^2(\partial \Omega ,\Lambda ^1TM)\right\}$. By using, e.g., \cite[Theorem 6.1]{M-T}{, this implies} that ${\boldsymbol \varphi}\in {\mathbb R}\nu $, and thus, $[{\boldsymbol \varphi}]={\mathbb R}\nu $, that is, $[{\boldsymbol \varphi}]=[{\bf 0}]$ in $L^2(\partial \Omega ,\Lambda ^1TM)/{{\mathbb R}\nu }$. Consequently, the Fredholm operator of index zero \eqref{aux} is one-to-one, and hence, an isomorphism. By duality and the property $\left(L_\nu ^2(\partial \Omega ,\Lambda ^1TM)\right)'=L^2(\partial \Omega ,\Lambda ^1TM)/{{\mathbb R}\nu }$, we deduce that the operator \eqref{aux-adjoint} is also invertible.
\end{proof}

\section{Mapping properties of nonlinear operator ${\mathring{\mathcal I}}_{k;\beta ;{\Omega }}$}
\label{A3}
{Recall that $\mathring E$ is the operator of extension of vector fields (or one forms) defined in $\Omega $ by zero on $M\setminus \Omega $. For {$s\in [\frac{1}{2},1)$} we {will prove in the following lemma that} the non-linear operator ${\mathring{\mathcal I}}_{k;\beta ;{\Omega }}:{H^{s+\frac{1}{2}}({\Omega },\Lambda ^1TM)}\to \widetilde{H}^{s-\frac{3}{2}}({\Omega },\Lambda ^1TM)$,
\begin{align}
\label{Newtonian-oper-Brinkman1-s}
{\mathring{\mathcal I}}_{k;\beta ;{\Omega }}({\bf v}):=\mathring E\left(k|{\bf v}|{\bf v}+\beta \nabla _{\bf v}{\bf v}\right),
\end{align}
{is bounded and continuous, which implies that it} can be used in Section~\ref{nonlin} {as an alternative to} the operator $\widetilde{\mathcal I}$ defined by \eqref{nonlin-1} and \eqref{nonlin-2}.
\begin{lem}
\label{particular-operator}
{Let $M$ satisfy {Assumption} $\ref{H}$ and $\dim (M)\geq 2$}. Let $\Omega \subset M$ be a Lipschitz domain. Let {$s\in [\frac{1}{2},1)$}. Then there exist some constants $C_1'\equiv C_1'({\Omega },k,\beta )>0$ and $C_1\equiv C_1({\Omega },k,\beta )>0$ such that
\begin{align}
\label{Ibound}
&\|{\mathring{\mathcal I}}_{k;\beta ;{\Omega }}({\bf u})\|_{\widetilde{H}^{s-\frac{3}{2}}({\Omega },\Lambda ^1TM)}\leq C_1'\|{\bf u}\|_{H^{s+\frac{1}{2}}(\Omega ,\Lambda ^1TM)}^2,
\end{align}
\begin{align}
\label{NS-s3-new}
&\|{\mathring{\mathcal I}}_{k;\beta ;{\Omega }}({\bf v})-{\mathring{\mathcal I}}_{k;\beta ;{\Omega }}({\bf w})\|_{\widetilde{H}^{s-\frac{3}{2}}(\Omega ,\Lambda ^1TM)}\\
&\leq {C_1}\left(\|{\bf v}\|_{H^{s+\frac{1}{2}}(\Omega ,\Lambda ^1TM)}+\|{\bf w}\|_{H^{s+\frac{1}{2}}(\Omega ,\Lambda ^1TM)}\right)\|{\bf v}-{\bf w}\|_{H^{s+\frac{1}{2}}(\Omega ,\Lambda ^1TM)},\nonumber
\end{align}
for all ${\mathbf u, \mathbf v, \mathbf w}\in {H^{s+\frac{1}{2}}({\Omega },\Lambda ^1TM)}$.
\end{lem}}
\begin{proof}
Since $\Omega $ is a Lipschitz domain on the compact Riemannian manifold $M$ with $\dim (M)\!\in \!\{2,3\}$, the Sobolev embedding theorem yields that the inclusions
\begin{align}
\label{Newtonian-D-B-F-new5-new1a-S}
H^{s+\frac{1}{2}}({\Omega },\Lambda ^1TM)\hookrightarrow H^{1}({\Omega },\Lambda ^1TM)\hookrightarrow L^q({\Omega }_{+},\Lambda ^1TM)
\end{align}
are continuous for all $q\in [{1},6]$ (see, e.g., \cite[Chapter 2]{Aubin}). The inclusions
\begin{align}
\label{mod}
{H^{\frac{3}{2}-s}({\Omega },\Lambda ^1TM)\hookrightarrow L^{\frac{2n}{n-3+2s}}({\Omega }_{+},\Lambda ^1TM)\hookrightarrow L^r({\Omega }_{+},\Lambda ^1TM)}
\end{align}
are also continuous for all $r\in [1,3]$. Then a density and duality argument implies that the embedding
\begin{equation}
\label{Newtonian-D-B-F-Robin2}
{\mathring E}L^{r'}({\Omega },\Lambda ^1TM)\hookrightarrow \widetilde{H}^{s-\frac{3}{2}}({\Omega },\Lambda ^1TM)
\end{equation}
is continuous, where $r'\in (1,\infty )$, $\frac{1}{r}+\frac{1}{r'}=1$. Hence, ${\mathring E}{\bf u}\in \widetilde{H}^{s-\frac{3}{2}}({\Omega },\Lambda ^1TM)$ for ${\bf u}\in L^{r'}({\Omega },\Lambda ^1TM)$, and there is a constant ${C}\equiv {C}({\Omega },s)>0$ such that
\begin{align}
\label{Newtonian-D-B-F-Robin2-S}
\|{\mathring E}{\bf u}\|_{\widetilde{H}^{s-\frac{3}{2}}({\Omega },\Lambda ^1TM)}\le {C}\|{\bf u}\|_{L^{r'}({\Omega },\Lambda ^1TM)},\ \forall \ {\bf u}\in L^{r'}({\Omega },\Lambda ^1TM),
\end{align}
By (\ref{Newtonian-D-B-F-new5-new1a-S}) with $q=4$ and by the H\"{o}lder inequality there exists a constant $C'\equiv C'({\Omega },s)>0$ such that
\begin{align}
\label{Newtonian-D-B-F-new5-new1-s}
\|\ |{\bf v}|{\bf w}\ \|_{L^2({\Omega },\Lambda ^1TM)}&\leq \|{\bf v}\|_{L^4({\Omega },\Lambda ^1TM)}\|{\bf w}\|_{L^4({\Omega },\Lambda ^1TM)}\nonumber
\\
&\leq C'\|{\bf v}\|_{H^{s+\frac{1}{2}}({\Omega },\Lambda ^1TM)}\|{\bf w}\|_{H^{s+\frac{1}{2}}({\Omega },\Lambda ^1TM)},
\end{align}
and hence $|{\bf v}|{\bf w}\in L^2({\Omega },\Lambda ^1TM)$ for all ${\bf v},\, {\bf w}\in H^{s+\frac{1}{2}}({\Omega },\Lambda ^1TM)$.
Moreover, by (\ref{Newtonian-D-B-F-Robin2-S}) (with $r=2$) and (\ref{Newtonian-D-B-F-new5-new1-s}), the positively homogeneous operator ${\mathcal B}$ of order $2$, defined by
$${\mathcal B}({\bf v},{\bf w}):={\mathring E}(|{{\bf v}|{\bf w}}),$$ maps $H^{s+\frac{1}{2}}({\Omega },\Lambda ^1TM)\!\times \!H^{s+\frac{1}{2}}({\Omega },\Lambda ^1TM)$ to $\widetilde{H}^{0}({\Omega },\Lambda ^1TM)\!\!\hookrightarrow \!\!\widetilde{H}^{s-\frac{3}{2}}({\Omega },\Lambda ^1TM)$ and there exists a constant $c\equiv c({\Omega },s)>0$ such that
\begin{align}
\label{Newtonian-D-B-F-new5-new1}
\|{\mathcal B}({\bf v},{\bf w})\|_{{\widetilde{H}^{s-\frac{3}{2}}({\Omega },\Lambda ^1TM)}}&\leq c \|{\bf v}\|_{H^{s+\frac{1}{2}}({\Omega },\Lambda ^1TM)}\|{\bf w}\|_{H^{s+\frac{1}{2}}({\Omega },\Lambda ^1TM)}.
\end{align}
Consequently, the following operators are bounded
\begin{align}
&{\mathcal B}: H^{s+\frac{1}{2}}({\Omega },\Lambda ^1TM)\times H^{s+\frac{1}{2}}({\Omega },\Lambda ^1TM)\to L^{2}({\Omega },\Lambda ^1TM),
\\
&{\mathcal B}: H^{s+\frac{1}{2}}({\Omega },\Lambda ^1TM)\times H^{s+\frac{1}{2}}({\Omega },\Lambda ^1TM)\to \widetilde{H}^{s-\frac{3}{2}}({\Omega },\Lambda ^1TM).
\end{align}
On the other hand, recall that if ${\bf v}=v^j\partial _j$ and ${\bf w}=w^\ell \partial _\ell $, then $\nabla _{\bf v}{\bf w}$ has the following local representation
$\left(\nabla _{{\bf v}}{\bf w}\right)^\ell =v^j\partial _jw^\ell+\Gamma _{rj}^\ell w^rv^j.$
By exploiting the H\"{o}lder inequality, the embedding $\partial _jw^\ell\in H^{s-\frac{1}{2}}({\Omega })\hookrightarrow L^2 ({\Omega })$ for $s\ge 1/2$, and embedding \eqref{Newtonian-D-B-F-new5-new1a-S} with $q=6$, we deduce that there exists a constant $c_{(0)}\equiv c_{(0)}({\Omega },s)>0$ such that

\begin{align}
\|v^j\partial _jw^\ell\|_{L^{\frac{3}{2}}({\Omega })}&
\leq \|v^j\|_{L^{6}({\Omega })}\|\partial _jw^\ell\|_{L^{2}({\Omega })}
\leq c_{(0)}\|v^j\|_{H^{s+\frac{1}{2}}({\Omega })}
\|w^\ell\|_{H^{s+\frac{1}{2}}({\Omega })}, \nonumber
\end{align}
for all ${\bf v},\, {\bf w}\in H^{s+\frac{1}{2}}(\Omega ,\Lambda ^1TM)$. Therefore, the following inequality holds
\begin{align}
\label{Newtonian-D-B-F-Robin1-s}
\|\nabla _{\bf v}{\bf w}\|_{L^{\frac{3}{2}}({\Omega },\Lambda ^1TM)}\leq C_0'\|{\bf v}\|_{H^{s+\frac{1}{2}}({\Omega },\Lambda ^1TM)}\|{\bf w}\|_{H^{s+\frac{1}{2}}({\Omega },\Lambda ^1TM)}
\end{align}
for all ${\bf v},\, {\bf w}\in H^{s+\frac{1}{2}}({\Omega },\Lambda ^1TM)$, with some constant $C_0'\equiv C_0'({\Omega },s)>0$.
{Moreover,} by the continuity of the embedding (\ref{Newtonian-D-B-F-Robin2}), with $r'=\frac{3}{2}$, and by (\ref{Newtonian-D-B-F-Robin1-s}), there exists a constant $C_0\equiv C_0({\Omega },s)>0$ such that the bilinear operator ${\mathring N}({\bf v},{\bf w}):={\mathring E}(\nabla _{\bf v}{\bf w})$ satisfies the inequality
\begin{align}
\label{Newtonian-D-B-F-Robin1}
\|{\mathring N}({\bf v},{\bf w})\|_{\widetilde{H}^{s-\frac{3}{2}}({\Omega },\Lambda ^1TM)}&\leq C_0\|{\bf v}\|_{H^{s+\frac{1}{2}}(\Omega ,\Lambda ^1TM)}\|{\bf w}\|_{H^{s+\frac{1}{2}}(\Omega ,\Lambda ^1TM)},
\end{align}
for all ${\bf v},{\bf w}\in H^{s+\frac{1}{2}}(\Omega ,\Lambda ^1TM)$. Hence, the following {operator} is bounded
\begin{align}
\label{bound-4}
{\mathring N}: H^{s+\frac{1}{2}}(\Omega ,\Lambda ^1TM)\times H^{s+\frac{1}{2}}(\Omega ,\Lambda ^1TM)\to \widetilde{H}^{s-\frac{3}{2}}(\Omega ,\Lambda ^1TM).
\end{align}
Thus, ${\mathring{\mathcal I}}_{k;\beta ;{\Omega }}$ is also bounded in the sense of \eqref{Ibound}
{with} the constants
\begin{align}
\label{const}
C_1'={|k|}C'+{|\beta |}C'_0 \text{ and } C_1={|k|}c+{|\beta |}C_0.
\end{align}
{In addition, due to \eqref{Newtonian-oper-Brinkman1-s}, ${\mathring{\mathcal I}}_{k;\beta ;{\Omega }}$ is positively homogeneous of order 2.}

We now show inequality \eqref{NS-s3-new}. Indeed, by \eqref{Newtonian-D-B-F-new5-new1}, \eqref{Newtonian-D-B-F-Robin1}, and the expression of the constant $C_1$ given in \eqref{const}, we obtain that
\begin{align}
\label{continuity}
&\|{\mathring{\mathcal I}}_{k;\beta ;{\Omega }}({\bf v})-{\mathring{\mathcal I}}_{k;\beta ;{\Omega }}({\bf w})\|_{\widetilde{H}^{s-\frac{3}{2}}(\Omega ,\Lambda ^1TM)}\\
&\leq {|k|}\|\ |{\bf v}|{\bf v}-|{\bf w}|{\bf w}\ \|_{\widetilde{H}^{s-\frac{3}{2}}(\Omega ,\Lambda ^1TM)}+{|\beta |}\|\nabla _{{\bf v}}{\bf v}-\nabla _{{\bf w}}{\bf w}\|_{\widetilde{H}^{s-\frac{3}{2}}(\Omega ,\Lambda ^1TM)}\nonumber
\\
&\leq {|k|}\|(|{\bf v}|-|{\bf w}|){\bf v}+|{\bf w}|({\bf v}-{\bf w})\|_{\widetilde{H}^{s-\frac{3}{2}}(\Omega ,\Lambda ^1TM)}\nonumber\\
&\hspace{4em}+{|\beta |}\|\nabla _{{\bf v}-{\bf w}}{\bf v}+\nabla _{{\bf w}}({\bf v}-{\bf w})\|_{\widetilde{H}^{s-\frac{3}{2}}(\Omega ,\Lambda ^1TM)}\nonumber\\
&\leq ({|k|}c+{|\beta |}C_0)\|{\bf v}-{\bf w}\|_{H^{s+\frac{1}{2}}(\Omega ,\Lambda ^1TM)}\left(\|{\bf v}\|_{H^{s+\frac{1}{2}}(\Omega ,\Lambda ^1TM)}+\|{\bf w}\|_{H^{s+\frac{1}{2}}(\Omega )}\right)\nonumber\\
&=C_1\|{\bf v}-{\bf w}\|_{H^{s+\frac{1}{2}}(\Omega ,\Lambda ^1TM)}\left(\|{\bf v}\|_{H^{s+\frac{1}{2}}(\Omega ,\Lambda ^1TM)}+\|{\bf w}\|_{H^{s+\frac{1}{2}}(\Omega ,\Lambda ^1TM)}\right)\nonumber,
\end{align}
for all ${\bf v},{\bf w}\in H^{s+\frac{1}{2}}(\Omega ,\Lambda ^1TM),$ i.e., inequality (\ref{NS-s3-new}) {holds.}
\end{proof}


\end{document}